\documentclass[11pt]{article}

\usepackage{amsmath,amsthm,amssymb,amsfonts}
\usepackage{graphics,graphicx,color}

\graphicspath{{/EPSF/}{Figures/}}

\makeatletter

\@addtoreset{equation}{section}
\makeatother

\newtheorem{theorem}{Theorem}[section]
\newtheorem{lemma}[theorem]{Lemma}
\newtheorem{proposition}[theorem]{Proposition}

\newtheorem{remark}[theorem]{Remark}

\def\ep{\varepsilon}

\def\C{\mathbb C}
\def\R{\mathbb R}
\def\S{\mathbb S}
\def\N{\mathbb N}
\def\pa{\partial}
\def\b{\backslash}

\def\diam{{\rm diam}(X)}

\def \Rm {\mathbb R}

\newcommand{\eps}{\varepsilon}

\newcommand{\mC}{\mathcal C}
\newcommand{\mG}{\mathcal G}

\newcommand{\cout}[1]{}

\def\ep{\varepsilon}

\def\C{\mathbb C}
\def\R{\mathbb R}
\def\S{\mathbb S}
\def\N{\mathbb N}
\def\mU{\mathcal U}
\def\pa{\partial}
\def\b{\backslash}

\def\r{\rangle}
\def\l{\langle}

\newcommand{\vertiii}[1]{{\left\vert\kern-0.25ex\left\vert\kern-0.25ex\left\vert #1 
    \right\vert\kern-0.25ex\right\vert\kern-0.25ex\right\vert}}

\title{Boundary Control for Transport Equations}

\author{Guillaume Bal\thanks{Departments of Statistics and Mathematics, University of Chicago, Chicago, IL 60637, USA; {\tt guillaumebal@uchicago.edu}}   \and Alexandre Jollivet\thanks{Laboratoire de Math\'ematiques Paul Painlev\'e, CNRS UMR 8524/Universit\'e Lille 1 Sciences et Technologies, 59655 Villeneuve d'Ascq Cedex, France; {\tt alexandre.jollivet@math.univ-lille1.fr}}}

\begin{document}
 
\maketitle


\begin{abstract}
  This paper considers two types of boundary control problems for linear transport equations. The first one shows that transport solutions on a subdomain of a domain $X$ can be controlled exactly from incoming boundary conditions for $X$ under appropriate convexity assumptions. This is in contrast with the only approximate control one typically obtains for elliptic equations by an application of a unique continuation property, a property which we prove does not hold for transport equations. We also consider the control of an outgoing solution from incoming conditions, a transport notion similar to the Dirichlet-to-Neumann map for elliptic equations. We show that for well-chosen coefficients in the transport equation, this control may not be possible. In such situations and by (Fredholm) duality, we obtain the existence of non-trivial incoming conditions that are compatible with vanishing outgoing conditions.
\end{abstract}

\renewcommand{\thefootnote}{\fnsymbol{footnote}}
\renewcommand{\thefootnote}{\arabic{footnote}}

\renewcommand{\arraystretch}{1.1}

\noindent{\bf Keywords:} Transport theory; boundary control; albedo operator; diffusion approximation; unique continuation.




%
\section{Introduction}
\label{sec:intro}

This paper concerns the control of steady-state (linear Boltzmann) transport equations from boundary data.  The most general transport equation considered here is of the form:
\begin{eqnarray}
v\cdot \nabla_x u(x,v)+\sigma(x,v) u(x,v)=\int_V k(x,v',v)u(x,v')dv',\ (x,v)\in X\times V,\ \label{boltz}
\end{eqnarray}
where $u(x,v)$ is the transport solution, posed on a spatial domain $x\in X$,  an open bounded subset of $\Rm^d$ for $d\geq1$ of class $C^1$, and a set of velocities $v\in V$, which is either a bounded open subset in $\Rm^d$ excluding a vicinity of $v=0$, or a co-dimension $1$ closed surface also excluding $v=0$ such as for instance the unit sphere for concreteness. Here, $\sigma(x,v)$ is the total absorption coefficient and $k(x,v',v)$ the scattering coefficient. 

Let us define the sets of incoming and outgoing directions $\Gamma_\pm=\Gamma_\pm(X):=\{(x,v)\in \pa X\times V\ |\ \pm v\cdot \nu(x)>0\}$ while $\Gamma:=\Gamma_+\cup\Gamma_-$. Under appropriate conditions on $(\sigma,k)$, the above equation is well-posed when augmented with prescribed incoming conditions $u=g$ on $\Gamma_-$. Within this framework, we consider two types of boundary controls.

The first one concerns the control of a prescribed transport solution $u_0(x,v)$ on a subdomain $X_0\times V$, where $X_0$ is an open subset such that $\bar X_0\subset X$: can one find $g$ on $\Gamma_-(X)$ such that $u_{|X_0}=u_0$ for $u$ solution of the above transport equation on $X$? A similar question arises in the setting of second-order scalar elliptic operators. In such a case, the control is approximate, as an application of the Runge approximation result, itself a rather direct consequence of the (weak) Unique Continuation Property (UCP) enjoyed by such operators. One of the main results of this paper is to show, under appropriate convexity assumptions on $X_0$ and $X$, that such a control is in fact {\em exact} for transport equations. Moreover, the control is not unique, so that the difference of two controls may result in a solution $u$ that does not vanish on $X\backslash X_0$ whereas $u_{|X_0}=0$. In other words, UCP (in that sense) does not hold for the transport equation independently of the scattering coefficients $(\sigma,k)$. 

The two models, transport and elliptic equations, both describe transport in highly scattering media. For instance, if $k=k(x)$ is replaced by $k(x)/\eps$ and $\sigma(x)=k(x)/\eps+\eps\sigma_a(x)$, then the transport solution is well approximated by the solution of an elliptic equation \cite{dlen6,LK} as the mean free path $\eps\to0$. For an infinite mean free path, with $k\equiv0$, transport solutions can be supported arbitrarily close to any line segment in $X$ so that UCP clearly does not hold. The limit $\eps\to0$ is, however, singular in the sense that boundary control is exact at $\eps>0$ and only approximate in the limit $\eps=0$, whereas what makes the approximation possible at $\eps=0$, namely UCP, does not hold for any $\eps>0$.

The second type of control we consider aims to find an incoming condition $u=g$ on $\Gamma_-$ such that $u_{|\Gamma_+}=f$ is prescribed. For second-order elliptic equations, this corresponds to finding Dirichlet conditions for prescribed Neumann conditions, or vice-versa, which is a well posed problem as the Dirichlet-to-Neumann map is a {\em isomorphism} in appropriate topologies. For the transport equation, the problem is richer. We show that the albedo operator, a well defined operator which maps $u$ on $\Gamma_-$ to $u_{|\Gamma_+}$, is surjective when scattering is sufficiently small. However, we identify a number of cases where the albedo operator, while still Fredholm (of index $0$), has a non-trivial kernel. In such settings, the outgoing condition $f$ needs to satisfy appropriate orthogonality properties to be controlled as the trace $u_{|\Gamma_+}=f$ of a transport solution $u$ in $X\times V$. In such cases, we also find by duality the existence of non trivial incoming conditions $u=g$ on $\Gamma_-$ such that $u_{|\Gamma_+}=0$, i.e., the effect of $g$ is invisible on the outgoing solution. 

\medskip
The outline of the paper is as follows. To understand the first control problem, we need to construct a transport solution on $X\backslash X_0$ with prescribed boundary conditions on $\Gamma_-(X_0)\cup\Gamma_+(X_0)$, which is not a subset of $\Gamma_-(X\backslash X_0)$. Section \ref{sec:forward} is devoted to the analysis of \eqref{boltz} augmented with boundary conditions on general sets $\mC$ defined such that for each line segment passing through $X$, we prescribe a boundary condition on one of its two points of intersections with $\partial X$. Under appropriate conditions on the coefficients $(\sigma,k)$, we obtain a transport solution and the definition of an albedo operator mapping conditions on $\mC$ to conditions on $\Gamma\backslash\mC$. This general transport theory will be the starting point of the analysis of  the boundary control of solutions on a subdomain presented in section \ref{sec:convex}. The Fredholm theory of the albedo operator from $\mC$ to $\Gamma\backslash\mC$ is presented in section \ref{sec:albedo} and used in section \ref{sec:outgoing} to analyze the specific control of outgoing boundary conditions on $\Gamma\backslash\mC=\Gamma_+$ from incoming conditions $\mC=\Gamma_-$.

\medskip

The boundary control of solutions on subdomains find several applications in the field of hybrid inverse problems; see \cite{alberti2018lectures,BMU-SIAP-15,BU-CPAM-13} for results in the elliptic framework and \cite{BCS-SIMA-16} for results in the setting of transport equations.

\section{Forward transport theory}
\label{sec:forward}

We assume that $X$ is an open bounded subset of $\Rm^d$ for $d\geq1$ of class $C^1$. The set of velocities $V$ is either a bounded open subset in $\Rm^d$ excluding a vicinity of $v=0$, or a co-dimension $1$ closed surface also excluding $v=0$. 

Let $\Gamma_\pm:=\{(x,v)\in \pa X\times V\ |\ \pm v\cdot \nu(x)>0\}$ and $\Gamma:=\Gamma_+\cup\Gamma_-$. We endow $\Gamma$ with the measure $d\xi(x,v)=|\nu(x)\cdot v|d\mu(x) dv$, where $d\mu$ denotes the canonical measure on $\pa X$ and $dv$ that on $V$.  Both $\Gamma_+$ and $\Gamma_-$ parametrize the set of lines in $\Rm^n$ passing through $X$. Typical boundary conditions consist in prescribing the trace of the solution on $\Gamma_-$. We consider the following general boundary conditions. 

Let $C_+$ be an arbitrary measurable subset of $\Gamma_+$ and $C_-$ the measurable subset of $\Gamma_-$ defined by
$$
C_-:=\Gamma_-\b\{(x,v)\in \Gamma_-\ |\ (x+\tau_+(x,v)v,v)\in C_+\}.
$$
We then define $\mC=C_-\cup C_+$. Note that $\mC$ also provides a parametrization of the set of lines passing through $X$. We now consider a transport theory with boundary conditions prescribed on $\mC$.

We first need to define a functional setting for the transport solution and its boundary traces. For $1\leq p<\infty$, we define two natural topologies to describe transport solutions
$$
W^p(X\times V)=\{\phi\in {\cal D}'(X\times V)\ |\ \phi\in \tau^{1\over p}L^p(X\times V),\ v\cdot\nabla_x\phi\in \tau^{{1\over p}-1}L^p(X\times V)\},
$$
$$
\tilde W^p(X\times V)=\{\phi\in L^p(X\times V)\ |\ \tau (v\cdot\nabla_x\phi)\in L^p(X\times V)\},
$$
endowed with the norms
\begin{eqnarray}
&&\|u\|_{W^p}:=\|\tau^{1-{1\over p}}(v\cdot \nabla_x)u\|_{L^p}+\|\tau^{-{1\over p}}u\|_{L^p},\\
&&\|u\|_{\tilde W^p}:=\|\tau( v\cdot \nabla_x ) u\|_{L^p}+\|u\|_{L^p}.
\end{eqnarray}
Above, we have introduced the usual travel times 
\begin{eqnarray*}
&&\tau_\pm(x,v)=\inf\{s\in (0,+\infty)\ |\ x\pm sv\not\in X\}, \textrm{ for }(x,v)\in (X\times V)\cup \Gamma_\mp,\\
&&\tau_\pm(x,v)=0\textrm{ for }(x,v)\in \Gamma_\pm.
\end{eqnarray*}
We also define the real valued function $\tau$ on $\bar X\times V$ by 
$$
\tau=\tau_-+\tau_+.
$$
Traces are then well defined according to the
\begin{lemma}
Let $\mG$ be a measurable subset  of $\Gamma$. Then, for $1\le p<\infty$ (it also holds for $p=\infty$),
\begin{eqnarray}
\|u_{|\mG}\|_{L^p(\mG,d\xi)}\le \|u\|_{W_p},
\ \|u_{|\mG}\|_{L^p(\mG,\tau d\xi)}\le \|u\|_{\tilde W_p}.
\end{eqnarray}
\end{lemma}
The Lemma will be used for $\mG=\mC$ or $\mG=\Gamma\b \mC$.
\begin{proof}
Set $\mG_\pm:=\mG\cap \Gamma_\pm$. Let $u\in C^1(\bar X\times V)$. Then for $(x,v)\in \Gamma_\pm$
$$
u(x,v)=\mp{1\over \tau(x,v)}\int_0^{\tau(x,v)}\big((\tau(x,v)-t)(v\cdot\nabla_x) u(x \mp tv,v)+u(x \mp tv,v)\big)dt.
$$
Therefore
\begin{eqnarray*}
&&\Big(\int_{\Gamma_\pm}\chi_{\mG_\pm}(x,v)|u(x,v)|^p d\xi(x,v)\Big)^{1\over p}\le \Big(\int_{\Gamma_\pm}|u(x,v)|^p d\xi(x,v)\Big)^{1\over p}\\
&\le &\Big(\int_{\Gamma_\pm}{\tau(x,v)^{-p}}\\
&&\times
\Big(\int_0^{\tau(x,v)}\big((\tau(x,v)-t)|(v\cdot\nabla_x) u(x \mp tv,v)|+|u|(x \mp tv,v)\big)dt\Big)^p d\xi(x,v)\Big)^{1\over p}\\
&\le &\Big(\int_{\Gamma_\pm}\Big(\int_0^{\tau(x,v)}|(v\cdot\nabla_x) u(x \mp tv,v)|dt\Big)^p d\xi(x,v)\Big)^{1\over p}\\
&&+\Big(\int_{\Gamma_\pm}{1\over \tau(x,v)^p}\Big(\int_0^{\tau(x,v)}|u|(x \mp tv,v)dt\Big)^p d\xi(x,v)\Big)^{1\over p}\\
&\le&\Big(\int_{\Gamma_\pm}\tau(x,v)^{p-1}\int_0^{\tau(x,v)}|(v\cdot\nabla_x) u(x \mp tv,v)|^p dt d\xi(x,v)\Big)^{1\over p}\\
&&+\Big(\int_{\Gamma_\pm}\tau^{-1}(x,v)\int_0^{\tau(x,v)}|u|^p(x \mp tv,v)dt d\xi(x,v)\Big)^{1\over p}=\|u\|_{W^p}.
\end{eqnarray*}
Similarly replacing $u$ by $\tau^{1\over p}u$ above we obtain
\begin{equation*}
\Big(\int_{\Gamma_\pm}\chi_{\mG_\pm}(x,v)|u(x,v)|^p \tau(x,v)d\xi(x,v)\Big)^{1\over p}\le \|u\|_{\tilde W^p}.
\end{equation*}

\end{proof}

Adjoint to the notion of trace is that of a lifting operator $J_\mC$ defined by 
$$
J_{\mC}g(x,v)=g(x\pm \tau_\pm(x,v)v,v),\ \textrm{ for }(x,v)\in X\times V\textrm{ s.t. }(x\pm \tau_\pm(x,v)v,v)\in\mC_\pm, 
$$ 
for $g\in L^\infty(\mC,\R)$. Note that $(v\cdot \nabla_x) J_{\mC}g=0$. We have the
\begin{lemma}\label{lem:lift}
Let $1\le p<\infty$ (it actually holds for $p=\infty$). 
The operator $J_{\mC}$ extends as a bounded operator from $L^p(\mC, d\xi)$ to $W^p$, and from $L^p(\mC,\tau d\xi)$ to $\tilde W^p$.
\end{lemma}
\begin{proof}This is the calculation
\begin{eqnarray*}
\|J_{\mC}g\|_{W^p}&=&\|\tau^{-{1\over p}}J_{\mC}g\|_{L^p}=\Big(\int_{\mC_-}\int_0^{\tau_+(x,v)}\tau_+(x,v)^{-1}|J_{\mC}g(x+tv,v)|^p dt d\xi(x,v)\\
&&+\int_{\mC_+}\int_0^{\tau_-(x,v)}\tau_-(x,v)^{-1}|J_{\mC}g(x-tv,v)|^p dt d\xi(x,v)\Big)^{1\over p}\\
&=&\Big(\int_{\mC_-}|g(x,v)|^p d\xi(x,v)+\int_{\mC_+}|g(x,v)|^p d\xi(x,v)\Big)^{1\over p}=\|g\|_{L^p(\mC,d\xi)}.
\end{eqnarray*}
Similarly,
$
\|J_{\mC}g\|_{\tilde W^p}=\|g\|_{L^p(\mC,\tau d\xi)}.
$
\end{proof}

We now consider the transport equation \eqref{boltz} with general boundary condition
\begin{equation}
u=g,\ (x,v)\in \mC. 
\end{equation}
We assume that
$\tau \sigma\in L^\infty$,
and define the scattering coefficients
$$
\sigma_s(x,v)=\int_V k(x,v,v')dv',\ \sigma_s'(x,v)=\int_V k(x,v',v)dv',
$$
as well as the following integrating factor
\begin{equation}\label{eq:E}
E_\pm(x,v,t):=e^{\pm\int_0^t\sigma(x\pm sv,v)ds}.
\end{equation}
It is useful to introduce the operators $L_\mC$, $T^{-1}_\mC$ and $K$ as
\begin{eqnarray*}
L_{\mC} g(x,v)&=&J_{\mC}g(x,v)
\left\lbrace 
\begin{array}{l}
E_-(x,v,\tau_-(x,v)),\ (x-\tau_-(x,v)v,v)\in \mC_-,\\
E_+(x,v,\tau_+(x,v)),\ (x+\tau_+(x,v)v,v)\in \mC_+,
\end{array}
\right.\\
Ku(x,v)&=&\int_V k(x,v',v)u(x,v')dv',\ (x,v)\in X\times V,\\
T_{\mC}^{-1}f(x,v)&=&
\left\lbrace 
\begin{array}{l}
\int_0^{\tau_-(x,v)}E_-(x,v,t)f(x-tv,v)dt,\ (x-\tau_-(x,v)v,v)\in \mC_-,\\
-\int_0^{\tau_+(x,v)}E_+(x,v,t)f(x+tv,v)dt,\ (x+\tau_+(x,v)v,v)\in \mC_+.
\end{array}
\right.
\end{eqnarray*}
We note that 
\begin{eqnarray*}
&&v\cdot\nabla_x L_{\mC} g+\sigma L_{\mC} g=0,\quad L_{\mC} g_{|\mC}=g,\\
&&v\cdot\nabla_x T^{-1}_\mC f+\sigma T^{-1}_\mC f=f,\quad
T^{-1}_\mC f_{|\mC}=0.
\end{eqnarray*}
\begin{lemma}
\label{lem_bound1}
Let $1\le p\le \infty$. We have the following bounds
\begin{eqnarray}
\|L_{\mC} g\|_{W^p}&\le&  e^{\|\tau\sigma\|_\infty}(1+\|\tau\sigma\|_\infty)\|g\|_{L^p(\mC,d\xi)},\label{i1a}\\
\|L_{\mC} g\|_{\tilde W^p}&\le& e^{\|\tau\sigma\|_\infty}(1+\|\tau\sigma\|_\infty)\|g\|_{L^p(\mC,\tau d\xi)},\label{i1b}\\
\|\tau^{-\ep}T^{-1}_\mC f\|_{L^p}&\le& e^{\|\tau\sigma\|_\infty}\|\tau^{1-\ep}f\|_{L^p}\ \textrm{ for }\ep\in \R,\label{i2}
\end{eqnarray}
\begin{eqnarray}
\|T^{-1}_\mC f\|_{W^p}&\le& (2+\|\tau\sigma\|_\infty)e^{\|\tau\sigma\|_\infty}\|\tau^{1-{1\over p}}f\|_{L^p},\label{i3}\\
\|T^{-1}_\mC f\|_{\tilde W^p}&\le& (2+\|\tau\sigma\|_\infty)e^{\|\tau\sigma\|_\infty}\|\tau f\|_{L^p}.\nonumber
\end{eqnarray}
\end{lemma}
We provide a proof of the Lemma in Appendix \ref{sec:appbound}. 
The above results allow us to solve the transport equation in the absence of scattering, which is then incorporated as a perturbation. We then need the following bounds
\begin{lemma}
\label{lem_K}
Let $1\le p<\infty$.
Assume either
\begin{equation}
\kappa_p(x):= \int_V\big(\int_V|k|^{p\over p-1}(.,v',v)dv'\big)^{p-1}dv\in L^\infty(X)\textrm{ when } p>1,\label{i10a}
\end{equation}
\begin{equation}
\mbox{ or } \qquad \|\sigma_s\|_\infty+\|\sigma_s'\|_\infty<\infty.\label{i10b}
\end{equation}
Then the operator $K$ is bounded in $L^p(X\times V)$ and 
\begin{equation}
\|K\|_{\mathcal{L}(L^p(X\times V))}\le 
\left\lbrace
\begin{array}{l}
\big\|\kappa_p(x)\big\|_\infty^{1\over p}\textrm{ when }p>1\textrm{ and }\eqref{i10a} \textrm{ holds},\\
\|\sigma_s\|_\infty^{1\over p}\|\sigma_s'\|_\infty^{p-1\over p} \textrm{ when }\eqref{i10b} \textrm{ holds}.
\end{array}
\right.
\label{i10c}
\end{equation}

Assume either
\begin{equation}
\tilde\kappa_p(x) := \int_V \tau^{p-1} \big(\int_V |k(x,v'',v)|^{p\over p-1}\tau(x,v'')^{1\over p-1} dv''\big)^{p-1}dv \in {L^\infty(X)}\textrm{ when } p>1,\label{i11a}
\end{equation}
\begin{equation}
\mbox{ or } \qquad \|\tau\sigma_s\|_\infty+\|\tau\sigma_s'\|_\infty<\infty.\label{i11b}
\end{equation}
Then the operator $K$ is bounded from $\tau^{1\over p}L^p$ to $\tau^{-{p-1\over p}}L^p$, and
\begin{eqnarray}
\|K\|_{\mathcal{L}(\tau^{1\over p}L^p(X\times V)), \tau^{-{p-1\over p}}L^p(X\times V))}
&\le&
\left\lbrace
\begin{array}{l}
\|\tilde\kappa_p\|_{L^\infty(X)}^{1\over p}\textrm{ when }p>1\textrm{ and }\eqref{i11a} \textrm{ holds},\\
\|\tau\sigma_s\|_\infty^{1\over p}\|\tau\sigma_s'\|_\infty^{p-1\over p} \textrm{ when }\eqref{i11b} \textrm{ holds}.
\end{array}
\right.
\end{eqnarray}
\end{lemma}
We provide a proof of Lemma \ref{lem_K} in Appendix \ref{sec:appbound}.

\begin{remark}
Assumptions \eqref{i10a} and \eqref{i10b} (resp. \eqref{i11a} and \eqref{i11b}) are not 
equivalent. Indeed we provide the following two examples.

(i) Let $k(\theta,\theta')={1\over \sqrt{|\theta-\theta'|}}$, $(\theta,\theta')\in (0,2\pi)^2$. Then $k\in L^\infty((0,2\pi)_\theta,L^1((0,2\pi)_{\theta'}))$ but  $k\not\in L^2((0,2\pi)^2)$. We can construct a scattering coefficient $k$ that satisfies \eqref{i10b} but not \eqref{i10a} for $p=2$.

(ii) Let $g\in L^2(0,2\pi)\b L^\infty(0,2\pi)$, $g>0$, and set $k(\theta,\theta')=g(\theta)g(\theta')$ for a.e. $(\theta,\theta')\in (0,2\pi)^2$. Then, $k\in L^2((0,2\pi)^2)$ but  $k\not\in L^\infty((0,2\pi)_\theta,L^1((0,2\pi)_{\theta'}))$. We can again construct a scattering coefficient $k$ that satisfies \eqref{i10a} but not \eqref{i10b} for $p=2$.

Note that estimates on the norm of the operator $K$ under assumptions \eqref{i10b} or \eqref{i11b} may also be obtained by interpolation \cite{ES-ApplAnal-93}.
\end{remark}

Condition \eqref{i10b} can be relaxed to the condition $\sigma_s\in L^\infty(X\times V)$ when $p=1$ and to the condition $\sigma_s'\in L^\infty(X\times V)$ when $p=\infty$ with the obvious convention $\|f\|_\infty^0=1$ for any measurable function $f$. Similarly condition \eqref{i11b} can be relaxed to the condition $\tau\sigma_s\in L^\infty(X\times V)$ when $p=1$ and to the condition $\tau\sigma_s'\in L^\infty(X\times V)$ when $p=\infty$.

\medskip

We recast the stationary linear Boltzmann equation as the following integral equation
\begin{equation}
(I-T_\mC^{-1}K)u=L_\mC g.
\end{equation}
We may now collect the above results to obtain the following existence and uniqueness result. 
\begin{theorem}
\label{thm_solv}
Let $1\le p\le \infty$. Under conditions \eqref{i11a} or \eqref{i11b}, further assume that
\begin{equation}
I-T_\mC^{-1}K\textrm{ is invertible in }\mathcal{L}(\tau^{1\over p}L^p(X\times V)).\label{i20a}
\end{equation}
Then, for any $f\in L^p(\mathcal{C},d\xi)$, there is a unique solution $u$ of \eqref{boltz} in $W^p(X\times V)$ and the albedo operator $A=((I-T_\mC^{-1}K)^{-1}L_\mC)_{|\Gamma\b\mC}$ is well defined and bounded from $L^p(\mC,d\xi)$ to $L^p(\Gamma\b \mC,d\xi)$.

Similarly, under conditions \eqref{i10a} or \eqref{i10b}, further assume that 
\begin{equation}
I-T_\mC^{-1}K\textrm{ is invertible in }\mathcal{L}(L^p(X\times V)).\label{i20b}
\end{equation}
Then, for any $f\in L^p(\mathcal{C},\tau d\xi)$ there is a unique solution $u$ of \eqref{boltz} in $\tilde W^p(X\times V)$, and the albedo operator $A$ is  bounded from $L^p(\mC,\tau d\xi)$ to $L^p(\Gamma\b \mC,\tau d\xi)$.
\end{theorem}

\begin{proof}
Assume \eqref{i20a}.
Let $f\in L^p(\mC, d\xi)$. Then set $u=(I-T_\mC^{-1}K)^{-1}L_\mC f\in \tau^{1\over p}L^p(X\times V)$. We have 
$$u=T_\mC^{-1}Ku+L_\mC f.$$
Hence we have 
$v\cdot \nabla_x u=\sigma u+Ku$
and we obtain that $u\in W^p(X\times V)$ and 
$$
\|\tau^{1-{1\over p}}v\cdot \nabla_x u\|_{L^p}\le \big(\|\tau\sigma\|_\infty+\|K\|_{\mathcal{L}(\tau^{1\over p}L^p(X\times V)), \tau^{-1+{1\over p}}L^p(X\times V))}\big)\|\tau^{-{1\over p}}u\|_{L^p}.
$$
We perform the analysis for $f\in L^p(\mC,\tau d\xi)$ similarly.
\end{proof}

We verify that \eqref{i20a} and \eqref{i20b} hold as soon as the scattering coefficient $k$ is sufficiently small. 
Let $1\le p<\infty$ and assume that $K$ defines a bounded operator in $L^p(X\times V)$. Then $I-T_\mC^{-1}K$ is invertible in $\tau^{1\over p}L^p(X\times V)$ when $I-T_\mC^{-1}K$ is invertible in $L^p(X\times V)$. This statement follows from the fact that $\tau^{1\over p}L^p(X\times V)$ is a {\em subspace}
of $L^p(X\times V)$ and that $T^{-1}_\mC$ maps $L^p(X\times V)$ to $\tau L^p(X\times V)$.
Then we have the quantitative result:
\begin{lemma}
Assume that
\begin{equation}
\big\|\kappa_p(x)\|_{L^\infty(X)}^{1\over p}< {e^{-\|\tau\sigma\|_\infty}\over \|\tau\|_\infty}\label{i70}
\end{equation}
when $p>1$ or 
\begin{equation}
\|\sigma_s\|_\infty< {e^{-\|\tau\sigma\|_\infty}\over \|\tau\|_\infty}\label{i70b}
\end{equation}
when $p=1$.
Then condition \eqref{i20b} (and \eqref{i20a}) holds. 
\end{lemma}
\begin{proof}
By \eqref{i2}, \eqref{i10c} and \eqref{i70}, it follows that
$
\|T_\mC^{-1}K\|_{{\mathcal L}(L^p(X\times V))}<1,
$ 
which proves \eqref{i20b}. When $p=1$ we recall that condition \eqref{i10b} can be relaxed by the condition $\sigma_s\in L^\infty(X\times V)$  and then \eqref{i2}, \eqref{i10c} and \eqref{i70b} give 
$\|T_\mC^{-1}K\|_{{\mathcal L}(L^p(X\times V))}<1$.
\end{proof}
This confirms the intuition that scattering may be seen as a perturbation to the (scattering-)free transport equation and provides existence and uniqueness results for the transport equation under a smallness condition. More general results will be obtained below under the assumption that $T_\mC^{-1}K$ is compact as a standard application of the theory of Fredholm operators (with vanishing index). We refer to, e.g., \cite{dlen6}, \cite{ES-ApplAnal-93} for other sufficient conditions in the setting of boundary conditions imposed on $\Gamma_-$.


Before considering (standard) compactness results, we show that backward equations may be solved under conditions on the forward operators as follows. The dual operators $K^*$ and $(T^{-1}_\mC)^*$ for the usual inner products are given by 
\begin{equation}
K^*f(x,v)=\int_V k(x,v,v')f(x,v')dv',\ f\in L^{p\over p-1}(X\times V)
\end{equation}
and
\begin{eqnarray*}
{T^{-1}_{\mC}}^*\phi(x,v)=
\left\lbrace
\begin{array}{l}
\int_0^{\tau_+(x,v)}E_-(x+tv,v,t)\phi(x+tv,v)dt,\ (x-\tau_-(x,v)v,v)\in \mC_-,\\
-\int_0^{\tau_-(x,v)}E_+(x-tv,v,t)\phi(x-tv,v)dt,\ (x+\tau_+(x,v)v,v)\in \mC_+.
\end{array}
\right.
\end{eqnarray*}
We find
\begin{equation*}
v\cdot\nabla_x (T^{-1}_\mC)^*\phi(x,v)=\sigma(x,v)(T^{-1}_\mC)^*\phi(x,v)-\phi(x,v),
\end{equation*}
and the identity
$
\big((T^{-1}_\mC)^*\big)_{|\Gamma\b\mC}=0.
$
The operator $T_\mC^{-1}$ is bounded from $\tau^{\ep-1} L^p(X\times V)$ to $\tau^\ep L^p(X\times V)$ and, hence, $\big(T_\mC^{-1}\big)^*$ is bounded from  $\tau^{-\ep} L^{p\over p-1}(X\times V)$ to $\tau^{1-\ep} L^{p\over p-1}(X\times V)$.
We consider the following backward equation
\begin{eqnarray}
&&v\cdot \nabla_x u(x,v)-\sigma(x,v) u(x,v)=-K^*u(x,v),\label{boltzback}\\
&&u_{|\Gamma\b\mC}=g.
\end{eqnarray}
We may recast it as the following integral equation 
\begin{equation}
(I-(T_\mC^{-1})^*K^*)u=L_{\Gamma\b\mC,-\sigma}g.
\end{equation}
Here $L_{\Gamma\b\mC,-\sigma}$ denotes the operator ``$L_{\mC}$'' related to ``$(\mC,\sigma)$" given by $(\Gamma\b\mC,-\sigma)$.

\begin{theorem}
Assume that \eqref{i20b}  holds. Then 
$$
I-(T_\mC^{-1})^*K^* \textrm{ is invertible in }\mathcal{L}(L^{p\over p-1}(X\times V))
$$
and the backward equation is well-posed in $\tilde W^{p\over p-1}(X\times V)$ and 
the albedo operator $A_{\rm back}\in \mathcal{L}(L^{p\over p-1}(\Gamma\b\mC, \tau d\xi),$ 
$L^{p\over p-1}(\mC, \tau d\xi))$ is bounded. In addition the following Green's formula is valid
\begin{equation}
\int_{\Gamma_+\b\mC_+}\psi (A\phi)\tau d\xi-\int_{\Gamma_-\b\mC_-}\psi (A\phi)\tau d\xi
=\int_{\mC_-}\phi (A_{\rm back}\psi)\tau d\xi-\int_{\mC_+}\phi (A_{\rm back}\psi) \tau d\xi.\label{green_a}
\end{equation}
\end{theorem}

\begin{proof}
Condition \eqref{i20b} already provides that 
$$
I-K^*(T_\mC^{-1})^* \textrm{ is invertible in }\mathcal{L}(L^{p\over p-1}(X\times V))
$$
Then by Jacobson's Lemma we obtain that $I-(T_\mC^{-1})^*K^*$ is invertible in $\mathcal{L}(L^{p\over p-1}(X\times V))$ and we can repeat Thm. \ref{thm_solv} for the backward equation.
Then a simple application of Green's formula over $X\times V$ on $v\cdot \nabla_x(\tau u_1u_2)=\tau  v\cdot \nabla_x (u_1 u_2)$ where $u_1$ (resp. $u_2$) solves the  forward (resp. backward) Boltzmann equation with boundary condition $\phi$ on $\mC$ (resp. $\psi$ on $\Gamma\b\mC$).
\end{proof}
The same proof provides the following result.
\begin{theorem}
Assume that \eqref{i20a}  holds. Then 
$$
I-(T_\mC^{-1})^*K^* \textrm{ is invertible in }\mathcal{L}(\tau^{1-{1\over p}}L^{p\over p-1}(X\times V))
$$
and the backward equation is well-posed in $ W^{p\over p-1}(X\times V)$ and 
the albedo operator $A_{\rm back}\in \mathcal{L}(L^{p\over p-1}(\Gamma\b\mC, d\xi),$ 
$L^{p\over p-1}(\mC, d\xi))$ is bounded. In addition the following Green's formula is valid
\begin{equation}
\int_{\Gamma_+\b\mC_+}\psi (A\phi)d\xi-\int_{\Gamma_-\b\mC_-}\psi (A\phi)d\xi
=\int_{\mC_-}\phi (A_{\rm back}\psi)d\xi-\int_{\mC_+}\phi (A_{\rm back}\psi) d\xi.\label{green_b}
\end{equation}
\end{theorem}


\bigskip

We conclude this section with a technical lemma on a gain of regularity in $L^p$-spaces of scattered components of the transport solution for a very specific geometry used in section \ref{sec:diffcontrol}. 
In that lemma, the velocity set $V$ is the unit sphere $\S^{d-1}$ and $X$ is the annulus $B(0,r_1)\b\overline{B(0,r_2)}$ where $0<r_2<r_1$.

We introduce the following norm in dimension $d\ge 3$.
Let ${\mathcal M}$ the $(3d-5)$- dimensional smooth manifold $\{(v_-,\omega,y)\in (\S^{d-1})^2\times \R^d\  |\ v_-\cdot \omega=v_-\cdot y=\omega\cdot y=0\}$ and $P_{v_-,\omega,y}$ the 2-dimensional plane passing through $y$ with directions $\omega$ and $v_-$ for $(v_-,\omega,y)\in {\mathcal M}$. We also use the notation  ${\mathcal M}_{v_-}$ for the $(2d-4)$- dimensional smooth manifold $\{(\omega,y)\in \S^{d-1}\times \R^d\  |\ (v_-,\omega,y)\in {\mathcal M}\}$ for $v_-\in \S^{d-1}$.

We denote by $\gamma_{v_-,\omega,y}$ the intersection of $P_{v_-,\omega,y}$ with the boundary $\pa X=\pa B(0,r_1)\cap \pa B(0,r_2)$. The set $\gamma_{v_-,\omega,y}$ is either a circle or the union of 2 circles for almost every $(v_-,\omega,y)\in{\mathcal M}$, and we denote by $\nu_{P_{v_-,\omega,y}}(z)$ a unit normal vector to $X\cap P_{v_-,\omega,y} $ at the boundary point $z\in \gamma_{v_-,\omega,y}$.

For a measurable bounded function $\phi$ on $\pa X\times V$, we consider the norm 
$$
\|\phi\|_{L*(\Gamma)}=\int_V \Big({\rm supess}_{(\omega,y)\in {\mathcal M}_{v_-}} \int_{\gamma_{v_-,\omega,y}} |\phi|(\gamma_{v_-,\omega,y},v_-)|\nu_{P_{v_-,\omega,y}}(\gamma_{v_-,\omega,y})\cdot v_-|d\gamma_{v_-,\omega,y} \Big)dv_-.
$$
We denote by $L_*(\Gamma)$ the completion of $L^\infty(\pa X\times V)$ under this norm. For $\phi \in L_*(\Gamma)$, we also denote by $L_{\mC}\phi$ the function obtained after applying the operator $L_{\mC}$ to the restriction of $\phi$ to the set $\mC$.

In dimension $d=2$, the space $L_*(\Gamma)$ is the usual space $L^1(\Gamma, d\xi)$. We then have the following regularity result, whose proof is postponed to Appendix \ref{sec:app1p}:
\begin{lemma}
\label{lem_1p}
Assume that $\sigma\in L^\infty(X\times V)$ and $k\in L^\infty(X\times V^2)$ where $V$ is the unit sphere $\S^{d-1}$ and $X$ is the annulus $B(0,r_1)\b\overline{B(0,r_2)}$ for some $0<r_2<r_1$.

In dimension $d=2$, the operator $\big[T_{\mC}^{-1}KL_{\mC}\Big]_{|\Gamma\b\mC}$ is bounded from $L^1(\mC,d\xi)$ to 

\noindent$L^p(\Gamma\b\mC, d\xi)$ and the operator $T_{\mC}^{-1}KL_{\mC}$ is bounded from $L^1(\mC,d\xi)$ to $L^p(X\times V)$ for $1\le p< 2$. When condition\eqref{i20a} is satisfied for some $1\le p<2$ then the operator $\big[(I-T_{\mC}^{-1}K)^{-1}T_{\mC}^{-1}KL_{\mC}\Big]_{|\Gamma\b\mC}$ is bounded from  $L^1(\mC,d\xi)$ to $L^p(\Gamma\b\mC, d\xi)$.

In dimension $d\ge 3$, the operator $\big[T_{\mC}^{-1}KL_{\mC}\Big]_{|\Gamma\b\mC}$ is bounded from $L_*(\Gamma)$ to $L^p(\Gamma\b\mC, d\xi)$ and the operator $T_{\mC}^{-1}KL_{\mC}$ is bounded from $L_*(\Gamma)$ to $L^p(X\times V)$ for $1\le p< d$. When condition \eqref{i20a} is satisfied for some $1\le p<d$, then the operator $\big[(I-T_{\mC}^{-1}K)^{-1}T_{\mC}^{-1}KL_{\mC}\Big]_{|\Gamma\b\mC}$ is bounded from  $L_*(\Gamma,d\xi)$ to $L^p(\Gamma\b\mC, d\xi)$ and the operator $\big[L_{\mC}\Big]_{|\Gamma}$ is a bounded operator from $L_*(\Gamma)$ to $L_*(\Gamma)$ with a uniform bound $1+e^{2\sqrt{r_1^2-r_2^2}\|\sigma\|_\infty}$.
\end{lemma}

\section{Albedo operator and Fredholm theory}
\label{sec:albedo} 

We now develop results on the well-posedness of the transport equation under additional regularity assumptions on the scattering coefficient. We assume below that $1<p<\infty$ and define
\begin{equation}
W^p_{\mC}(X\times V):=\{u\in L^p(X\times V)\ |\ v\cdot \nabla_x u\in L^p(X\times V),\ u_{|\mC}\in L^p(\mC,d\xi)\}.
\end{equation}
We have the following result, identifying $W^p_{\mC}(X\times V)$ with more standard spaces:
\begin{lemma}
We have 
\begin{eqnarray}
W^p_{\mC}(X\times V)&=&W^p_\pm(X\times V)\\ &:=&\{u\in L^p(X\times V)\ |\ v\cdot \nabla_x u\in L^p(X\times V),\ u_{|\Gamma_\pm}\in L^p(\Gamma_\pm,d\xi)\}.\nonumber
\end{eqnarray}
\end{lemma}

\begin{proof}
According to \cite{C-CRAS-1985}, we have that $W^p_-(X\times V)=W^p_+(X\times V)=\{u\in L^p(X\times V)\ |\ v\cdot \nabla_x u\in L^p(X\times V),\ u_{|\Gamma}\in L^p(\Gamma,d\xi)\}$. Hence we have $W^p_-(X\times V)\subseteq W^p_{\mC}(X\times V)$.
We now prove $W^p_{\mC}(X\times V)\subseteq W^p_-(X\times V)$. Let $u\in W^p_{\mC}(X\times V)$ and consider the functions $\chi_\pm$ defined by 
$$
\chi_{\pm}(x,v)=
\left\lbrace
\begin{array}{l}
1\ \textrm{ when}\ (x\pm\tau_\pm(x,v)v,v)\in \mC_\pm,\\
0\ \textrm{otherwise.}
\end{array}
\right.
$$
Note that 
$
\chi_++\chi_-=1
$
and we have 
$$
v\cdot \nabla_x (\chi_\pm u)=\chi_\pm v\cdot \nabla_x u,\ (\chi_{\pm}u)_{|\Gamma_\pm\b\mC_\pm}=0,\ (\chi_{\pm}u)_{|\mC_\pm}=u_{|\mC_\pm}.
$$   
Thus, $\chi_\pm u\in W^p_-(X\times V)$ and 
$
u=\chi_+u+\chi_-u\in W^p_-(X\times V).
$
\end{proof}
\begin{lemma}
The operator $T^{-1}_\mC$ is bounded from $L^p(X\times V)$ to $W^p_-(X\times V)$.
\end{lemma}
\begin{proof}
We start with $u\in L^p(X\times V)$. Then $T^{-1}_\mC u$ has the following properties
$$
v\cdot\nabla_x T^{-1}_\mC u+\sigma T^{-1}_\mC u=u,
$$
and
\begin{eqnarray*}
\|\sigma T^{-1}_\mC u\|_{L^p(X\times V)}&\le& \|\tau\sigma\|_\infty e^{\|\tau\sigma\|_\infty}\|u\|_{L^p},\\
\ \|T^{-1}_\mC u\|_{L^p(X\times V)}&\le & e^{\|\tau\sigma\|_\infty}\|\tau u\|_{L^p}\le \|\tau\|_\infty e^{\|\tau\sigma\|_\infty}\|u\|_{L^p},
\end{eqnarray*}
and $(T^{-1}_\mC)_{|\mC} u=0$.
Thus, $T^{-1}_\mC$ is a bounded operator from $L^p(X\times V)$ to $W^p_\mC(X\times V)=W^p_-(X\times V)$.
\end{proof}

We are now ready to state several results on the compactness of the operator $T_\mC^{-1}K$; see also \cite[Chapter 4]{MK}. We first recall the 
\begin{lemma}[\cite{GLPS-JFA-88}]
The bounded operator $M$ from $W^p_-(X\times V)$ to $L^p(X)$ defined by
$$
Mu(x,v)=\int_V u(x,v) dv,\ u\in W^p_-(X\times V),
$$
is compact.
\end{lemma}
\begin{lemma}
Let $1<p<\infty$. Assume that
$$
k\in C\big(\bar X, L^p(V_v,L^{p\over p-1}(V_{v'}))\big).
$$
The bounded operator $K$ from $W^p_-(X\times V)$ to $L^p(X\times V)$ is compact.
\end{lemma}
\begin{proof}
By density of $C^1(\bar X\times V\times V)$ in $C\big(\bar X, L^p(V_v,L^{p\over p-1}(V_{v'}))\big)$ there exists a sequence $(k_n)\in C^1(\bar X\times V\times V)$ so that
\begin{equation}
\lim_{n\to +\infty}k_n=k\textrm{ in }C\big(\bar X, L^p(V_v,L^{p\over p-1}(V_{v'}))\big),\label{i30a}
\end{equation}
and 
\begin{equation}
k_n(x,v',v)=\sum_{j=0}^n k_{n,j}(x,v)\phi_{n,j}(v'),\ (x,v',v)\in \bar X\times V\times V,\label{i30b}
\end{equation}
for each $n\in\N$ and some functions $(k_{n,j})_{0\le j\le n}\in C(\bar X\times V)^{n+1}$, $(\phi_{n,j})_{0\le j\le n}\in C(V)^{n+1}$.
We define the operators
\begin{eqnarray}
K_n f(x,v)=\int_V k_n(x,v',v)f(x,v')dv',\ f\in L^p(X\times V).\label{i30c}
\end{eqnarray}
We have 
\begin{equation*}
K_n f(x,v)=\sum_{j=0}^n  k_{n,j}(x,v)M(\phi_{n,j}f)(x),\textrm{ a.e. }(x,v)\in X\times V;\ f\in L^p(X\times V).
\end{equation*}
The multiplication operator by the function  $\phi_{n,j}$ (that we still denote by $\phi_{n,j}$) is bounded in $W^p_-(X\times V)$.
Then $M(\phi_{n,j} .)$ is a compact operator from $W^p_-(X\times V)$ to $L^p(X\times V)$. Then multiplication operators by the functions  $k_{n,j}$ (that we still denote by $k_{n,j}$) is bounded in $L^p(X\times V)$. Therefore $k_{n,j}M(\phi_{n,j} .)$  is a compact operator from $W^p_-(X\times V)$ to $L^p(X\times V)$. Hence we obtain that $K_n$ is a compact operator from $W^p_-(X\times V)$ to $L^p(X\times V)$.

The limit \eqref{i30a} shows that $K$ is the uniform limit of $(K_n)$ as bounded operators in $L^p(X\times V)$. Therefore $K$ is also a compact operator from $W^p_-(X\times V)$ to $L^p(X\times V)$.
\end{proof}

Our main compactness results is then:
\begin{proposition}
Let $1<p<\infty$. Assume that
\begin{equation}
k\in C\big(\bar X, L^p(V_v,L^{p\over p-1}(V_{v'}))\big).\label{hyp_comp}
\end{equation}
Then the operator $KT^{-1}_\mC$ is a compact operator in $L^p(X\times V)$.
\end{proposition}

\begin{theorem}
Under the assumption \eqref{i10a} the operator $I-KT^{-1}_\mC$ is invertible in $L^p(X\times V)$ if and only if $I-T_\mC^{-1}K$ is invertible in $L^p(X\times V)$ or in other words condition \eqref{i20b} is fulfilled (as well as the stronger condition \eqref{i20a}). 

Under the stronger assumption \eqref{hyp_comp},  $I-T_\mC^{-1}K$ is invertible if and only if $1$ is not an eigenvalue of the compact operator $KT^{-1}_\mC$. 
And if the scattering coefficient $k$ is replaced by $\lambda k$, then the forward transport equation with general boundary conditions is invertible for all values of $\lambda$ except possibly for a countable number of values $\lambda$ such that $1$ belongs to the spectrum of $(\lambda K)T^{-1}_\mC$, in which case the transport equation is {\em not} invertible.

\end{theorem}

The proof of the Theorem relies on the identity
$
(I-AB)^{-1}=I+A(I-BA)^{-1}B
$
for bounded operators $A$, $B$ with $I-BA$ invertible. 
Indeed the operator $K$ defines a bounded operator in $L^p(X\times V)$ by Lemma \ref{lem_K}. The rest of the statements relies on the standard spectral properties of compact operators.

Using a compactness result in \cite{GS-CRAS-02}, we may also derive a compactness result in the $L^1$ framework; see \cite[Chapter 4]{MK}.

\medskip

We now develop the Fredholm theory of the albedo operator under the above compactness conditions. Let $1<p<\infty$ and $\mG$ be a subset of $\Gamma$. We define $\gamma_{\mG}$ the involution in $L^{p\over p-1}(\mG,\tau d\xi)$ (or $L^{p\over p-1}(\mG,d\xi)$) given by
\begin{equation}
\gamma_\mG(\phi)(x,v)=
\left\lbrace
\begin{array}{l}
\phi(x,v),\ (x,v)\in \mG\cap \Gamma_+\\
-\phi(x,v),\ (x,v)\in \mG\cap\Gamma_-.
\end{array}
\right.
\end{equation}

\begin{theorem}
\label{thm_kerA}
Assume that \eqref{i20b} and \eqref{hyp_comp} holds. Then the albedo operator

\noindent$A\in\mathcal{L}(L^p(\mC, \tau d\xi),
L^p(\Gamma\b\mC, \tau d\xi))$ is a Fredholm operator of index 0. Moreover,
\begin{equation}
{\rm dim}({\rm ker}A)={\rm dim}\big({\rm ker}(I-KT_{\Gamma\b\mC}^{-1})\big).\label{ker_dim}
\end{equation} 
\end{theorem}

\begin{proof}
We first prove that the albedo operator has a finite dimensional kernel. By the last proposition we know that $(T_{\Gamma\b\mC}^{-1}K)^2$ is a compact operator in $L^p(X\times V)$. Hence $(I-T_{\Gamma\b\mC}^{-1}K)$ has a finite dimensional kernel. Moreover for $u\in \textrm{ker}(I-T_{\Gamma\b\mC}^{-1}K)$, we have
$u=T_{\Gamma\b\mC}^{-1}Ku\in \tilde W^p(X\times V)$, 
\begin{eqnarray}
&&A(u_{|\mC})=u_{|\Gamma\b\mC}=0,\label{i40a}\\
&& v\cdot \nabla_x u+\sigma u=Ku.\label{i40b}
\end{eqnarray}
Hence, the operator $\Phi:\textrm{ker}(I-T_{\Gamma\b\mC}^{-1}K)\to \textrm{ker}A$ defined by
$
\Phi u:=u_{|\mC}
$
is well defined and actually an isomorphism. It is one-to-one, since by \eqref{i40b} and assumption \eqref{i20b}, we have
$u=(I-T_{\mC}^{-1}K)L_\mC u_{|\mC}$ for any $u\in \textrm{ker}(I-T_{\Gamma\b\mC}^{-1}K)$.

It is also onto. Take $g\in L^p(\mC,\tau d\xi)$ and consider the solution $u\in \tilde W^p(X\times V)$ of the stationary Boltzmann equation \eqref{boltz} with boundary condition given by $g$. Then 
$
u=T_{\Gamma\b\mC}^{-1}Ku+L_{\Gamma\b \mC}Ag,
$
and we obtain that $u\in \textrm{ker}(I-T_{\Gamma\b\mC}^{-1}K)$ when $Ag=0$. Hence
$$
{\rm dim}({\rm ker}A)={\rm dim}({\rm ker}(I-T_{\Gamma\b\mC}^{-1}K)).
$$

We now prove that the cokernel of $A$ is finite-dimensional. It is equivalent to proving that the kernel of $A^*$ is finite-dimensional. By \eqref{green_a} we actually have that 
\begin{equation}
\textrm{ker}A^*=\gamma_{\Gamma\b \mC}\big(\textrm{ker}A_{\rm back}\big).\label{ker_A}
\end{equation}
Indeed we have
\begin{eqnarray*}
&&\int_{\mC}\psi (A_{\rm back}\phi) \tau d\xi =-[-\int_{\mC_+}\gamma_\mC(\psi)A_{\rm back}(\phi)\tau d\xi+\int_{\mC_-}\gamma_\mC(\psi)(A_{\rm back}\phi)\tau d\xi]\\
&=&-[\int_{\Gamma_+\b\mC_+}(A\gamma_\mC(\psi))\phi \tau d\xi-\int_{\Gamma_-\b\mC_-}(A\gamma_\mC(\psi))\phi\tau d\xi] \
=\ -\int_{\Gamma\b\mC}\gamma_\mC(\psi)A^*(\gamma_{\Gamma\b\mC}(\phi))\tau d\xi.
\end{eqnarray*}
Then as above we have that ${\rm ker} A_{\rm back}$ is isomorphic to ${\rm ker}(I-(T_{\Gamma\b\mC}^{-1})^*K^*)$. The latter kernel is finite dimensional since  $(T_{\Gamma\b\mC}^{-1})^*K^*=\big(KT_{\Gamma\b\mC}^{-1}\big)^*$ is a compact operator and
\begin{equation}
{\rm dim}({\rm ker}A^*)={\rm dim}({\rm ker}(I-{T_{\Gamma\b\mC}^{-1}}^*K^*)).\label{i80}
\end{equation}
It remains to prove that
$$
{\rm dim}({\rm ker}A^*)={\rm dim}({\rm ker}(I-T_{\Gamma\b\mC}^{-1}K))= {\rm dim}({\rm ker}(I-{T_{\Gamma\b\mC}^{-1}}^*K^*))={\rm dim}({\rm ker}A).
$$
Since $(T_{\Gamma\b\mC}^{-1})^*K^*$ is a compact operator, then $I-((T_{\Gamma\b\mC}^{-1})^*K^*)$ is a Fredholm operator of index $0$ \cite[Page 208]{Es} and
$
{\rm dim}\big({\rm ker}(I-{T_{\Gamma\b\mC}^{-1}}^*K^*)\big)={\rm dim}\big({\rm ker}(I-KT_{\Gamma\b\mC}^{-1})\big).
$
Therefore it remains to prove that 
\begin{equation}
{\rm dim}\big({\rm ker}(I-KT_{\Gamma\b\mC}^{-1})\big)={\rm dim}\big({\rm ker}(I-T_{\Gamma\b\mC}^{-1}K)\big).\label{i60}
\end{equation}
We just have to prove that
$T_{\Gamma\b\mC}^{-1}$ defines an isomorphism from ${\rm ker}(I-KT_{\Gamma\b\mC}^{-1}\big)$ to  ${\rm ker}(I-T_{\Gamma\b\mC}^{-1}K\big)$.
Indeed we have $T_{\Gamma\b\mC}^{-1}u=T_{\Gamma\b\mC}^{-1}KT_{\Gamma\b\mC}^{-1}u$ when $u=KT_{\Gamma\b\mC}^{-1}u$ and 
$v\cdot\nabla_x T_{\Gamma\b\mC}^{-1}u+\sigma T_{\Gamma\b\mC}^{-1}u=u $, which proves that $T_{\Gamma\b\mC}^{-1}$ is one-to-one.
For $w\in {\rm ker}(I-T_{\Gamma\b\mC}^{-1}K\big)$ then $w=T_{\Gamma\b\mC}^{-1}Kw\in \tilde W^p(X\times V)$ and 
$v\cdot\nabla_x w+\sigma w=Kw$. Hence set $u=Kw$ and we have $KT_{\Gamma\b\mC}^{-1}u=KT_{\Gamma\b\mC}^{-1}Kw=Kw=u$. Hence $T_{\Gamma\b\mC}^{-1}$ is onto from ${\rm ker}(I-KT_{\Gamma\b\mC}^{-1})$ to  ${\rm ker}(I-T_{\Gamma\b\mC}^{-1}K)$.
Identity \eqref{i60} is proved.
\end{proof}

\begin{theorem}
Assume that \eqref{i20a} and \eqref{hyp_comp} holds. Then the albedo operator $A\in \mathcal{L}(L^p(\mC, d\xi),$ 
$L^p(\Gamma\b\mC, d\xi))$ is a Fredholm operator of index 0 and \eqref{ker_dim} holds.
\end{theorem}

\begin{proof}
The proof follows the same lines as the proof of Theorem \ref{thm_kerA}.
By the last proposition $(T_{\Gamma\b\mC}^{-1}K)^2$ is a compact operator in $\tau^{1\over p}L^p(X\times V)$. Hence $(I-T_{\Gamma\b\mC}^{-1}K)$ has a finite dimensional kernel. Moreover for $u\in \textrm{ker}(I-T_{\Gamma\b\mC}^{-1}K)$, we have
$u=T_{\Gamma\b\mC}^{-1}Ku\in W^p_{\Gamma\b \mC}(X\times V)$ and \eqref{i40a} and \eqref{i40b} hold.
Hence the operator $\Phi:\textrm{ker}(I-T_{\Gamma\b\mC}^{-1}K)\to \textrm{ker}A$ defined by
$
\Phi u:=u_{|\mC}
$
is well defined and it is actually an isomorphism. Use \eqref{i20a}, $W^p(X\times V)$, $L^p(\mC,d\xi)$ in place of \eqref{i20b}, $\tilde W^p(X\times V)$, $L^p(\mC,\tau d\xi)$ appearing in the proof of Theorem \ref{thm_kerA}. Hence
$$
{\rm dim}({\rm ker}A)={\rm dim}({\rm ker}(I-T_{\Gamma\b\mC}^{-1}K)).
$$

The cokernel of $A$ is finite-dimensional (replace $\tau d\xi$ by $d\xi$ in the Green's formula in the proof of Theorem \ref{thm_kerA}), and
$$
{\rm dim}({\rm ker}A^*)={\rm dim}({\rm ker}(I-{T_{\Gamma\b\mC}^{-1}}^*K^*)).
$$
Note that ${\rm ker}(I-(T_{\Gamma\b\mC}^{-1})^*K^*)$ does not depend on the ambient space $\tau^{1-{1\over p}}L^{p\over p-1}(X\times V)$ or $L^{p\over p-1}(X\times V)$ because of nice properties of $(T_{\Gamma\b\mC}^{-1})^*$.
We conclude as in the end of the previous proof.
\end{proof}

\section{Control of solutions on convex subdomains}
\label{sec:convex}

We now consider boundary control problems for the transport equation. By boundary control, we mean control of the transport solution on $\Gamma\backslash\mC$ from prescriptions on $\mC$. We primarily consider two situations. The first one, treated in this section, concerns the control of a transport solution on $X_0\subset X$ from the incoming conditions on $\Gamma_-(X)$. Denoting by $Y=X\backslash X_0$, this is a problem of control of $\Gamma(Y)\backslash\mC$ from $\mC\subset\Gamma(Y)$.  The second one is the control on the outgoing set $\Gamma_+$ from the prescription of the incoming conditions on $\Gamma_-$. It will be analyzed in the next section.

\subsection{Extension result}

We start with an extension result, which is a relatively direct consequence of the forward transport theory developed in section \ref{sec:forward}.

Let $Z$ and $Y$ be bounded open domains of class $C^1$, $Z\cap Y=\emptyset$ with $Z$ convex and $\pa Z\subset \pa Y$. 
Define $Z_{\rm ext}=\bar Z\cup Y$ the open domain containing $Z, Y$. The objective is to write any transport solution on $Z$ as the restriction of a transport solution on $Z_{\rm ext}$. We will see that such extended solutions are not unique.

We assume that $\sigma\in L^\infty(Y\times V)$. Let $1\le p<\infty$. We consider the boundary value problem \eqref{boltz} in $Z$ where the boundary conditions are taken at ``$(\mC_-,\mC_+)=(\Gamma_-(Z),\emptyset)$".
For $g\in L^p(\Gamma_-(Z),d\xi)$, let us assume that there exists $u_0\in W^p(Z\times V)$ solution of the Boltzmann equation in $Z\times V$ with ${u_0}_{|\Gamma_-(Z)}=g$. Existence of  $u_0$ is granted when the condition \eqref{i20b} related to $(Z, \Gamma_-(Z),\emptyset)$ holds.

Set
\begin{eqnarray*}
\mG_-(Z_{\rm ext},Z)=\{(x,v)\in \Gamma_-(Z_{\rm ext})\ | \ x+tv\not\in Z,\ t\in [0,\tau_+(x,v)] \},\\
\mC_{-,1}=\mG_-(Z_{\rm ext},Z)\cup \Gamma_+(Z),\quad 
\mC_{+,1}=\Gamma_-(Z),\quad \mC_1=\mC_{-,1}\cup\mC_{+,1}.
\end{eqnarray*}
Note that $\Gamma_\pm(Y)\cap(\pa Z\times V)=\Gamma_\mp(Z)$. 
Let $\gamma\in L^p(\mG_-(Z_{\rm ext},Z),d\xi)$ and assume that there exists a solution $u_1\in W^p(Y\times V)$ of the Boltzmann equation in $Y\times V$ with ${u_1}_{|\mG_-(Z_{\rm ext},Z)}=\gamma$ and ${u_1}_{|\Gamma(Z)}={u_0}_{|\Gamma(Z)}$. 
Existence of $u_1$ is granted when the related condition \eqref{i20b} holds for $(Y,\mC_1)$, for instance when $Y$ is `sufficiently small'.
Set now 
\begin{equation}
u(x,v):=\left\lbrace
\begin{array}{l}
u_0(x,v),\ (x,v)\in Z\times V,\\
u_1(x,v),\ (x,v)\in Y\times V.
\end{array}
\right.
\end{equation}

\begin{lemma}
\label{lem_ext}
The function $u$ belongs to $W^p(Z_{\rm ext}\times V)$ and satisfies the Boltzmann equation in $Z_{\rm ext}\times V$.
\end{lemma}

\begin{proof}
We have $u_0\in W^p(Z\times V)$ and $u_1\in W^p(Y\times V)$. Hence we have $\tau^{-{1\over p}}u\in L^p(Z_{\rm ext}\times V)$ (indeed we always have $\max(\tau(Y), \tau(Z))\le \tau(Z_{\rm ext})$, where $\tau(X)$ is the function $\tau$ related to a bounded open set $X$).
We now prove that the weak derivative $v\cdot \nabla_x u$ exists and is equal to $Ku-\sigma u$, which will prove that $u\in W^p(Z_{\rm ext}\times V)$ and that $u$ solves the Boltzmann equation in $Z_{\rm ext}\times V$.
Indeed, for $\phi\in C^1_0(Z_{\rm ext}\times V)$ 
\begin{eqnarray*}
&&\int_{Z_{\rm ext}\times V}u(v\cdot \nabla_x \phi)dx dv=\int_{Z\times V}u_0(v\cdot \nabla_x \phi)dx dv
+\int_{Y\times V}u_1(v\cdot \nabla_x \phi)dx dv\\
&=&-\int_{Z\times V}(v\cdot \nabla_x)u_0\phi dx dv+\int_{\Gamma_+(Z)}u_0\phi d\xi-\int_{\Gamma_-(Z)}u_0\phi d\xi\\
&&-\int_{Y\times V}(v\cdot \nabla_x)u_1\phi dx dv
-\int_{\Gamma_-(Y)\cap(\pa Z\times V)}u_1\phi d\xi+\int_{\Gamma_+(Y)\cap(\pa Z\times V)}u_1\phi d\xi\\
&=&-\int_{Z\times V}(Ku_0-\sigma u_0)\phi dx dv-\int_{Y\times V}(Ku_1-\sigma u_1)\phi dx dv
= -\int_{Z_{\rm ext}\times V}(Ku-\sigma u)\phi dx dv.
\end{eqnarray*}
\end{proof}

\begin{remark}
The extension result may fail when $\pa Y\cap \pa Z$ is not the boundary of the convex hull of $Z$. In that case, the solution of the Boltzmann equation in $Z$ must satisfy additional compatibility conditions at the boundary $\pa Z \cap \pa Y$.
\end{remark}

Note that the extension result was very specific to the transport equations. Another way of stating the above result is to observe that any transport solution on a convex domain $Z$ may be seen as the restriction of a transport solution on a larger domain. This is clearly not true for solutions of elliptic or even wave equations, where restrictions to subsets of solutions defined on large domains have regularity properties that arbitrary solutions on such subsets do not need to possess. Such extensions are clear in the absence of scattering (only for convex domains $Z$). Our main result is that they continue to hold in the presence of scattering.

\subsection{Control by a layer peeling argument}
In this section, we still assume $1\le p<\infty$. Let $\rho\in C^2(\R^d)$ so that 
\begin{equation}
{\rm Hess}\ \rho(x)>0\ x\in \R^d\b\{0\},\ \nabla \rho(0)=0 \textrm{ and }\rho(0)=-1. 
\end{equation}
We also assume that $|\rho(x)|\to\infty$ as $|x|\to\infty$.
For $-1<s\leq 1$, we denote by $Z_s$ the convex and bounded domain $\{x\in \R^d\ |\ \rho(x)<s\}$.

For $-1<s_1<s_2\le 1$, we denote
\begin{equation}
\mG_-(s_2,s_1):=\{(x,v)\in \Gamma_-(Z_{s_2})\ |\ (x+\R v)\cap Z_{s_1}=\emptyset\},\ 1\le k\le N.\label{defg}
\end{equation}

Consider $X:=Z_1\b\bar Z_0$. Then the following result holds.

\begin{theorem}
\label{thm_peel}
Let $\sigma\in L^\infty(X\times V)$. Let $k\in L^\infty\big(X, L^p(V_v,L^{p\over p-1}(V_{v'})\big)$ when $p>1$ and $k\in L^\infty(X\times V\times V)$ when $p=1$.
Then there exists $\ep>0$ so that for any $(\beta,\gamma)\in L^p(\pa Z_0\times V,d\xi)\times L^p(\mG_-(1,1-\ep),d\xi)$  there exists a solution of the linear Boltzmann transport equation $u\in W^p(X \times V)$ with the following boundary conditions 
\begin{equation}
u_{|\pa Z_0\times V}=\beta,\ u_{|\mG_-(1,1-\ep)}=\gamma.
\end{equation}
\end{theorem}
The theorem should be interpreted as follows. A transport solution is constructed on the domain $X\times V$ with boundary conditions arbitrarily prescribed on the whole boundary of $Z_0$, i.e., on $\Gamma_+(Z_0)\cup\Gamma_-(Z_0)$, as well as on the incoming directions of $Z_1$ that are sufficiently grazing. How grazing, i.e., how small $\eps$ has to be, depends on the scattering coefficients.
 

The proof uses the following Lemma.
Let $N\in \N$, and denote 
$$
Y_{k,N}:=\{x\in X\ |\ {k-1\over N}<\rho(x)<{k\over N}\},\ 1\le k\le N.
$$
In each layer $Y_{k,N}$ we denote $\tau_k$ the function $\tau$ related to $Y_{k,N}$,
$$
\mC_{-,k}=\mG_-\big({k\over N},{k-1\over N}\big)\cup \Gamma_+(Z_{k-1\over N}),\ 
\mC_{+,k}=\Gamma_-(Z_{k-1\over N}),\ \mC_k=\mC_{-,k}\cup\mC_{+,k},
$$
and we denote $L_{\mC_k}$, $T_{\mC_k}^{-1}$, $K$, $A_{\mC_k}$ the operators  ``$L_{\mC}$, $T_{\mC}^{-1}$, $K$ and $A$" introduced in section \ref{sec:forward} for the couple ``domain--boundary conditions $ (X,\mC)"= (Y_{k,N},\mC_k)$. 

\begin{lemma}
\label{lem_phi}
There exists a constant $C_\rho$ so that
\begin{equation}
\sup_{N\in \N\atop 1\le k\le N}\|\tau_k\|_{L^\infty(Y_{k,N}\times V)}\le {C_\rho\over \sqrt{N}}.\label{phi}
\end{equation}
\end{lemma}

\begin{proof}[Proof of Lemma \ref{lem_phi}]
Let $(x,v)\in \Gamma_-(Y_{k,N})$. Let $s_0\in[0,\tau_k(x,v)]$ so that $|\nabla \rho(x+s_0v)\cdot v|=\inf_{0<s<\tau_k(x,v)} |\nabla \rho(x+sv)\cdot v|$.
First assume that $|\nabla \rho(x+s_0v)\cdot v|\le {1\over \sqrt{N}}$. Then
\begin{eqnarray*}
{1\over N}&\ge&|\rho(x+\tau_k(x,v) v)-\rho(x+s_0v)|
=(\tau_k(x,v)-s_0)\Big|\nabla\rho(x+s_0v)\cdot v\nonumber\\
&&+(\tau_k(x,v)-s_0)\int_0^1(1-\ep){\rm Hess}\rho(x+s_0v+\ep(\tau_k(x,v)-s_0)v)(v,v)d\ep\Big|\\
&\ge&C_1 (\tau_k-s_0)^2(x,v)-{1\over \sqrt{N}}(\tau_k(x,v)-s_0),
\end{eqnarray*}
where $C_1=\inf_{(z,w)\in X\times V}{\rm Hess} \rho(z)(w,w)$, and we obtain that $\tau_k(x,v)-s_0\le {1\over \sqrt{N}}{1+\sqrt{1+4C_1}\over 2C_1}$.
Similarly replacing $x+\tau_k(x,v)v$ by $x$, we obtain  that $s_0\le {1\over \sqrt{N}}{1+\sqrt{1+4C_1}\over 2C_1}$, and we get $\tau_k(x,v)\le {2\over \sqrt{N}}{1+\sqrt{1+4C_1}\over 2C_1}$.
Now assume that $|\nabla \rho(x+s_0v)\cdot v|\ge {1\over \sqrt{N}}$. Then 
$
{1\over N}=|\rho(x+\tau_k(x,v) v)-\rho(x)|\ge{\tau_k(x,v)\over \sqrt{N}}.
$
\end{proof}

\begin{proof}[Proof of Theorem \ref{thm_peel}]
Let $N\in \N$, $N\ge 3$ so that
\begin{equation}
\big\|\int_V\big(\int_V|k|^{p\over p-1}(.,v',v)dv'\big)^{p-1}dv\big\|_{L^\infty(X)}^{1\over p}\le {\sqrt{N}e^{-{C_\rho\|\sigma\|_\infty\over\sqrt{N}}}\over C_\rho}
\end{equation} 
when $p>1$, or
\begin{equation}
\|k\|_{L^\infty}\le {\sqrt{N}e^{-{C_\rho\|\sigma\|_\infty\over\sqrt{N}}}\over C_\rho |V|}
\end{equation}
when $p=1$ ($|V|$ is the nonzero Lebesgue measure of $V$ as a $(n-1)$-dimensional closed hypersurface or as a bounded open set in $\R^n$ excluding $0$).
We have 
\begin{equation}
X=Y_{1,N}\cup Y_{N,N}\cup(\cup_{k=2}^{N-1}\bar Y_{k,N}).
\end{equation}
Using Lemma \ref{lem_phi}, it follows that estimate \eqref{i70} is fulfilled in each subdomain $Y_{k,N}$, $1\le k\le N$ when $p>1$. When $p=1$ the estimate \eqref{i70b} is fulfilled instead.
We first consider \eqref{boltz} in $Y_{1, N}$ and obtain that there exists a unique solution $u_1\in W^p(Y_{1,N}\times V)$ so that
${u_1}_{|\Gamma(Z_{1\over N})}=\beta$, ${u_1}_{|\mG_{{1\over N},0}}=0$, where $\mG_{{k\over N},{k-1\over N}}$ is defined by \eqref{defg}. 
Then inductively we consider the unique solution $u_k\in W^p(Y_{k,N}\times V)$ of \eqref{boltz} in $Y_{k,N}$ with boundary conditions
${u_k}_{|\Gamma(Z_{k-1\over N})}={u_{k-1}}_{|\Gamma(Z_{k-1\over N})}$ and 
${u_k}_{|\mG_-({k\over N},{k-1\over N})}=0$ for $k<N$. For $k=N$ we consider the  unique solution $u_N\in W^p(Y_{N,N}\times V)$ of \eqref{boltz} in $Y_{N,N}$ with boundary conditions ${u_N}_{|\Gamma(Z_{N-1\over N})}={u_{N-1}}_{|\Gamma(Z_{N-1\over N})}$ and 
${u_N}_{|\mG_-({k\over N},{k-1\over N})}=\gamma$.
Then using the extension Lemma \ref{lem_ext} it follows that the function $u$ equal to $u_k$ in each $Y_{k,N}$, $1\le k\le N$, solves the transport equation with boundary condition 
$u_{|\pa Z_0\times V}=\beta$, $u_{|\mG_-(1,{N-1\over N})}=\gamma$, $\ep={1\over N}$.
\end{proof}

A first conclusion of this theorem is that a transport solution posed on a convex domain $Z_0$ can be controlled from the boundary of a larger convex domain $Z_1$. This generalizes results obtained in \cite{BCS-SIMA-16} and finds applications in a class of hybrid inverse problems for the transport equation; see \cite{BCS-SIMA-16}.

Another conclusion concerns the violation of a  Unique Continuation Property (UCP) for the transport equation in the following sense. No matter ``how big" the scattering coefficient is, this does not imply a unique continuation property. Indeed let $u_1$ and $u_2$ be solutions of the transport equation with boundary conditions 
${u_1}_{|\Gamma\cap (\pa Z_0\times V)}={u_2}_{|\Gamma\cap (\pa Z_0\times V)}$ and $ {u_i}_{|\mG_\ep}=\gamma_i$, $i=1,2$, $\gamma_1\not=\gamma_2$.
Thus, $w=u_1-u_2$ is a non zero solution of the transport equation with zero incoming and outgoing boundary conditions at $\pa Z_0$. Therefore $w$ can be continued by zero in $Z_0$ and the continuation remains a solution of the Boltzmann transport equation in $Z_1$.

The above results should be contrasted with the case of second-order scalar elliptic equations. For such equations, the UCP holds and is equivalent to a Runge Approximation stating that any elliptic solution on a (not necessarily convex) subdomain may be {\em approximately} controlled from the boundary. On the other hand, transport solutions in the presence of large scattering (small mean free paths $\eps$) are well approximated by diffusion solutions, with a vanishing error as $\eps\to0$. The latter limit is therefore singular as $\eps\to0$. For each $\eps>0$, we obtain an exact controllability of the transport solution on a convex sub-domain while UCP does not hold. In the limit `$\eps=0$' of the diffusion equation, the control is only approximate and a result of the UCP.

Of course, the above exact control has to become unstable in the limit $\eps\to0$. This instability is analyzed in the next section.

\subsection{Example of boundary control in the diffusive regime}
\label{sec:diffcontrol}
In this subsection the domain $X$ is the annulus $B(0,2)\b \overline{B(0,1)}$, and the velocity space $V$ is the $(d-1)$-dimensional unit sphere $\S^{d-1}$. 
We assume that the absorption and scattering coefficients are constant: $\sigma=\ep^{-1}$ and $k=|V|^{-1}\ep^{-1}$. Hence the linear Boltzmann equation is 
$$
v\cdot \nabla_x u(x,v)+{1\over \ep}(u(x,v)-\bar u(x))=0,\ (x,v)\in X\times V.
$$ 
Here $\bar u$ is the mean value of $u$ over $V$.

We also slightly change the setting of the previous subsection: $Z_0$ is the inner ball $B(0,1)$ of center 0 and radius 1 and the layers $Y_{k,N}$ are now the annuli of thickness ${1\over N}$ $B(0,{k\over N}+1)\b\overline{B(0,{k-1\over N}+1)}$ and $Z_{k\over N}=B(0,{k\over N}+1)$ while the sets $\mG_-({k\over N},{k-1\over N})$ are defined by \eqref{defg}.
In this new setting, Lemma \ref{lem_phi} is replaced by the explicit estimates: For $(x,v)\in Y_{k,N}\times V$
$$
\tau_k(x,v)\in [0,{2\over N}\sqrt{2N+2k-1}],\ \sup_{N\in \N\atop 1\le k\le N}\|\tau_k\|_{L^\infty(Y_{k,N}\times V)}\le {4\over \sqrt{N}}.
$$
Conditions \eqref{i70} and \eqref{i70b} are satisfied inside any layer $Y_{k,N}$ when
$$
\|\sigma_s\|_\infty=\|\kappa_p(x)\|_{L^\infty(X)}^{1\over p}= \ep^{-1}< {\sqrt{N}\over 4}<{e^{-\|\tau\sigma\|_\infty}\over \|\tau\|_\infty},
$$
or equivalently when $N> {16 \over \ep^2}$. Therefore we set now
$$
N=\lfloor 16\ep^{-2}\rfloor +1.
$$

We construct a solution $u$ as in the proof of Theorem \ref{thm_peel}. We set $\beta=0$ on $\pa Z_0\times V$ which means that our constructed solution vanishes in the inner ball $Z_0$, and we place a boundary source $\phi$ at the boundary of the first layer:
$$
u_{|\mG_-({1\over N},0)}=\phi
$$
where $\phi\in L^\infty\big(\mG_-({1\over N},0)\big)$. The incoming boundary condition on each other layers $Y_{k,N}$ is given by 
$$
u_{|\mG_-({k\over N},{k-1\over N})}=0,\ k=2\ldots N.
$$

The next Lemma describes the solution $u$ at the boundaries of the layers.
For $k=1\ldots N$, $s={k\over N}$, we define
\begin{equation}
\mU_{+,s}=\{(x,v)\in \Gamma_+(Z_s)\ |\ (x-t v,v) \in \mG_-({1\over N},0) \textrm{ for some positive }t\},
\end{equation}
and 
$$
d_{-,{1\over N}}=\tau_-(Z_{1\over N}),
$$
as well as 
$$
d_{-,s}(x,v)=\tau_-(Z_s\b\bar Z_{1\over N})(x,v)+\tau_-(Z_{1\over N})(x-\tau_-(Z_s\b\bar Z_{1\over N})(x,v)v,v)\textrm{ for }(x,v)\in \mU_{+,s},\ k\ge 2.
$$
Also, $C_p$ and $C_{p,\ep}$, $1\le p<d$, denote the constants:
\begin{eqnarray}
C_p=\sup_{k=1\ldots N}\|A_{\mC_k}\|_{\mathcal L\big(L^p(\mC_k,d\xi),L^p(\Gamma(Y_{k,N})\b \mC_k,d\xi)\big)},\ \nonumber \\
C_{p,\ep}=\sup_{k=1\ldots N}\vertiii{\big[(I-T_{\mC_k}^{-1}K)^{-1}T_{\mC_k}^{-1}KL_{\mC_k}\Big]_{|\Gamma(Z_{k\over N})\b\mC_k}}\label{B0}
\end{eqnarray}
where $\vertiii{.}$ denotes either the uniform norm of linear bounded operators from $L^1(\mC_k,d\xi)$ to $L^p(\Gamma(Z_{k\over N})\b\mC_k,d\xi)$ in dimension $d=2$
or the uniform norm of linear bounded operators from $L_*(\Gamma(Z_{k\over N}))$ to $L^p(\Gamma(Z_{k\over N})\b\mC_k,d\xi)$ in dimension $d\ge 3$. 

\begin{lemma}
\label{lem_trace_layer}
Let $1\le p<d$. For $k=1\ldots N$, $s={k\over N}$, we have
$$
u_{|\Gamma(Z_s)}=f_0^{(k)}+f_1^{(k)}, \qquad \mbox{ with } \quad
$$
\begin{eqnarray}
f_0^{(k)}(x,v)&=&\exp^{\tau_+(Z_s\b\bar Z_{1\over N})(x,v)\over \ep}\phi(x+\tau_+(Z_s\b\bar Z_{1\over N})(x,v)v,v)  \nonumber \\ && +\chi_{\mU_{+,s}}(x,v)e^{-{d_{-,s}(x,v)\over \ep}}\phi(x-d_{-,s}(x,v)v,v)\label{f0i}
\end{eqnarray}
for $(x,v)\in \Gamma(Z_s)$ and 
$$
f_1^{(k)}\in L^p\Big(\Gamma(Z_s), d\xi\Big),
$$
where the source $\phi\in L^\infty\Big(\mG_-({1\over N},0)\Big)$ is extended by $0$ on $\cup_{l=1}^N\pa Z_{l\over N}\times V$. Both $f_0^{(k)}$ and $f_1^{(k)}$ vanish on $\mG_-(s,s-{1\over N})$ when $k\ge 2$.

In addition 
\begin{equation}
\|f_0^{(k)}\|_k\le \big(1+e^{2\sqrt{(s-{1\over N})(2+s+{1\over N})}\over\ep}\big)\|\phi\|_1, \label{f0e}
\end{equation}
\begin{equation}
\|f_1^{(k)}\|_{L^p\big(\Gamma(Z_s)),d\xi\big)}\le  C_{p,\ep}C_p \big(C_p^{k-1}+{C_p^{k}-1\over C_p-1}(1+e^{2\sqrt{3+{2\over N}}\over\ep})\big)\|\phi\|_1,\label{f1e}
\end{equation}
where $\|\cdot\|_k=\|\cdot\|_{L^1(\Gamma(Z_s),d\xi)}$ in dimension $d=2$ and $\|\cdot\|_k=\|\cdot\|_{L_*(\Gamma(Y_{k,N}))}$ in dimension $d\ge 3$. 
\end{lemma}
This lemma is proved in appendix \ref{sec:proofeps}.

\paragraph{Construction of the source.} Let $\eta\in (0,1)$ and $\rho\in C_0^\infty(\R,\R)$, $\rho\ge 0$, ${\rm supp}\rho\subseteq(-1,1)$ with $\int_{-1}^1\rho(s)ds=1$ and $\rho$ even. We set 
$$
\rho_\eta(x,v)= \eta^{-1}\rho\big({1+{1\over 2N}-|x-(x\cdot v) v|\over \eta}\big),\ (x,v)\in \R^d\times V,
$$
and define
$$
\mU_\eta=\{(x,v)\in \Gamma_-(Z_1)\ |\ |x-(x\cdot v)v|\in (-\eta+1+{1\over 2N},\eta+1+{1\over 2N} ) \}.
$$
When $\eta\le {1\over 2N}$, then any ray $x+\R v$, $(x,v)\in {\rm supp}\rho_\eta$, intersects the ball $Z_{1\over N}$ without penetrating $Z_0$.
We denote by $u_\eta$ the transport solution in Lemma \ref{lem_trace_layer} where ``$\phi$"$=\phi_\eta=[\rho_\eta]_{|\Gamma_-(Z_{1\over N})}$.

We obtain the following result.
\begin{theorem}
\label{thm_eta}
Let $N\ge 2$. The source $\phi_\eta$ satisfies
\begin{equation}
|V||\S^{d-2}|\le \|\phi_\eta\|_{L^1(\Gamma_-(Z_{1\over N}),d\xi)}\le |V||\S^{d-2}|\big({3\over 2}\big)^{d-2}.\label{B30a}
\end{equation}
(In dimension $d=2$ this is an equality.)
When $d=2$ and $1<p<2$, $q^{-1}+p^{-1}=1$ and   
\begin{equation}
\eta\le {e^{ q\over 2\ep}\over C C_{p,\ep}^qC_p^q \big(C_p^{N-1}+{C_p^N-1\over C_p-1}(1+e^{4\over\ep})\big)^q},\label{B30b}
\end{equation}
for some universal constant $C$ then
\begin{equation}
\int_{\mU_\eta} u_\eta d\xi\ge e^{ 1\over 2\ep}|V||\S^{d-2}|.\label{B30c}
\end{equation}
When $d\ge 3$ and  $2<p<d$, $2^{-1}<q^{-1}<1-d^{-1}$ and  
\begin{equation}
\eta\le {e^{ q\over (2-q)\ep}(|V||\S^{d-2}|)^{2q\over 2-q}\over \big(C^{1\over q} \tilde C C_{p,\ep}C_p\big)^{2q\over 2-q}\big(C_p^{N-1}+{C_p^N-1\over C_p-1}(1+e^{4\over\ep})\big)^{2q\over 2-q}} \label{B30d}
\end{equation}
for some universal constant $C$ and a constant $\tilde C$ which depends only on $\rho$, then the lower bound \eqref{B30c} holds.

\end{theorem}
The theorem is also proved in appendix \ref{sec:proofeps}. Its interpretation is as follows. For any source $\phi_\eta$ generating a solution of order $O(1)$ in the vicinity of the ball $B(0,1)$, while the solution is exactly $0$ inside that ball, then the transport solution inside the ball $B(0,2)$ is necessarily exponentially large (see \eqref{B30c}) close to $|x|=2$. Equivalently, by linearity of the transport equation, the above rescaled control of order $O(1)$ at the boundary $|x|=2$ generates a source in the vicinity of $|x|=1$ of order at most $e^{-\frac1{2\eps}}$. This provides a quantitative example of the instability of the boundary control in the diffusive regime that is consistent with the unique continuation principle that applies in the diffusion limit.

\section{Control of outgoing boundary conditions}
\label{sec:outgoing}

We now consider the control problem aiming to find incoming conditions on $\Gamma_-$ such that the outgoing conditions on $\Gamma_+$ are prescribed. This is a question on the range of the albedo operator. We obtained in an earlier section that the albedo operator was a Fredholm operator with vanishing index. We now show that the dimensions of its kernels and co-kernels, which have to be equal, do not necessarily vanish. In that case, some outgoing conditions cannot be controlled from $\Gamma_-$, answering the control problem negatively. By duality, this also shows the existence of non-trivial incoming conditions such that nothing comes out of the domain (the trace of the solution on $\Gamma_+$ vanishes.

\subsection{Selfadjoint operators}
In this section we assume that $p=2$ and consider the setting with $\mC=\Gamma_-$ and hence $\Gamma\b\mC=\Gamma_+$. We denote by $L^2_{\rm even}(X\times V)$, resp. $L^2_{\rm odd}(X\times V)$, the closed subspace of $L^2(X\times V)$ that consists of functions that are even, resp. odd, in the $v$-variable:
$$L^2_{\rm even}(X\times V)=\{f\in L^2(X\times V)\ |\ f(x,v)=f(x,-v) \textrm{ a.e. }(x,v)\in X\times V\},$$
and we denote by $P_{\rm even}$, resp. $P_{\rm odd}$, the orthogonal projection onto that subspace $L^2_{\rm even}(X\times V)$, resp. $L^2_{\rm odd}(X\times V)$. 

We start with  the following lemma:
\begin{lemma}
Assume that 
\begin{equation}
\sigma(x,v)=\sigma(x,-v) \textrm{ a.e. }(x,v)\in X\times V.\label{H1a}
\end{equation}
Then
\begin{equation}
\big(T_{\Gamma_+}^{-1}\big)^*f(x,-v)=T_{\Gamma_+}^{-1}f(x,v),\ \textrm{ a.e. }(x,v)\in X\times V,\label{i100}
\end{equation}
and $f\in L^2(X\times V)$. As a consequence 
$$P_{\rm even}\big(T_{\Gamma_+}^{-1}\big)^*=P_{\rm even}T_{\Gamma_+}^{-1},\ P_{\rm odd}\big(T_{\Gamma_+}^{-1}\big)^*=-P_{\rm odd}T_{\Gamma_+}^{-1}$$ 
and $P_{\rm even}T_{\Gamma_+}^{-1}P_{\rm even}$ is a selfadjoint operator in $L^2(X\times V)$ and $P_{\rm odd}T_{\Gamma_+}^{-1}P_{\rm odd}$ is skew-Hermitian.
\end{lemma}

\begin{proof}
We recall that
$$
T_{\Gamma_+}^{-1}f(x,v)=
-\int_0^{\tau_+(x,v)}E_+(x,v,t)f(x+tv,v)dt,\ (x,v)\in X\times V,
$$
and 
$$
\big(T_{\Gamma_+}^{-1}\big)^*f(x,v)=
-\int_0^{\tau_-(x,v)}E_+(x-tv,v,t)f(x-tv,v)dt\ (x,v)\in X\times V.
$$
From \eqref{H1a}, it follows that $E_+(x+tv,-v,t)=E_+(x,v,t)$, which proves the Lemma.
\end{proof}

Let $(x_0,y_0)\in X^2$, $x_0\not=y_0$, and let $\eta\in (0,{|x_0-y_0|\over 4})$ so that the Euclidean balls $B(x_0,\eta)$ and $B(y_0,\eta)$ centered at $x_0$ and $y_0$ with radius $\eta$ are included in $X$. Then let 
\begin{equation}
\psi=K(\chi_{B(x_0,\eta)}-\chi_{B(y_0,\eta)})=\big(\chi_{B(x_0,\eta)}-\chi_{B(y_0,\eta)}\big)\sigma'_s.\label{i101}
\end{equation}

Performing the change of variables ``$y=x+tv$", we have
\begin{eqnarray*}
&&\l T_{\Gamma_+}^{-1}\psi,\psi\r 
=-\int_{X\times V}\int_0^{\tau_+(x,v)}E_+(x,v,t)\psi(x+tv,v)dt \psi(x,v)dx dv\\
&=&-\int_{X^2}{e^{|y-x|\int_0^1\sigma(x+\ep(y-x),\widehat{y-x})d\ep}\over |y-x|^{d-1}}\psi(y,\widehat{y-x}) \psi(x,\widehat{y-x})dx dy\\
&=&-\int_{B(x_0,\eta)^2\cup B(y_0,\eta)^2}{e^{|y-x|\int_0^1\sigma(x+\ep(y-x),\widehat{y-x})d\ep}\sigma'_s(x,\widehat{y-x})\sigma'_s(y,\widehat{y-x})dxdy\over |y-x|^{d-1}}\\
&&+\int_{\big(B(x_0,\eta)\times B(y_0,\eta)\big)\cup\big(B(y_0,\eta)\times B(x_0,\eta)\big)}
\hskip-1cm{e^{|y-x|\int_0^1\sigma(x+\ep(y-x),\widehat{y-x})d\ep}\sigma'_s(x,\widehat{y-x})\sigma'_s(y,\widehat{y-x})dxdy\over |y-x|^{d-1}}\\
&\ge &-4\eta e^{2\eta\|\sigma\|_\infty} {\rm Vol}(B(0,\eta))\sup_{W_\eta} (\sigma'_s)^2
+2e^{|x_0-y_0|\inf \sigma\over 2}{\rm Vol}(B(0,\eta))^2{2^{d-1}\over |x_0-y_0|^{d-1}}\inf_{W_\eta} (\sigma'_s)^2,
\end{eqnarray*}
where we have defined
$$
W_\eta=(B(x_0,\eta)\cup B(y_0,\eta))\times V.
$$
We used that 
\begin{eqnarray*}
&&\int_{B(x_0,\eta)^2}{dx dy\over|x-y|^{d-1}}\le 2\eta{\rm Vol}(B(0,\eta)),\\ 
&&\int_{\big(B(x_0,\eta)\times B(y_0,\eta)\big)}{dx dy\over |y-x|^{d-1}}dx dy\ge {\rm Vol}(B(0,\eta))^2{2^{d-1}\over |x_0-y_0|^{d-1}}. 
\end{eqnarray*}
Hence we arrive at the conclusion that
\begin{eqnarray}
\l T_{\Gamma_+}^{-1}\psi,\psi\r 
&\ge& {2^dc_d^2\eta^{2d}\over |x_0-y_0|^{d-1}}e^{|x_0-y_0|\inf \sigma\over 2}\inf_{W_\eta} (\sigma'_s)^2\label{i102}\\
&&\times\Big(1-{2^{2-d}|x_0-y_0|^{d-1}\sup_{W_\eta} (\sigma'_s)^2\over c_d\eta^{d-1}\inf_{W_\eta} (\sigma'_s)^2}e^{2\eta\|\sigma\|_\infty-{|x_0-y_0|\inf \sigma\over 2}}\Big),\nonumber
\end{eqnarray}
where $c_d={\rm Vol}(B(0,1))$.
Therefore when the right-hand side of the above equality is positive then the operator $P_{\rm even}T_{\Gamma_+}^{-1}P_{\rm even}$ has a positive eigenvalue by the min-max principle.

We also assume that 
$k\in C\big(\bar X, L^2(V^2)\big)$ (see \eqref{hyp_comp}). In particular
\begin{equation}
{\rm dim}({\rm ker}A)={\rm dim}\big({\rm ker}(I-KT_{\Gamma_+}^{-1})\big).\label{i110}
\end{equation}

\begin{lemma}
Assume that
\begin{equation}
k(x,v',v)=k(x,v,v')=k(x,v',-v)\textrm{ a.e. } (x,v',v)\in X\times V^2.\label{H1c}
\end{equation}
Then $K$ is a bounded selfadjoint operator in $L^2(X\times V)$ and  $K$ maps $L^2(X\times V)$ to $L^2_{\rm even}(X\times V)$.
\end{lemma}

For the rest of this section we assume that hypotheses \eqref{H1a} and \eqref{H1c} hold and we assume that $K$ defines a positive operator:
$$
\l K\phi,\phi\r\ge 0,\ \phi\in L^2(X\times V).
$$
We denote by 
\begin{equation}\label{eq:R}
R= K^{\frac12}
\end{equation}
the non negative square root of the operator $K$. Note that $L^2_{\rm odd}(X\times V)\subseteq{\rm ker}K={\rm ker}R$.
We are interested in the existence of a positive eigenvalue for $RT_{\Gamma_+}^{-1}R$. 
Note that
\begin{equation}
\psi=R\phi,\ \phi:=R(\chi_{B(x_0,\eta)}-\chi_{B(y_0,\eta)}).
\end{equation}
Then 
\begin{equation*}
\|\phi\|_{L^2}^2=\l K(\chi_{B(x_0,\eta)}-\chi_{B(y_0,\eta)}), \chi_{B(x_0,\eta)}-\chi_{B(y_0,\eta)}\r
\le 2c_d\sup_{W_\eta}\sigma_s\eta^d.
\end{equation*}
Hence from \eqref{i102} it follows that
\begin{eqnarray}
{\l RT_{\Gamma_+}^{-1}R\phi,\phi\r\over \|\phi\|^2}&\ge &
{2^{d-1}c_d\eta^d\over |x_0-y_0|^{d-1}}e^{|x_0-y_0|\inf \sigma\over 2}{\inf_{W_\eta} \sigma_s^2\over \sup_{W_\eta}\sigma_s}
\label{i103}\\
&&\times
\Big(1-{2^{2-d}|x_0-y_0|^{d-1}\sup_{W_\eta} \sigma_s^2\over c_d\eta^{d-1}\inf_{W_\eta} \sigma_s^2}e^{2\eta\|\sigma\|_\infty-{|x_0-y_0|\inf \sigma\over 2}}\Big).\nonumber
\end{eqnarray}

\subsection{Non-controllability result}
Under the assumption
\begin{equation}
\sigma-\sigma_s\ge 0,\label{H1d}
\end{equation}
the transport equation is well-posed \cite{ES-ApplAnal-93} and the albedo operator is well-defined (here $\sigma_s=\sigma_s'$).

We also recall that pairs $(\sigma_i,k_i)\in L^\infty(X\times V)\times L^1(X\times V\times V)$, $i=1,2$, are said to be equivalent if there exists $\varphi\in {\tilde W}^\infty(X\times V)$, $\varphi>0$, $\varphi_{|\pa X\times V}=1$, so that
\begin{equation*}
\sigma_2(x,v)=\sigma_1(x,v)-v\cdot\nabla_x\ln\varphi(x,v),\ k_2(x,v',v)={\varphi(x,v)\over \varphi(x,v')}k_1(x,v',v).
\end{equation*}
Equivalent pairs define the same albedo operator \cite{ST-ProcAmer-09} when solvability conditions are satisfied.

\begin{theorem}
\label{thm_neg_eig}
Let $(M_1,M_2, M_3)\in (0,+\infty)^2$ and let $k\in C(\bar X,L^2(V^2))$ so that \eqref{H1c} holds and $K$ is non-negative and that
$$
M_1\le \sigma_s\le M_2 \textrm{ a.e. }
$$
Then there exists a constant $C\ge M_2$ that depends on $X$, $d$, $M_1,$ $M_2$ and $M_3$ so that for any $\sigma\in L^\infty(X\times V)$ satisfying \eqref{H1a},
\begin{equation}
 C\le \sigma\le C+M_3\textrm{ a.e.}\label{i105}
\end{equation}
the operator $KT_{\Gamma_+}^{-1}$ has a positive eigenvalue greater than or equal to 1. In particular for such a $\sigma$ then there exists $\lambda\in (0,1)$ so that the albedo operator defined with respect to the coefficients $(\sigma,\lambda k)$ (or any pair of coefficients equivalent to $(\sigma,\lambda k)$) has a non zero kernel. 
\end{theorem}

\begin{proof}
Fix $(x_0,y_0,\eta)$ as above and let $C\in (M_2,+\infty)$. Then for any $\sigma$ satisfying \eqref{i105} and for the function $\phi$ previously defined we have 
\begin{equation}
{\l RT_{\Gamma_+}^{-1}R\phi,\phi\r\over \|\phi\|^2}\ge 
{2^{d-1}c_d\eta^d e^{C|x_0-y_0|\over 2}\over |x_0-y_0|^{d-1}}{M_1^2\over M_2}
\Big(1-{2^{2-d}|x_0-y_0|^{d-1}M_2^2\over c_d\eta^{d-1}M_1^2}e^{{2\eta-|x_0-y_0|\over 2}C+2M_3\eta}\Big).
\end{equation}
The right-hand side goes to $\infty$ as $C\to +\infty$. Hence for $C$ large enough the right-hand side of the latter estimate is bigger than 1 and the selfadjoint operator $RT_{\Gamma_+}^{-1}R$ admits a positive eigenvalue $\mu_0$ bigger than one. And $\mu_0$ is also an eigenvalue for the operator $KT_{\Gamma_+}^{-1}=R^2T_{\Gamma_+}^{-1}$. Then rescaling $k$ by $\mu_0^{-1} k$ we obtain that $\mu_0^{-1}KT_{\Gamma_+}^{-1}$ has $1$ as an eigenvalue. By \eqref{i110} it follows that the albedo operator for the coefficients $(\sigma,\mu_0^{-1}k)$ has a nontrivial kernel.
\end{proof}

We have the following improvement of Theorem \ref{thm_neg_eig}.
\begin{theorem}
\label{thm_neg_eig2}
Let $(M_1,M_2, M_3)\in (0,+\infty)^3$, let $N\in \N$ and let $k\in C(\bar X,L^2(V^2))$ so that \eqref{H1c} holds and $K$ is non-negative and that
$$
M_1\le \sigma_s\le M_2 \textrm{ a.e. }
$$
Then there exists a constant $C\ge M_2$ that depends on $X$, $d$, $M_1,$ $M_2$, $M_3$ and $N$ so that for any $\sigma\in L^\infty(X\times V)$ satisfying \eqref{H1a},
\begin{equation}
 C\le \sigma\le C+M_3\textrm{ a.e.}\label{i106}
\end{equation}
the operator $KT_{\Gamma_+}^{-1}$ has $N$ positive eigenvalues greater or equal to 1. 
\end{theorem}

\begin{proof}
Fix $x_0\in X$ and let $r\in (0,+\infty)$ so that $B(x_0,2r)\subseteq X$. Consider a plane $P$ passing through $x_0$ and spanned by two unit vectors $v_1$ and $v_2$ orthogonal to each other. Consider the sequence of $N$ distinct points $z_i$ on the sphere $S(x_0,r)$:
$$
z_l:=x_0+r\cos({l\pi\over N})v_1+r\sin({l\pi\over N})v_2.
$$ 
and set 
$$
z_{N+l}:=2x_0-z_l,\ l=1\ldots N.
$$
Then there exists $\alpha(N,r)\in (0,r)$ which depends only on $N$ and $r$ so that
$$
2r=|z_{N+l}-z_l|=\max_{m,n=1\ldots 2N}|z_m-z_n|\ge \alpha(N,r)+\max_{m,n=1\ldots N\atop m \not=n+N\textrm{ or }
n\not=m+N}|z_m-z_n|,
$$
and
$
\inf_{m,n=1\ldots 2N\atop m\not=n}|z_m-z_n|\ge\alpha(N,r).
$
Consider the family of $N$ functions 
$$
\phi_l=R\big(\chi_{B(z_{l+N}, {\alpha(N,r)\over 4})}-\chi_{B(z_l, {\alpha(N,r)\over 4})}\big),\ l=1\ldots N.
$$
Then for any $\sigma$ satisfying \eqref{i105} and for the function $\phi_l$ we have 
\begin{equation}
{\l RT_{\Gamma_+}^{-1}R\phi_l,\phi_l\r\over \|\phi_l\|^2}\ge \gamma_1(C)
\end{equation}
where
\begin{equation}
\gamma_1(C):=
{2^{-2d}c_d\alpha(N,r)^d e^{Cr}\over r^{d-1}}{M_1^2\over M_2}
\Big(1-{2^{d-1}r^{d-1}M_2^2\over c_d\alpha(N,r)^{d-1}M_1^2}e^{{\alpha(N,r)-4r\over 4}C+{M_3\alpha(N,r)\over 2}}\Big).
\label{i121b}
\end{equation}

Then write
\begin{equation}
\psi_l=R(\phi_l).\label{i120}
\end{equation}
Performing a change of variables ``$y=x+tv$", we have for $l\not=k$
\begin{eqnarray*}
|\l T_{\Gamma_+}^{-1}\psi_l,\psi_k\r| 
&=&\int_{X^2}{e^{|y-x|\int_0^1\sigma(x+\ep(y-x),v)d\ep}\over |y-x|^{d-1}}\psi_l(y,v) \psi_k(x,v)dx dy\\
&\le&\max\Big[\int_{W_{1,k,l}}
{e^{|y-x|\int_0^1\sigma(x+\ep(y-x),v)d\ep}\sigma'_s(x,v)\sigma'_s(y,v)dx dy\over |y-x|^{d-1}},\\
&&\int_{W_{2,k,l}}
{e^{|y-x|\int_0^1\sigma(x+\ep(y-x),v)d\ep}\sigma'_s(x,v)\sigma'_s(y,v)dxdy\over |y-x|^{d-1}}\Big]\\
&\le &c_d^2e^{(r-{\alpha(N,r)\over 4})\|\sigma\|_\infty}2^{-3d}\alpha(N,r)^{d+1}\|\sigma'_s\|_\infty^2,
\end{eqnarray*}
where we used the notation $v$ for the unit vector $\widehat{y-x}$ inside the above integrals and we introduced the sets 
$$
W_{1,k,l}=\big(B(z_l,{\alpha(N,r)\over 4})\times B(z_{k+N},{\alpha(N,r)\over 4})\big)\cup\big(B(z_{l+N},{\alpha(N,r)\over 4})\times B(z_k,{\alpha(N,r)\over 4})\big),
$$
$$
W_{2,k,l}=\big(B(z_l,{\alpha(N,r)\over 4})\times B(z_k,{\alpha(N,r)\over 4})\big)\cup\big(B(z_{l+N},{\alpha(N,r)\over 4})\times B(z_{k+N},{\alpha(N,r)\over 4})\big).
$$
In addition,
\begin{eqnarray*}
\|\phi_l\|_{L^2}^2&=&\l K(\chi_{B(z_{l+N}, {\alpha(N,r)\over 4})}-\chi_{B(z_l, {\alpha(N,r)\over 4})}),\chi_{B(z_{l+N}, {\alpha(N,r)\over 4})}-\chi_{B(z_l, {\alpha(N,r)\over 4})}\r \\
&= &\int_{X\times V}\sigma_s'(x,v)(\chi_{B(z_{l+N}, {\alpha(N,r)\over 4})}+\chi_{B(z_l, {\alpha(N,r)\over 4})}) \ 
\ge  2^{1-2d}c_d\alpha(N,r)^d\inf \sigma_s'. 
\end{eqnarray*}
Hence we arrive at the conclusion that
\begin{eqnarray}
&&\big|\l RT_{\Gamma_+}^{-1}R{\phi_l\over \|\phi_l\|_{L^2}},{\phi_k\over \|\phi_k\|_{L^2}}\r\big| 
\le c_d 2^{-d-1}
e^{(r-{\alpha(N,r)\over 4})\|\sigma\|_\infty}\alpha(N,r){\|\sigma'_s\|_\infty^2\over \inf \sigma_s'}
\nonumber\\
&\le& c_d 2^{-d-1}
e^{(r-{\alpha(N,r)\over 4})(C+M_3)}\alpha(N,r){M_2^2\over M_1}.
\label{i121}
\end{eqnarray}
Comparing \eqref{i121} and \eqref{i121b}, we obtain
\begin{equation*}
{\l RT_{\Gamma_+}^{-1}R\phi_m,\phi_m\r\over \|\phi_m\|^2}
\ge \gamma_2(C)\big|\l RT_{\Gamma_+}^{-1}R{\phi_l\over \|\phi_l\|_{L^2}},{\phi_k\over \|\phi_k\|_{L^2}}\r\big|,
\end{equation*}
for $m,l,k=1 \ldots N$ with $l\not=k$, where
$$
\gamma_2(C)={2^{-d+1}\alpha(N,r)^{d-1} e^{-M_3r+{\alpha(N,r)\over 4}(C+M_3)}M_1^3\over r^{d-1}M_2^3}
\Big(1-{2^{d-1}r^{d-1}M_2^2e^{{\alpha(N,r)-4r\over 4}C+{M_3\alpha(N,r)\over 2}}\over c_d\alpha(N,r)^{d-1}M_1^2}\Big).
$$
Note that $\gamma_2(C)\to+\infty$ as $C\to +\infty$.  Then consider the $N$-dimensional vector space $E$ spanned by the orthonormal vectors $\phi_l$, $l=1\ldots N$ (the balls $B(z_i,{\alpha(N,r)\over 4})$ are mutually disjoint). Assume that the compact selfadjoint operator $RT_{\Gamma_+}^{-1}R$ has (only) $N-j$ positive eigenvalues greater than one with $j\geq1$. Consider the subspace $F$ generated by the $N-j$ eigenvectors related to these eigenvalues. Then there exists $(\beta_i)_{i=1\ldots N}\in \C^N$ so that
$$
\sum_{j=1}^N|\beta_j|^2=1\textrm{ and }w:=\sum_{j=1}^N\beta_j {\phi_j\over \|\phi_j\|_{L^2}}\in F^\bot\ \textrm{(orthogonal space to $F$)} 
$$
and we obtain that 
$$
\l RT_{\Gamma_+}^{-1}R w,w\r \le \|w\|_{L^2}^2 \leq N.
$$
However
\begin{eqnarray*}
\l RT_{\Gamma_+}^{-1}R w,w\r&=&\sum_{j=1}^N|\beta_j|^2\l RT_{\Gamma_+}^{-1}R{\phi_j\over \|\phi_j\|_{L^2}},{\phi_j\over \|\phi_j\|_{L^2}}\r
+\sum_{i,j=1\ldots N\atop i\not=j}\bar \beta_j\beta_i\l RT_{\Gamma_+}^{-1}R{\phi_i\over \|\phi_i\|_{L^2}},{\phi_j\over \|\phi_j\|_{L^2}}\r\\
&\ge&\gamma_1(C)(1-{N(N-1)\over 2\gamma_2(C)})\to +\infty,
\end{eqnarray*}
as $C\to \infty$.
Therefore, for $C$ sufficiently large,  $RT_{\Gamma_+}^{-1}R$ must have (at least) $N$ positive eigenvalues greater than one. As a consequence, $KT_{\Gamma_+}^{-1}$ also has $N$ positive eigenvalues greater than one.
\end{proof}


Theorem \ref{thm_neg_eig} shows that the albedo operator has a non-trivial kernel for specific choices of the scattering coefficients $(\sigma,k)$. Theorem \ref{thm_neg_eig2} states that $(\sigma,\lambda k)$ gives rise to non-trivial kernels for $N$ different values of $\lambda$ (counting multiplicities). Since the index of the albedo operator vanishes, we obtain by duality a corresponding number of incoming conditions on $\Gamma_-$ leading to vanishing outgoing solutions on $\Gamma_+$.

The next result shows that when the kernel is non-trivial, it is generically one-dimensional.


\subsection{Simple eigenvalues are generic.}

\begin{theorem}
\label{thm_simple}
Let $\sigma\in C(\bar X\times V)$ satisfy \eqref{H1a}. Let $H$ be the {\em subspace} of $C(X, L^2(V^2))$ defined by 
$$
H:=\{k\in C(\bar X\times V\times V )\ |\ k\textrm{ satisfies }\eqref{H1c}\textrm{ and }K \textrm{ nonnegative operator in }L^2(X\times V)\}.
$$
Then there is a dense subset $D$  of $H$ in $C(\bar X\times V\times V )$ so that for any $k\in D$ the non zero eigenvalues of the operator $T^{-1}_{\Gamma_+}K$ are simple. In addition all the eigenvalues of the operator  $T^{-1}_{\Gamma_+}K$ restricted to $L^2_{\rm even}(X\times V)$ are generically simple.
\end{theorem}

\begin{proof}
We first prove that $0$ is {\em generically not} an eigenvalue of the operator $T^{-1}_{\Gamma_+}K$ restricted to $L^2_{\rm even}(X\times V)$.
Assume that $\lambda=0$ is an eigenvalue for $T^{-1}_{\Gamma_+}K$ restricted to $L^2_{\rm even}(X\times V)$. This is equivalent to saying that $0$ is an eigenvalue for $KP_{\rm even}$ since  $T^{-1}_{\Gamma_+}$ is one-to-one. Therefore we just need to consider an appropriate deformation of the nonnegative operator $K$.

Let $(e_n)_{n\in \N}$ be an orthornormal basis of $L^2_{\rm even}(V)=\{\phi\in L^2(V)\ |\ \phi(v)=\phi(-v)\}$ so that the basis is made of smooth functions. Then we would like to consider the following deformation
$$
K_\tau :=K+\tau K_w.
$$
where 
$$
w(v',v)=\sum_{n\in \N}{e_n(v)e_n(v')\over \gamma_n}
$$
and the $\gamma_n$ are appropriately chosen positive constants so that the series is absolutely convergent and $w\in C^\infty(V\times V)$.
Here
$$
K_w\phi(x,v)=\int_V w(v', v)\phi(x,v')dv'.
$$
We would like to prove that for small positive $\tau$, then $K_\tau$ is one-to-one on $L^2_{\rm even}(X\times V)$.
We have 
$$
\l K_\tau \phi,\phi\r\ge \tau \sum_{n\in \N}\gamma_n^{-1}\int_{X}|\l e_n,\phi(x,.)\r_{L^2(V)}|^2,
$$
for any $\phi\in L^2_{\rm even}(X\times V)$. Therefore, for $\phi\in L^2_{\rm even}(X\times V)$ and $\tau>0$, $K_\tau \phi=0$ and we have 
$\sum_{n\in \N}\gamma_n^{-1}\int_{X}|\l e_n,\phi(x,.)\r_{L^2(V)}|^2dx=0$, which proves that $\phi=0$. Hence $0$ is not an eigenvalue of the operator $T^{-1}_{\Gamma_+}K_\tau$ restricted to $L^2_{\rm even}(X\times V)$.
This proves the second statement of the Theorem assuming the first one is proved.


For $\gamma\in C(\bar X\times V\times V)$ satisfying \eqref{H1c}, we denote by $K_{\gamma}$ the  bounded self-adjoint operator in $L^2(X\times V)$
$$
K_{\gamma}\phi(x,v)=\int_V \gamma(x,v',v)\phi(x,v')dv',\ (x,v)\in X\times V, \phi\in L^2(X\times V).
$$
For $(\gamma_1,\gamma_2)\in C(\bar X\times V\times V)^2$ both satisfying \eqref{H1c}, the following operator identity holds
$$
K_{\gamma_1}K_{\gamma_2}+K_{\gamma_2}K_{\gamma_1}=K_{\gamma_3},
$$
where $\gamma_3\in C(\bar X\times V\times V)$ also satisfies \eqref{H1c} and is defined by
$$
\gamma_3(x,v',v)=\int_V(\gamma_1(x,v'',v)\gamma_2(x,v',v'')+\gamma_2(x,v'',v)\gamma_1(x,v',v''))dv''.
$$

We also consider  a selfadjoint square root $R$ of the operator $K$, see \cite[Theorem VI.9]{RS}. We recall that
$$
R=\sqrt{M}\sum_{n=0}^\infty c_n\big(I-{K\over M}\big)^n,
$$
for $M$ large enough where the $c_i$'s are all negative numbers for $i>0$ and their series is convergent.  Here the operator $I-{K\over M}$ is a selfadjoint bounded and nonnegative operator and its uniform norm is $1$ by assumption on the function $k\in H$. 
We note that for $(\phi,\psi)\in L^2(X\times V)^2$
$$
\int_V (K\phi)(x,v)\psi(x,v)dv=\int_V\phi(x,v)(K\psi)(x,v)dv \textrm{ a.e. }x\in X.
$$
Hence we obtain that
\begin{eqnarray}
\int_V (R\phi)(x,v)\psi(x,v)dv=\int_V\phi(x,v)(R\psi)(x,v)dv \textrm{ a.e. }x\in X.\label{i302}
\end{eqnarray}

Now we prove the first statement of the Theorem. 
We consider the bounded operator 
$$
A=K_{\gamma}R\textrm{ for some }\gamma\in C(\bar X\times V\times V) \textrm{ satisfying }\eqref{H1c},\label{i300}
$$
and the analytic deformation of $R$ given by
$R_{\tau}=(R+\tau A)$.
Hence
$$
AR+RA^*=K_{\gamma}K+KK_{\gamma}=K_{\gamma_4},\ AA^*=K_{\gamma}KK_{\gamma}=K_{\gamma_5}
$$
where  $\gamma_4$ and $\gamma_5$ belong to $C(\bar X\times V\times V)$, also satisfy \eqref{H1c}, and are defined by
$$
\gamma_4(x,v',v)=\int_V(\gamma(x,v'',v)k(x,v',v")+k(x,v'',v)\gamma(x,v',v''))dv'',
$$
$$
\gamma_5(x,v',v)=\int_{V^2}\gamma(x,v_1,v)k(x,v_2,v_1)\gamma(x,v',v_2)dv_1dv_2,
$$
for $(x,v',v)\in \bar X\times V\times V$.
Therefore  $R_{\tau}R_{\tau}^*$ is the analytic deformation of $K=R^2$ given by 
$$
R_{\tau}R_{\tau}^*=K+\tau(AR+RA^*)+\tau^2(AA^*)=K_{k_\tau},\ k_{\tau}=k+\tau\gamma_4+\tau^2\gamma_5.
$$
By construction the function $k_\tau$ belongs to $H$ for any real $\tau$. 
The operator $T^{-1}_{\Gamma_+}K_{k_\tau}=T^{-1}_{\Gamma_+}R_\tau R_\tau^*$ also defines a $\tau$-analytic deformation of $T^{-1}_{\Gamma_+}K$.

We then apply Albert's approach \cite{A} on the analytic deformation $R_\tau^* T^{-1}_{\Gamma_+}R_\tau$.

Let $\lambda$ be a nonnegative (isolated) eigenvalue of $T^{-1}_{\Gamma_+}K=T^{-1}_{\Gamma_+}R^2$ with multiplicity $h>1$. Then $\lambda$ is also a nonnegative eigenvalue of $RT^{-1}_{\Gamma_+}R$ with multiplicity $h$.
By Rellich's Theorem \cite[Chapter 2, Section 2]{R}, there exist analytic deformations $\lambda_j(\tau)$ and $u_j(\tau)$, $j=1\ldots h$, so that $\lambda_j(0)=\lambda$ and $(u_j(\tau))_{j=1\ldots h}$ is an orthonormal family of eigenvectors of $R_\tau^* T^{-1}_{\Gamma_+}R_\tau$, $R_\tau^* T^{-1}_{\Gamma_+}R_\tau u_j(\tau)
=\lambda_j(\tau)u_j(\tau)$. 
We expand $R_{\tau}^*T^{-1}_{\Gamma_+}R_{\tau}$ and use \eqref{i300} to obtain
$$
R_{\tau}^*T^{-1}_{\Gamma_+}R_{\tau}=RT^{-1}_{\Gamma_+}R+\tau(RK_{\gamma}T^{-1}_{\Gamma_+}R+RT^{-1}_{\Gamma_+}K_{\gamma}R)+\tau^2(A^*T^{-1}_{\Gamma_+}A).
$$
We write
$\lambda_j(\tau)=\lambda+\tau \alpha_j+O(\tau^2)$ and $u_j(\tau)\ =\ u_j+\tau v_j+O(\tau^2)$.
We obtain
$$
\alpha_j  u_j+(\lambda-R T^{-1}_{\Gamma_+} R) v_j-(RK_{\gamma}T^{-1}_{\Gamma_+}R+RT^{-1}_{\Gamma_+}K_{\gamma}R)u_j=0.
$$
We note that 
$
(\lambda-R T^{-1}_{\Gamma_+} R) v_j\in {\rm ker}(\lambda-R T^{-1}_{\Gamma_+} R)^\bot.
$
Hence we obtain
\begin{equation}
\alpha_j \delta_{j,l}
=\l (K_{\gamma}T^{-1}_{\Gamma_+}+T^{-1}_{\Gamma_+}K_{\gamma})Ru_j,Ru_l\r=\l K_{\gamma}T^{-1}_{\Gamma_+}Ru_j,Ru_l\r+\l K_{\gamma}Ru_j,T^{-1}_{\Gamma_+}Ru_l\r.\label{i303}
\end{equation}
We use  that $R$ maps $L^2(X\times V)$ to $L^2_{\rm even}(X\times V)$  and we use selfadjointness of the operator $P_{\rm even}T^{-1}_{\Gamma_+}P_{\rm even}$ in the last identity.
We finally obtained that 
\begin{equation}
\alpha_j \delta_{j,l}
=\int_{X\times V^2}\gamma\big((P_{\rm even}T^{-1}_{\Gamma_+}Ru_j)(x,v')Ru_l(x,v)+Ru_j(x,v')(P_{\rm even}T^{-1}_{\Gamma_+}Ru_l)(x,v)\big) dx dvdv'.\label{i301}
\end{equation}
Above, $\gamma=\gamma(x,v',v)$.
Our goal is to prove that we can choose the function $\gamma$ in the definition of the operator $A$ so that there exists at least a couple of distinct values of $\alpha_j$'s, which proves that for small non zero $\tau$, $\lambda_1(\tau)$ has at most multiplicity $h-1$. This implies the analog of  \cite[Theorem 2]{A} in our setting. Then the conclusion of the proof of genericity of simple non zero eigenvalues of $T^{-1}_{\Gamma_+}K$ follows by reproducing the reasoning in \cite{A}.

We now proceed by contradiction and follow \cite[Lemma]{A}. Assume that the function $\gamma$ is such that all $\alpha_j$'s have the same value $\alpha$. Then it follows from \eqref{i303}  that for any couple $(\tilde u_1, \tilde u_2)$ of orthonormal real valued eigenvectors of ${\rm ker}(RT^{-1}_{\Gamma_+}R-\lambda)$
\begin{equation}
0
=\l (K_{\gamma}T^{-1}_{\Gamma_+}+T^{-1}_{\Gamma_+}K_{\gamma})R \tilde u_1,R \tilde u_2\r=\l K_{\gamma}T^{-1}_{\Gamma_+}R\tilde u_1,R\tilde u_2\r+\l K_{\gamma}R\tilde u_1,T^{-1}_{\Gamma_+}R\tilde u_2\r.\label{i304}
\end{equation}

\begin{lemma} 
\label{lem_smooth}
Under the assumptions of Theorem \ref{thm_simple}, 
$R u$ is continuous on $\bar X\times V$ for any $u\in {\rm ker}(RT^{-1}_{\Gamma_+}R-\lambda)$.
\end{lemma}
The proof of Lemma \ref{lem_smooth} is given in Appendix \ref{sec:five}.

\begin{lemma}
\label{lem_choice}
Let  $(\tilde u_1, \tilde u_2)$ be two orthonormal real valued eigenvectors of 
${\rm ker}(RT^{-1}_{\Gamma_+}R-\lambda)$.
One of the following statements is satisfied: 
\begin{itemize}

\item[i.] Identity \eqref{i304} does not hold when
$$
\gamma(x,v',v)=R\tilde u_1(x,v')R\tilde u_2(x,v)+R\tilde u_1(x,v)R\tilde u_2(x,v'),\  (x,v',v)\in X\times V\times V,
$$

\item[ii.]Identity \eqref{i304} does not hold when
$$
\gamma(x,v',v)=\  R\tilde u_1(x,v')R\tilde u_1(x,v),\ (x,v',v)\in X\times V\times V,
$$
and the couple $(\tilde u_1, \tilde u_2)$ is replaced by the couple $({\tilde u_1-\tilde u_2\over \sqrt{2}},{\tilde u_1+\tilde u_2\over \sqrt{2}})$ in \eqref{i304}. 
\end{itemize}
\end{lemma}
The proof of Lemma \ref{lem_choice} is given in Appendix \ref{sec:five}.
\end{proof}

\begin{remark}
The kernel of $T^{-1}_{\Gamma_+}K$ is infinite dimensional with the assumptions of Theorem \ref{thm_simple}.
For $k(x,v',v)=k(x,v,-v')$ and $u\in L^2_{\rm odd}(X\times V)$, we indeed have
$Ku=0,\ T^{-1}_{\Gamma_+}Ku=0.$
\end{remark}

\subsection{High-dimensional obstruction to controllability.}

We now show that in specific, highly symmetric, situations, the kernel of the albedo operator may have an arbitrarily large dimension. For the rest of this section, we consider the case $X=B(0,1)$ the unit Euclidean Ball of dimension $d$ centered at $0$, $\sigma$ is a constant and the function $k$ is identically equal to 1.

The operator of interest is $KT_{\Gamma_+}^{-1}$ on $L^2(X\times V)$. We restrict $KT_{\Gamma_+}^{-1}$ to $L^2_{\rm odd}(X)=\{g\in L^2(X)\ |\ g(-.)=-g\}$, and we denote by $F$ this restriction. 
Then 
$$
Ff(x)=-\int_X{e^{\sigma|x-y|}\over |x-y|^{d-1}}f(y),\ \textrm{ a. e. }x\in X,\ f\in L^2_{\rm odd}(X).
$$
The main properties of $F$ are invariance by complex conjugation and by orthogonal transformations of $\R^d$. 
The operator $F$ is a compact self-adjoint operator on $L^2_{\rm odd}(X)$. 

Consider the space $H_l(V)$ of spherical harmonics of degree $l$, $l\in \N$. The spherical harmonics of degree $l$ are the traces of harmonic homogeneous polynomials in $\R^d$ of degree $l$. We denote by $N_l$ the dimension of $H_l(V)$ and recall that $N_l=\big({d+l-1\atop d-1}\big)-\big({d+l-3\atop d-1}\big)$ for $l\ge 2$, $N_1=d$, $N_0=1$ (see \cite{BS}). We also denote by $(Y_{1,l},\ldots,Y_{N_l,l})$ an orthonormal basis for  $H_l(V)$ endowed with the $L^2$-product on $V$.

Let $l$ be odd and $L_{k,l}$ be the subspace of $L^2_{\rm odd}(X)$ consisting of square summable functions $f$ written in spherical coordinates as
$$
f(r\omega)=g(r)Y_{k,l}(\omega)\textrm{ a.e. }(r,\omega)\in (0,1)\times V,
$$
for some function $g\in L^2([0,1], r^{d-1}dr)$. 
The space $L_{k,l}$ can be identified to  $L^2([0,1], r^{d-1}dr)$.

\begin{lemma}
\label{lem_spec}
The operator $F$ leaves stable the subspace $L_{k,l}$. For any 
$g\in L^2([0,1], r^{d-1}dr),$
$$
F(gY_{k,l})=F_l(g)Y_{k,l},
$$
where $F_l$ is a selfadjoint compact (Hilbert-Schmidt) operator in $L^2([0,1], r^{d-1}dr)$.

\end{lemma}
The proof of the Lemma is given in Appendix \ref{sec:five}.
Note that the orthogonal sum
$$
\oplus_{l\in  \N,\ l\textrm{ odd },\ k=1\ldots N_l} L_{k,l}
$$
is dense in $L^2_{\rm odd}(X)$. Hence Lemma \ref{lem_spec} provides a spectral decomposition of the operator $F$ once each $F_l$'s are diagonalized. 
The eigenvalues of $F$ have a multiplicity greater than or equal to $d$ since $N_l\ge d$ for any $l\ge 1$. There does not exist a (uniform) bound on the multiplicity of eigenvalues for the operator $F$.

\begin{lemma}
\label{lem_specbd}
Let $l\in \N$, $l$ odd. When the constant $\sigma$ is large enough then the operator $F_l$ in $L^2([0,1], r^{d-1}dr)$ has a positive eigenvalue $\mu$ greater than 1 and $\mu$ is also an eigenvalue for the operator $F$ with multiplicity at least $N_l$.
\end{lemma}
The proof of the Lemma is given in Appendix \ref{sec:five}.
We observe that the dimension $N_l$ grows to $+\infty$ as $l\to +\infty$ when the spatial dimension $d$ is greater or equal to 3. As a corollary of Lemma \ref{lem_specbd}, we deduce the existence of kernels of the albedo operator with arbitrary large dimension when $d\ge 3$.

\section*{Acknowledgments}

GB's research was partially supported by the National Science Foundation, Grants DMS-1908736 and EFMA-1641100 and by the Office of Naval Research, Grant N00014-17-1-2096.

\appendix

\section{Proof of Lemma \ref{lem_1p}}
\label{sec:app1p}

We start with  the proof of the statements for the operator $T_{\mC}^{-1}K$. Let $\phi$ be a bounded mesurable function on $\pa X\times V$ and $(x,v)\in X\times V$. It follows from the definition of $T_{\mC}^{-1}$  and $K$:
\begin{eqnarray*}
&&T_{\mC}^{-1}K L_\mC\phi(x,v)\\
&=&\chi_{\mC_-}(x-\tau_-(x,v)v,v)\int_{\S}\int_0^{\tau_-(x,v)}\!\!\!\!\!\!\!\!\! E_-(x,v,t)\int_V k(x-tv,v_-,v)L_\mC\phi(x-tv,v_-)dv_-dt\\
&&-\chi_{\mC_+}(x+\tau_+(x,v)v,v)\int_{\S}\int_0^{\tau_+(x,v)}\!\!\!\!\!\!\!\!\! E_+(x,v,t)\int_V k(x+tv,v_-,v)L_\mC \phi(x+tv,v_-)dv_- dt.
\end{eqnarray*}
Hence 
\begin{eqnarray*}
\big|T_{\mC}^{-1}K\phi(x,v)\big|
\le \|k\|_\infty\chi_{\mC_-}(x-\tau_-(x,v)v,v)\int_{\S}\int_0^{\tau_-(x,v)}\int_V |L_\mC \phi|(x-tv,v_-)dv_-dt\\
+\|k\|_\infty e^{\|\sigma\|_\infty  \tau_+(x,v)}\chi_{\mC_+}(x+\tau_+(x,v)v,v)\int_{\S}\int_0^{\tau_+(x,v)}\int_V |L_\mC \phi|(x+tv,v_-)dv_- dt.
\end{eqnarray*}
From the definition  of $L_\mC$, we have
\begin{eqnarray*}
|L_{\mC} \phi(y,v_-)|&\le &
\left\lbrace 
\begin{array}{l}
|\phi|(y-\tau_-(y,v_-)v_-,v_-),\ (y-\tau_-(y,v_-)v_-,v_-)\in \mC_-,\\
e^{\|\sigma\|_\infty \tau_+(y,v)}|\phi|(y+\tau_+(y,v_-)v_-,v_-),\ (y+\tau_+(y,v_-)v_-,v_-)\in \mC_+,
\end{array}
\right.
\end{eqnarray*}
for a.e. $(y,v_-)\in X\times V$. Now set
$$
I_\pm^-(\phi, x,v)=\int_{\S}\int_0^{\tau_\pm(x,v)}\int_V |\phi|(x\pm tv-\tau_-(x\pm tv,v_-)v_-,v_-)dv_-dt,
$$
$$
I_\pm^+(\phi, x,v)=\int_{\S}\int_0^{\tau_\pm(x,v)}\int_V |\phi|(x\pm tv+\tau_+(x\pm tv,v_-)v_-,v_-)dv_-dt.
$$

We obtain 
\begin{eqnarray*}
\big|T_{\mC}^{-1}K\phi(x,v)\big|&\le& \|k\|_\infty\big[\chi_{\mC_-}(x-\tau_-(x,v)v,v)(I_-^-+e^{\|\sigma\|_\infty  \|\tau_+\|_\infty}I_-^+)\\
&&+\chi_{\mC_+}(x+\tau_+(x,v)v,v)e^{\|\sigma\|_\infty  \|\tau_+\|_\infty}(I_+^-+e^{\|\sigma\|_\infty  \|\tau_+\|_\infty}I_+^+)\Big].
\end{eqnarray*}
We omitted the arguments $(\phi, x,v)$ in front of $I_-^\pm$ and $I_+^\pm$ for clarity. We first find a bound for $I_-^-$. The study of $I_\pm^+$ and $I_+^-$ is similar. We make the change of variables
$$
\gamma_{v_-,v,x}(t)=x-tv-\tau_-(x-tv,v_-)v_-,\ dt= \big|{ \nu_{P_{v_-,v,x}}(\gamma_{v_-,v,x}(t))\cdot v_- \over v_-\cdot v^\bot}\big|d\gamma_{v_-,v,x}(t),
$$ 
where $P_{v_-,v,x}$ denotes the 2-dimensional plane passing through $x$ with directions $(v,v_-)$ for a.e. $v_-$ and where $\gamma_{v_-,v,x}$ is the intersection curve between $\pa X$ and the plane $P_{v_-,v,x}$ and where $ \nu_{P_{v_-,v,x}}(z)$ is the outward unit normal at the boundary point $z$ of the 2-dimensional domain $P_{v_-,v,x}\cap X$. Here $v^\bot$ denotes one unit vector orthogonal to $v$ in the vector plane spanned by $v$ and $v_-$. We set 
$$
\pa X_{-,v_-,v,x}=\{ x- tv-\tau_-(x - tv,v_-)v_-\ |\ t\in (0,\tau_-(x,v))\}\subset \pa X\cap P_{v_-,v,x}.
$$

We obtain 
\begin{equation}
I_-^-(\phi, x,v)
=\int_{\S^{d-1}}\int_{\gamma_{v_-,v,x}}\!\!\!\!\!\!\!\!\!\!\!\!\!{\chi_{\pa X_{-,v_-,v,x}}(\gamma_{v_-,v,x})|\phi|(\gamma_{v_-,v,x},v_-) |\nu_{P_{v_-,v,x}}(\gamma_{v_-,v,x})\cdot v_- |d\gamma_{v_-,v,x} dv_-\over |v_-\cdot v^\bot|}.
\label{A0}
\end{equation}

\paragraph{Dimension $d=2$.}
In that case, $\pa X$ is 1-dimensional and $X\cap P_{v_-,v,x}=X$ for a.e. $(v,v_-)$. Hence \eqref{A0} reduces to 
$$
I_-^-(\phi, x,v)
=\int_{\mC_-}{\chi_{\pa X_{-,v_-,v,x}}(x_-)\over |v_-\cdot v^\bot|}|\phi(x_-,v_-)|d\xi(x_-,v_-).
$$

Let $p\in (1,2-r)$ for some small positive number $r$ and denote its conjugate by $q$, $p^{-1}+q^{-1}=1$. We are going to prove that 
\begin{equation}
\|I_-^-(\phi, .,.)\|_{L^p(X\times V)}\le C_p\|\phi\|_{L^1(\mC, d\xi)},\ \|I_-^-(\phi, .,.)\|_{L^p(\Gamma\b\mC,d\xi) }\le C_p\|\phi\|_{L^1(\mC, d\xi)}.\label{A3}
\end{equation}
Similar estimates are derived for $I_\pm^+$ and $I_-^+$ which prove our statements for the dimension $d=2$.
We use the estimates 
\begin{eqnarray}
\sup_{x_-\in \pa X}\int_{\pa X}\chi_{\pa X_{-,v_-,v,x}}(x_-)|\nu(x)\cdot v|d\mu(x) &\le& C_X|v\cdot v_-^\bot|,\nonumber\\
\sup_{x_-\in \pa X}\int_X\chi_{\pa X_{-,v_-,v,x}}(x_-)dx &\le& C_X|v\cdot v_-^\bot|\label{A2}
\end{eqnarray}
for a constant $C_X$ and for a.e. $(v,v_-)\in V^2$, and
\begin{equation}
\sup_{v_-\in V}\int_V|v_-\cdot v^\bot|^{1-p}dv=\int_0^{2\pi}|\sin(\theta)|^{1-p}d\theta<\infty.\label{A4}
\end{equation}
A proof of \eqref{A2} is postponed to the end of this section.

Let us start with the first estimate in \eqref{A3}. We first prove that
\begin{equation}
\int_{X\times V}\int_{\mC_-}{\chi_{\pa X_{-,v_-,v,x}}(x_-)\over |v_-\cdot v^\bot|^p}|\phi(x_-,v_-)|d\xi(x_-,v_-) dx dv\le \tilde C_p\|\phi\|_{L^1(\mC, d\xi)}^p.\label{A5}
\end{equation}
Indeed  we use  Tonelli's theorem on nonnegative measurable functions and the second estimate in \eqref{A2} and \eqref{A4} to obtain 
\begin{eqnarray*}
&&\int_{X\times V}\int_{\mC_-}{\chi_{\pa X_{-,v_-,v,x}}(x_-)\over |v_-\cdot v^\bot|^p}|\phi(x_-,v_-)|d\xi(x_-,v_-) dx dv\\
&\le &C_X\int_{\mC_-} \int_V{dv\over |v_-\cdot v^\bot|^{p-1}}|\phi(x_-,v_-)|d\xi(x_-,v_-)\le C_X \int_0^{2\pi}|\sin(\theta)|^{1-p}d\theta \|\phi\|_{L^1(\mC,d\xi)}.
\end{eqnarray*}
Hence 
$$
\tilde C_p=C_X \int_0^{2\pi}|\sin(\theta)|^{1-p}d\theta<\infty
$$ 
and the integrand in \eqref{A5}  belongs to $L^1$ and we can apply H\"older estimate below
\begin{eqnarray*}
&&\|I_-^-(\phi, .,.)\|_{L^p(X\times V)}^p=\int_{X\times V}\Big|\int_{\mC_-}{\chi_{\pa X_{-,v_-,v,x}}(x_-)\over |v_-\cdot v^\bot|}|\phi(x_-,v_-)|d\xi(x_-,v_-)\Big|^p dx dv\\
&\le &\int_{X\times V}\Big[\int_{\mC_-}{\chi_{\pa X_{-,v_-,v,x}}(x_-)\over |v_-\cdot v^\bot|^p}|\phi(x_-,v_-)|d\xi(x_-,v_-) 
\\
&&\times\big(\int_{\mC_-}\chi_{\pa X_{-,v_-,v,x}}(x_-)|\phi(x_-,v_-)|d\xi(x_-,v_-)\big)^{p\over q}\Big]dxdv\\
&\le &\int_{X\times V}\int_{\mC_-}{\chi_{\pa X_{-,v_-,v,x}}(x_-)\over |v_-\cdot v^\bot|^p}|\phi(x_-,v_-)|d\xi(x_-,v_-)dx dv
 \|\phi\|_{L^1(\mC,d\xi)}^{p\over q}
\le\tilde C_p \|\phi\|_{L^1(\mC,d\xi)}^p,
\end{eqnarray*}
which proves the first estimate. The second estimate in \eqref{A3} is proved in the same way (use the first estimate in \eqref{A2} instead) and we start from the estimate
$$
\|I_-^-(\phi, .,.)\|_{L^p(\Gamma\b\mC,d\xi)}^p\le\int_{\pa X\times V}\Big|\int_{\mC_-}{\chi_{\pa X_{-,v_-,v,x}}(x_-)\over |v_-\cdot v^\bot|}|\phi(x_-,v_-)|d\xi(x_-,v_-)\Big|^p d\xi(x,v).
$$

\paragraph{Dimension $d\ge 3$.}
Let $1\le p< d$. We first prove that $I_-^-(\phi,.,.)\in L^p(X\times V)$ and 
\begin{equation}
\|I_-^-(\phi,.,.)\|_{L^p(X\times V)}\le C_p\|\phi\|_{L_*(\Gamma)},\label{A6}
\end{equation}
for some constant $C_p$ depending only on $p$ and $X$. This reduces to proving \eqref{A6} when $\phi$ is replaced by $|\phi|$.
By \eqref{A0} and by H\"older's inequality
\begin{eqnarray*}
&&\|I_-^-(|\phi|,.,.)\|_{L^p(X\times V)}^p\\
&=&\int_{X\times V}\Big|\int_V\int_{\gamma_{v_-,v,x}}{\chi_{\pa X_{-,v_-,v,x}}(\gamma)| \nu_{P_{v_-,v,x}}(\gamma)\cdot v_- ||\phi|(\gamma,v_-)d\gamma dv_-\over |v_-\cdot v^\bot|}\Big|^p dx dv\\
&\le& \int_{X\times V} \Big[\int_V\int_{\gamma_{v_-,v,x}}{\chi_{\pa X_{-,v_-,v,x}}(\gamma)| \nu_{P_{v_-,v,x}}(\gamma)\cdot v_- |\over |v_-\cdot v^\bot|^p}|\phi|(\gamma,v_-)d\gamma dv_- \\
&&\times  \big(\int_V\int_{\gamma_{v_-,v,x}}| \nu_{P_{v_-,v,x}}(\gamma)\cdot v_- ||\phi|(\gamma,v_-)d\gamma dv_-\big)^{p\over q}\Big]dx dv.
\end{eqnarray*}
We use the notation $\gamma$ for $\gamma_{v_-,v,x}$ in the integrands above and below.
By definition of the norm $\|\phi\|_{L_*(\Gamma)}$ we have
$$
\int_V\int_{\gamma_{v_-,v,x}}| \nu_{P_{v_-,v,x}}(\gamma)\cdot v_- ||\phi|(\gamma,v_-)d\gamma dv_-\le \|\phi\|_{L_*(\Gamma)},
$$
and we have 
\begin{eqnarray*}
&&\|I_-^-(|\phi|,.,.)\|_{L^p(X\times V)}^p\\
&\le& \|\phi\|_{L_*(\Gamma)}^{p\over q}\int_{X\times V} \Big[\int_V\int_{\gamma_{v_-,v,x}}{\chi_{\pa X_{-,v_-,v,x}}(\gamma)| \nu_{P_{v_-,v,x}}(\gamma)\cdot v_- |\over |v_-\cdot v^\bot|^p}|\phi|(\gamma,v_-)d\gamma dv_-\Big]dx dv.
\end{eqnarray*}

We introduce the spherical coordinates for the $v$-variable: 
$$
v(\theta,\omega)=\sin(\theta)v_-+\cos(\theta)\omega,\ (\theta,\omega)\in (-{\pi\over 2},{\pi\over 2})\times\S^{d-2}_{v_-}
$$ 
where 
$
 \S^{n-2}_{v_-}=\{\omega\in \S^{d-1}\ |\ v_-\cdot \omega=0\}.
$
We recall that $\gamma_{v_-,v,x}$ denotes the curves that intersect $\pa X$  and the 2-dimensional plane $P_{v_-,v,x}$ passing through $x$ with directions $(v,v_-)$ (or equivalently $(\omega,v)$ for a.e. $v$). Hence 
$$
\gamma_{v_-,v,x}= \gamma_{v_-,\omega,x},\ P_{v_-,v,x}=P_{v_-,\omega,x}\textrm{ for a.e. }(v,v_-)\in (\S^{d-1})^2.
$$
In addition
$$
|v_-\cdot v^\bot|= \cos(\theta)\textrm{ in the vector plane spanned by }\omega\textrm{ and }v_-,
$$
and we obtain 
\begin{eqnarray*}
\|I_-^-(|\phi|,.,.)\|_{L^p(X\times V)}^p
&\le&\|\phi\|_{L_*(\Gamma)}^{p\over q} \int_V\int_{X\times \S^{d-2}_{v_-}}\int_{-{\pi\over 2}}^{\pi\over 2}(\cos\theta)^{d-2-p}\int_{\gamma_{v_-,\omega,x}}\chi_{\pa X_{-,x,v(\theta,\omega),v_-}}(\gamma) \\
&&\times|\nu_{P_{v_-,\omega,x}}(\gamma)\cdot v_- ||\phi|(\gamma,v_-)d\gamma d\theta d\omega dx dv_-.
\end{eqnarray*}

We denote $H_{v_-,\omega}$ the $(d-2)$-dimensional space orthogonal to $v_-$ and $\omega$. We introduce the change of variables 
$$
x=y+tv_-+s\omega,\ y\in H_{v_-,\omega},\ |y|\le r_1,\ (s,t)\in (-r_1,r_1)^2,\ dx=dt ds dy.
$$
Hence $P_{v_-,\omega,x}=P_{v_-,\omega,y}$ and $\gamma_{v_-,\omega,x}=\gamma_{v_-,\omega,y}$ and we obtain
\begin{eqnarray*}
&&\|I_-^-(|\phi|,.,.)\|_{L^p(X\times V)}^p\\
&\le& \|\phi\|_{L_*(\Gamma)}^{p\over q}\int_V\int_{ \S^{d-2}_{v_-}}\int_{y \in H_{v_-,\omega}\atop |y|\le r_1}\int_{-{\pi\over 2}}^{\pi\over 2}(\cos\theta)^{d-2-p}\int_{(-r_1,r_1)^2}\chi_X(y+tv_-+s\omega)\\
&&\times
\int_{\gamma_{v_-,\omega,y}}\!\!\!\!\!\!\!\!\!\!\chi_{\pa X_{-,y+tv_-+s\omega,v(\theta,\omega),v_-}}(\gamma) 
|\nu_{P_{v_-,\omega,y}}(\gamma)\cdot v_- ||\phi|(\gamma,v_-)d\gamma  dt ds d\theta d\omega  dy dv_-.
\end{eqnarray*}
We use the following estimate (see above for the $2$-dimensional case)
\begin{eqnarray}
&&\sup_{y\in H_{v_-,\omega}}\int_{(-r_1,r_1)^2}\chi_X(y+tv_-+s\omega)\chi_{\pa X_{-,y+tv_-+s\omega,v(\theta,\omega),v_-}}(\gamma_{v_-,\omega,y})ds dt
\le C_X|v_-\cdot v^\bot|\nonumber\\
&=&C_X\cos(\theta)\label{A200}
\end{eqnarray}
for a constant $C_X$ which depends only on $X$ and for a.e. $v_-,v$, and we note that 
$$
C_p=\int_{-{\pi\over 2}}^{\pi\over 2}(\cos\theta)^{d-1-p}d\theta<\infty \textrm{ for }p<d,
$$
and we obtain that
\begin{eqnarray*}
&&\|I_-^-(|\phi|,.,.)\|_{L^p(X\times V)}^p\\
&&\hskip-0.5cm \le \|\phi\|_{L_*(\Gamma)}^{p\over q} C_X \int_V\int_{ \S^{d-2}_{v_-}}\int_{y \in H_{v_-,\omega}\atop |y|\le r_1}\int_{-{\pi\over 2}}^{\pi\over 2}(\cos\theta)^{d-1-p}\\
&&\times\int_{\gamma_{v_-,\omega,y}} |\nu_{P_{v_-,\omega,y}}(\gamma)\cdot v_- |
|\phi|(\gamma,v_-)d\gamma  d\theta d\omega  dy dv_-\\
&&\hskip-0.5cm \le \|\phi\|_{L_*(\Gamma)}^{p\over q} C_XC_p \int_V\int_{ \S^{d-2}_{v_-}}\int_{y \in H_{v_-,\omega}\atop |y|\le r_1}\int_{\gamma_{v_-,\omega,y}} |\nu_{P_{v_-,\omega,y}}(\gamma)\cdot v_- |
|\phi|(\gamma,v_-)d\gamma  d\omega  dy dv_-\\
&&\hskip-0.5cm\le C_XC_p |V||\S^{d-2}| |B_{d-2}(0,r_1)| \|\phi\|_{L_*(\Gamma)}^{{p\over q}+1}= C_XC_p |V||\S^{d-2}| |B_{d-2}(0,r_1)| \|\phi\|_{L_*(\Gamma)}^p.
\end{eqnarray*}
Here $|V|$, $|\S^{d-2}|$ $|B_{d-2}(0,r_1)|$ denote the $(d-1)$-dimensional volume of $V$, the $(d-2)$-dimensional volume of $\S^{d-2}$ and the $(d-2)$-dimensional volume of the Euclidean Ball of center $0$ and radius $r_1$ in $\R^{d-2}$ respectively. This concludes the proof on boundedness of $T_{\mC}^{-1}KL_\mC$ from $L_*(\Gamma)$ to $L^p(X\times V)$.

Similarly we prove that $I_-^-(\phi,.,.)\in L^p(\Gamma, d\xi)$ and 
\begin{equation}
\|I_-^-(\phi,.,.)\|_{L^p(\Gamma,d\xi)}\le C_p\|\phi\|_{L_*(\Gamma)},\label{A6b}
\end{equation}
Indeed we repeat the first lines of the proof above to obtain
\begin{eqnarray*}
&&\hskip-0.8cm\|I_-^-(|\phi|,.,.)\|_{L^p(\Gamma,d\xi)}^p
\le \|\phi\|_{L_*(\Gamma)}^{p\over q}\int_V\int_{\pa X\times \S^{d-2}_{v_-}}\int_{-{\pi\over 2}}^{\pi\over 2}(\cos\theta)^{d-2-p}\int_{\gamma_{v_-,\omega,x}}\!\!\!\!\!\chi_{\pa X_{-,x,v(\theta,\omega),v_-}}(\gamma_{v_-,\omega,x}) \\
&&\times|\nu_{P_{v_-,\omega,x}}(\gamma_{v_-,\omega,x})\cdot v_- ||\phi|(\gamma_{v_-,\omega,x},v_-)d\gamma_{v_-,\omega,x} |\nu(x)\cdot v(\theta,\omega)| d\theta d\omega d\mu(x) dv_-.
\end{eqnarray*}
Then recalling that the variable $x$ lives on the boundary $\pa B(0,r_1)\cup \pa B(0,r_2)$, we introduce the change of variables 
$$
x(r,y,\theta')=y+\sqrt{r^2-|y|^2}(\cos(\theta')\omega+\sin(\theta')v_-),\ y\in H_{v_-,\omega},\ 0\le |y|< r_1,\ \theta'\in (0,2\pi),
$$
$$
dx=r d\theta' dy.
$$
and  $r=r_1\textrm{ when }r_2<|y|<r_1$ while $\textrm{either  }r=r_1 \textrm{ or }r=r_2 \textrm{ when }0\le |y|<r_2$.
We omit the dependance of $x(r,y,\theta)$ with respect to $v_-$ and $\omega$. Note that $P_{v_-,\omega,x}=P_{v_-,\omega,y}$ and $\gamma_{v_-,\omega,x}=\gamma_{v_-,\omega,y}$. We obtain 
\begin{eqnarray*}
&&\|I_-^-(|\phi|,.,.)\|_{L^p(\Gamma,d\xi)}^p
\le\|\phi\|_{L_*(\Gamma)}^{p\over q} \int_V\sum_{i=1}^2 r_i\int_{\S^{d-2}_{v_-}}\int_{y\in H_{v_-,\omega}\atop|y|<r_i}\int_0^{2\pi}\int_{-{\pi\over 2}}^{\pi\over 2}(\cos\theta)^{d-2-p}\\&&\times \int_{\gamma_{v_-,\omega,y}}\chi_{\pa X_{-,x,v(\theta,\omega),v_-}}(\gamma_{v_-,\omega,y})|\nu(x)\cdot v(\theta,\omega)| \\
&&\times|\nu_{P_{v_-,\omega,y}}(\gamma_{v_-,\omega,y})\cdot v_- ||\phi|(\gamma_{v_-,\omega,y},v_-)d\gamma_{v_-,\omega,y}\Big|_{x=x(r_i,y,\theta')} d\theta d\omega d\theta' dy dv_-.
\end{eqnarray*}
We also use the following estimates (see the end of this section for their proof)
\begin{eqnarray}
&&\sup_{y\in H_{v_-,\omega},\ 0\le |y|<r_i}\int_{(0,2\pi)}\chi_{\pa X_{-,x,v(\theta,\omega),v_-}}(\gamma_{v_-,\omega,y})|\nu(x)\cdot v(\theta,\omega)|\Big|_{x=x(r_i,y,\theta')}d\theta'\nonumber\\
&&\le C_X|v_-\cdot v^\bot|=C_X\cos(\theta),\ i=1,2.\label{A201}
\end{eqnarray}
Hence we obtain as above 
\begin{eqnarray*}
\|I_-^-(|\phi|,.,.)\|_{L^p(\Gamma,d\xi)}^p
\le \|\phi\|_{L_*(\Gamma)}^{p\over q}C_X\int_V\Big(\sum_{i=1}^2 r_i\int_{\S^{d-2}_{v_-}}\int_{y\in H_{v_-,\omega}\atop|y|<r_i}\int_{-{\pi\over 2}}^{\pi\over 2}(\cos\theta)^{d-1-p} d\theta\\
\times\int_{\gamma_{v_-,\omega,y}}|\nu_{P_{v_-,\omega,y}}(\gamma_{v_-,\omega,y})\cdot v_- ||\phi|(\gamma_{v_-,\omega,y},v_-)d\gamma_{v_-,\omega,y}  d\omega dy dv_-\\
\le 2r_1C_X C_p|V||\S^{d-2}||B_{d-2}(0,r_1)| \|\phi\|_{L_*(\Gamma)}^p,
\end{eqnarray*}
which conclude the proof on boundedness of $T_{\mC}^{-1}KL_\mC$ from $L_*(\Gamma)$ to $L^p(\Gamma\b\mC,d\xi)$.

\paragraph{Boundedness of $\Big[(I-T_{\mC}^{-1}K)^{-1}(T_{\mC}^{-1}K)L_{\mC}\Big]_{|\Gamma\b\mC}$ under condition \eqref{i20a}.}
Under condition \eqref{i20a} we write 
$$
(I-T_{\mC}^{-1}K)^{-1}(T_{\mC}^{-1}K)L_{\mC}\phi=(T_{\mC}^{-1}K)L_{\mC}\phi+(I-T_{\mC}^{-1}K)^{-1}(T_{\mC}^{-1}K)^2L_{\mC}\phi,
$$
for any $\phi\in L^\infty(\mC,d\xi)$. 
Let us consider dimension $d=2$. The higher dimensional case $d\ge 3$ is handled similarly: replace $L^1(\mC,d\xi)$ by $L_*(\Gamma(X))$.

We have already proved that $T_{\mC}^{-1}KL_{\mC}$ defined a bounded operator from $L^1(\mC,d\xi)$ to $L^p(\Gamma\b\mC,d\xi)$.
It remains to prove that the second term also defines a bounded operator from $L^1(\mC,d\xi)$ to $L^p(\Gamma\b\mC,d\xi)$.

We have proved above that $T_{\mC}^{-1}KL_{\mC}$ defined a bounded operator from $L^1(\mC,d\xi)$ to $L^p(X\times V)$. Then by Lemmas \ref{lem_bound1} and \ref{lem_K},
$(T_{\mC}^{-1}K)^2L_{\mC}$ defines a bounded operator from $L^1(\mC,d\xi)$ to $W^p(X\times V)$. Hence we obtain that
$(I-T_\mC^{-1}K)(T_{\mC}^{-1}K)^2L_{\mC}$ is bounded from $L^1(\mC,d\xi)$ to $W^p(X\times V)$ and Lemma \ref{lem:lift} concludes the proof.

\paragraph{Boundedness of $\big[L_{\mC}]_{|\Gamma}$ on $L_*(\Gamma)$ in dimension $d\ge 3$.} Let us recall first a formula for a 2-dimensional $C^1$ bounded domain $X$, a velocity space $V$ and a boundary set $\mC$: For $f\in L^\infty(\mC)$
\begin{equation}
\int_{\pa X\b\mC_v} [J_\mC f]_{|\Gamma\b\mC}(x,v)|\nu(x)\cdot v|d\mu(x)= \int_{\mC_v} f(x,v)|\nu(x)\cdot v|d\mu(x)\label{A20}
\end{equation}
for a.e. $v\in \S^1$ where $\mC_v=\{x\in \pa X\ |\ (x,v)\in \mC\}$. The proof of \eqref{A20} follows from the identity
\begin{eqnarray*}
\int_V g(v)\big(\int_{\pa X\b\mC_v} [J_\mC f]_{|\Gamma\b\mC}(x,v)|\nu(x)\cdot v|d\mu(x)\big)dv=\int_{\Gamma\b\mC}[J_\mC f]_{|\Gamma\b\mC}(gf)d\xi=\int_{\mC}gfd\xi\\
=\int_Vg(v)\big(\int_{\mC_v}  f(x,v)|\nu(x)\cdot v|d\mu(x)\big)dv,
\end{eqnarray*}
for any $g\in L^1(V)$. 
Now let us return to dimension $d\ge 3$.
Let $\phi \in L^\infty(\pa X\times V)$ and let $(v_-,\omega,y)\in {\mathcal M}$ so that $\gamma_{v_-,\omega,y}$ is either a circle or the union of 2 circles and $X\cap P_{v_-,\omega,y} $ is either a disk or a planar annulus depending on the value of $y$ (either $0\le |y|< r_2$ or $r_2<|y|<r_1$). We denote by $\S_{\omega,v_-}$ the unit sphere in the vector plane spanned by $(\omega,v_-)$ and  $\Gamma_{v_-,\omega,y}=\gamma_{v_-,\omega,y}\times  \S_{\omega,v_-}$. The subsets $C_{\pm,v_-,\omega,y}=\Gamma_{v_-,\omega,y}\cap C_\pm $ are measurable subsets of $\Gamma_{v_-,\omega,y}$ and their union  $\mC_{v_-,\omega,y}=\Gamma_{v_-,\omega,y}\cap C$ has a non zero measure for a.e. $(v_-,\omega,y)\in {\mathcal M}$. We also denote by $\phi_{|\Gamma_{v_-,\omega,y}}$ the measurable bounded function obtained by restricting $\phi$ on $\Gamma_{v_-,\omega,y}$  for a.e. $(v_-,\omega,y)\in {\mathcal M}$.
In addition $L_{\mC_{v_-,\omega,y}}(\phi_{|\Gamma_{v_-,\omega,y}})=L_{\mC}\phi$ for a.e. $(v_-,\omega,y)\in {\mathcal M}$  where $B_{\mC_{v_-,\omega,y}}$ denotes any operator ``$B_{\mC}$" of Section \ref{sec:forward} for ``$(X,V,\mC)=(X\cap P_{v_-,\omega,y},\S_{\omega,v_-}, \mC_{v_-,\omega,y})$". In addition
\begin{equation}
\big[L_{\mC_{v_-,\omega,y}}\big]_{|\Gamma_{v_-,\omega,y}}(\phi_{|\Gamma_{v_-,\omega,y}})=\Big[\big[L_{\mC}\big]_{|\Gamma}(\phi)\Big]_{|\Gamma_{v_-,\omega,y}}\label{A22}
\end{equation}
for a.e. $(v_-,\omega,y)\in {\mathcal M}$.

Then, the above left-hand side defines a bounded function (see Lemma \ref{lem_bound1} and $\phi\in L^\infty(\pa X\times V)$) and we apply \eqref{A20} to obtain
\begin{eqnarray}
&&\int_{\gamma_{v_-,\omega,y}}|\big[L_{\mC_{v_-,\omega,y}}\big]_{|\Gamma_{v_-,\omega,y}\b\mC_{v_-,\omega,y}}(\phi_{|\Gamma_{v_-,\omega,y}})||\nu_{P_{v_-,\omega,y}}(\gamma_{v_-,\omega,y})\cdot v_-|d\gamma_{v_-,\omega,y}\nonumber\\
&\le &e^{\|\tau\|_\infty\|\sigma\|_\infty}\int_{\gamma_{v_-,\omega,y}}\Big|\big[J_{\mC_{v_-,\omega,y}}\big]_{|\Gamma_{v_-,\omega,y}\b\mC_{v_-,\omega,y}}(\phi_{|\Gamma_{v_-,\omega,y}})\Big||\nu_{P_{v_-,\omega,y}}(\gamma_{v_-,\omega,y})\cdot v_-|d\gamma_{v_-,\omega,y}\nonumber\\
&\le&e^{\|\tau\|_\infty\|\sigma\|_\infty}\int_{\gamma_{v_-,\omega,y}}|\phi_{|\mC_{v_-,\omega,y}}| |\nu_{P_{v_-,\omega,y}}(\gamma_{v_-,\omega,y})\cdot v_-|d\gamma_{v_-,\omega,y},\nonumber
\end{eqnarray}
for a.e. $(v_-,\omega,y)\in {\mathcal M}$. Since $
\big[L_{\mC_{v_-,\omega,y}}\big]_{|\mC_{v_-,\omega,y}}f=f$
for any $f\in L^\infty(\mC_{v_-,\omega,y})$, we finally obtain 
\begin{eqnarray*}
&&\int_{\gamma_{v_-,\omega,y}}|\big[L_{\mC_{v_-,\omega,y}}\big]_{|\Gamma_{v_-,\omega,y}}(\phi_{|\Gamma_{v_-,\omega,y}})||\nu_{P_{v_-,\omega,y}}(\gamma_{v_-,\omega,y})\cdot v_-|d\gamma_{v_-,\omega,y}\nonumber\\
&\le&(1+e^{\|\tau\|_\infty\|\sigma\|_\infty})\int_{\gamma_{v_-,\omega,y}}|\phi| |\nu_{P_{v_-,\omega,y}}(\gamma_{v_-,\omega,y})\cdot v_-|d\gamma_{v_-,\omega,y}.\label{A23}
\end{eqnarray*}
We combine \eqref{A22} and \eqref{A23} to obtain that $\big[L_{\mC}]_{|\Gamma}$ is a bounded operator from $L_*(\Gamma)$ to $L_*(\Gamma)$ with $e^{\|\tau\|_\infty\|\sigma\|_\infty}$ as a uniform bound. Here a straightforward computation for $X=B(0,r_1)\b B(0,r_2)$ gives 
$
\|\tau\|_\infty=2\sqrt{r_1^2-r_2^2}.
$
{}\hfill$\Box$

\paragraph{Proof of \eqref{A2} and their higher dimensional counterparts.}
We prove the first estimate in \eqref{A2}.
We first assume that $X$ is a disk of center 0 and radius $1$. The constant $C_X$ behaves linearly with the radius of the disk. By symmetry, we can assume that $x_-=(-\sqrt{1-q^2},q)$ for some $q\in (0,1)$ and $v_-=(1,0)$. 
We introduce the angle $\theta_q=\arcsin(q)$. The point $x_-$ is located at the angle $\pi-\theta_q$ and its tangent vector has an angle $\pi/2-\theta_q$. We write $v=(\cos(\theta),\sin(\theta)))$, $\theta\in (-\pi,\pi)$. Then for a point $x\in \pa X$
$$
\chi_{\pa X_{-,x,v,v_-}}(x_-)=1
$$
if and only if $x=x_-+sv_-+tv$ for some positive number $t$ and $0<s<2\sqrt{1-q^2}$. Let 
$
\tilde x_-=(\cos(\theta_q),\sin(\theta_q))$. The lines $x_-+\R v_-$, $x_-+\R v$ and $\tilde x_-+\R v$ cross the boundary at the point  $\tilde x_-$, $x_-+2\cos(\theta+\theta_q)v=(\cos(\theta_q+2\theta),\sin(\theta_q+2\theta))$ and $\tilde x_-+2\cos(\theta+\pi-\theta_q)v=(\cos(-\pi-\theta_q+2\theta),\sin(-\pi-\theta_q+2\theta))$, respectively.
Then $x$ lies at the intersection of the boundary with the sector  based at $x_-$ and directed by $(v_-,v)$ and the sector based at $\tilde x_-$ and directed by $(-v_-,v)$. 
We denote 
$
A(x_-, v_-,v)=\int_{\pa X}\chi_{\pa X_{-,v_-,v,x}}(x_-)
|\nu(x)\cdot v|d\mu(x).
$
Then,
\begin{equation}
A(x_-,v_-,v)=
\int_{[\theta_q,2\theta+\theta_q]}|\cos(\theta-\phi)|d\phi
\le 4\pi\sin(\theta)=4\pi|v_-\cdot v^\bot|.\label{A300}
\end{equation}

When $\theta\in (-\pi/2-\theta_q, -\pi/2+\theta_q)$, then $v$ is ingoing at $x_-$ and at $\tilde x_-$ and $x=(\cos(\phi),\sin(\phi))$ for some $\phi$ between the angles $\theta_q+2\theta$ and $\pi-\theta_q+2\theta$. Then $\phi-2\theta$ is between the angles $\theta_q$ and $\pi-\theta_q$. Hence $|\cos(\phi-2\theta)|\le \cos(\theta_q)\le |\sin(\theta)|$. Then we obtain that 
$|\cos(\phi-2\theta+\theta)|\le |\cos(\phi-2\theta)|+|\sin(\theta)|$ and $A(x_-,v_-,v)$ is bounded by $4\pi|\sin(\theta)|$ as above.

When $\theta\in (-\pi/2+\theta_q,\pi/2-\theta_q)$, then $v$ is ingoing at $x_-$ and outgoing at $\tilde x_-$ and 
$x=(\cos(\phi),\sin(\phi))$ for some $\phi$ between the angle $\theta_q$ and $2\theta+\theta_q$. In that case $A(x_-, v_-,v)$ is trivially bounded by the length $2|\theta|$ which is bounded by $C\sin(\theta)$ for some constant $C$.

When $\theta\in (\pi/2+\theta_q, \pi)$, then $v$ is outgoing at $x_-$ and ingoing at $\tilde x_-$ and $x=(\cos(\phi),\sin(\phi))$ for some $\phi$ between the angles $-\pi-\theta_q+2\theta$ and $\pi-\theta_q$.  Then for $\theta'=\theta-\pi$, we have $\theta'\in (-\pi/2+\theta_q, 0)$ and  $\phi$ between the angles $\pi-\theta_q+2\theta'$ and $\pi-\theta_q$. In that case $A(x_-, v_-,v)$ is bounded by the length $2|\theta'|$ which is bounded by $C|\sin(\theta')|=C\sin(\theta)$ for some constant $C$.

When $\theta\in (-\pi, -\pi/2-\theta_q)$, then $v$ is outgoing at $x_-$ and ingoing at $\tilde x_-$ and $x=(\cos(\phi),$ $\sin(\phi))$ for some $\phi$ between the angles $\pi-\theta_q+2\theta$ and $-\pi-\theta_q$. Then for $\theta'=\theta+\pi$, we have $\theta'\in (0, \pi/2-\theta_q)$ and  $\phi$ between the angles $-\pi-\theta_q+2\theta'$ and $-\pi-\theta_q$. Again $A(x_-, v_-,v)$ is bounded by $C|\sin(\theta)|$ for some constant $C$.

Now assume that $X$ is the annulus $B(0,1)\b \overline{B(0,r)}$ for some $0<r<1$. (Again $C_X$ behaves linearly with the radius of the bigger disk.) Let $A$ be a measurable subset of $\pa B(0,r)$ so that any straight line in direction $v$ has at most one point in $A$. Then for $v\in V$ consider the measurable set 
$A_+=\{ x+tv\in \pa B(0,1)\ |\ x\in A,\ t>0\}$. One has
$$
\int_{A}|\nu(x)\cdot v|d\mu(x)=\int_{A_+}|\nu(x)\cdot v|d\mu(x).
$$
Hence for $(x_-,v_-)\in \pa X\times V$
$$
\int_{\pa X}\chi_{\pa X_{-,v_-,v,x}}(x_-)
|\nu(x)\cdot v|d\mu(x)\le \int_{\pa \tilde X}\chi_{\pa \tilde X_{-,v_-,v,x}}(\tilde x_-)
|\nu(x)\cdot v|d\mu(x),
$$
where $\tilde X$ is the disk $D(0,1)$ and $\tilde x_-\in \pa B(0,1)$, $\tilde x_- =x$ if $x\in \pa B(0,1)$ and  $x_-=\tilde x_-+sv_-$, $\tilde x\in$ for some positive $s$ otherwise. This ends the proof of the first estimate in \eqref{A2}.

We now prove that the first estimate implies the second one in \eqref{A2} when $X$ is a planar domain. Let $(v,v_-)\in V^2$ and $x_-\in \pa X$. We perform the change of variables $x=y-tv$, $y\in \pa X$, $dx=|\nu(y)\cdot v|dt d\mu(y)$, $\nu(y)\cdot v>0$ and we obtain the formula 
$$
\int_X\chi_{\pa X_{-,v_-,v,x}}(x_-)dx 
=\int_{y\in \pa X,\ \nu(y)\cdot v>0}\int_0^{\tau_-(y,v)}\chi_{\pa X_{-,v_-,v,y-tv}}(x_-)dt
(\nu(y)\cdot v)d\mu(y).
$$
Then the definition of the set $\pa X_{-,v_-,v,y-tv}$ implies that $\pa X_{-,v_-,v,y-tv}\subseteq \pa X_{-,v_-,v,y}$ for any $t$. Hence the right hand side is bounded by
$$
\int_{y\in \pa X\atop \nu(y)\cdot v>0}{\tau_-(y,v)}\chi_{\pa X_{-,v_-,v,y}}(x_-)dt
(\nu(y)\cdot v) dtd\mu(y)
\le 2r_1\int_{y\in \pa X}\chi_{\pa X_{-,v_-,v,y}}(x_-)
|\nu(y)\cdot v|d\mu(y).
$$
Thus, the first estimate in \eqref{A2} implies the second one.

The  higher dimensional counterpart \eqref{A200} also follows from a straightforward use of the first estimate \eqref{A2} on the planar domain $X\cap P_{v_-,\omega,y}$ which is either an annulus with ``$(r_1^2,r_2^2)$"$=(r_1^2-|y|^2,r_2^2-|y|^2)$, $0\le |y|< r_2$, or a disk of center 0 and radius the square root of  $r_1^2-|y|^2$. We use notation introduced before \eqref{A200}. 

The  higher dimensional counterpart \eqref{A201} is also derived by \eqref{A2} applied on the planar domain $X\cap P_{v_-,\omega,y}$. Indeed consider a boundary point $x\in \gamma_{v_-,\omega,y}$ and a unit vector $v$ lying on the plane $P_{v_-,\omega,y}$. We denote by $\nu_{P_{v_-,\omega,y}}(x)^\bot$ a unit vector lying on the plane $P_{v_-,\omega,y}$ and orthogonal to the normal vector $\nu_{P_{v_-,\omega,y}}(x)$. The vector $\nu_{P_{v_-,\omega,y}}(x)^\bot$ is tangent to the boundary $\gamma_{v_-,\omega,y}$. Hence it is also a tangent vector to the d-dimensional domain $X$. Therefore 
$
\nu(x)$ is orthogonal to  $\nu_{P_{v_-,\omega,y}}(x)^\bot$. Then $v=\alpha_1\nu_{P_{v_-,\omega,y}}(x)+\alpha_2\nu_{P_{v_-,\omega,y}}(x)^\bot$ for some real numbers $\alpha_i$, $i=1,2$, and obviously
$$
|\nu(x)\cdot v|=|\alpha_1\nu(x)\cdot\nu_{P_{v_-,\omega,y}}(x)|\le |\alpha_1|=|\nu_{P_{v_-,\omega,y}}(x)\cdot v|.
$$
This proves the estimate
$|\nu(x)\cdot v(\theta,\omega)|\le|\nu_{P_{v_-,\omega,y}}(x)\cdot v(\theta,\omega)|$ when $x=x(r_i,y,\theta')$,
and we have reduced the proof of \eqref{A201} to \eqref{A2}.\hfill $\Box$

\section{Proof of Lemma \ref{lem_trace_layer} and Theorem \ref{thm_eta}}
\label{sec:proofeps}

\subsection{Proof of Lemma \ref{lem_trace_layer}}

The Lemma follows by induction on the number $k$ of layers $Y_{k,N}$, $1\le k\le N$.

\paragraph{The base case.}
The trace of the solution $u$ on the boundary $\pa Z_{1\over N}\times V$ is written as 
$$
u_{|\Gamma(Z_{1\over N})}=f_0^{(1)}+f_1^{(1)},\ f_0^{(1)}=\big[L_{\mC_1}\phi\big]_{|\Gamma(Z_{1\over N})},\ f_1^{(1)}=\big[(I-T_{\mC_1}^{-1}K)^{-1}T_{\mC_1}^{-1}K L_{\mC_1}\phi\big]_{|\Gamma(Z_{1\over N})}.
$$

Then by definition of $L_{\mC_1}$ we have 
\begin{equation}
f_0^{(1)}(x,v)=\phi(x-\tau_-(Z_{1\over N})(x,v)v,v)e^{-{\tau_-(Z_{1\over N})(x,v) \over \ep}},\ (x,v)\in \Gamma(Z_{1\over N}).\label{B11}
\end{equation}
Obviously 
\begin{equation}
\|f_0^{(1)}\|_1\le \|\phi\|,\label{B1}
\end{equation}
and by definition of the constant $C_{p,\ep}$
\begin{equation}
\|f_1^{(1)}\|_{L^p(\Gamma(Z_1),d\xi)}\le C_{p,\ep} \|\phi\|. \label{B3}
\end{equation}

\paragraph{The induction step.}
Now assume that the statements of the Lemma are proved in all layers $Y_{l,N}$, $1\le l\le k$. We have
$$
u_{|\Gamma(Z_{s+{1\over N}})\b\mG_-(s+{1\over N},s)}=f_0^{(k+1)}+f_1^{(k+1)},
$$
where
\begin{equation}
f_0^{(k+1)}=\big[L_{\mC_{k+1}}f_0^{(k)}\big]_{|\Gamma(Z_{s+{1\over N}})\b\mG_-((s+{1\over N},s)},\label{B4}
\end{equation}
and 
\begin{equation}
(f_1)^{(k+1)}_{\Gamma(Z_{s+{1\over N}})\b \mG_-(s+{1\over N},s)}=A_{\mC_{k+1}}f_1^{(k)}+\big[(I-T_{\mC_{k+1}}^{-1}K)^{-1}T_{\mC_{k+1}}^{-1}K L_{\mC_{k+1}}f_0^{(k)}\big]_{|\Gamma(Z_{s+{1\over N}})\b \mG_-(s+{1\over N},s)}.\label{B5}
\end{equation}
Here we extended $f_0^{(k)}$ and $f_1^{(k)}$ by $0$ on $\mG_-(s+{1\over N},s)$.

From the definition of the operator $L_{\mC_{k+1}}$ and formula \eqref{f0i} for the index $k$, we obtain \eqref{f0i} with $k$ replaced by $k+1$.
The ballistic term $f_0^{(k+1)}$ can be considered as the ballistic term in the domain $Z_{s+{1\over N}}\b\bar Z_{1\over N}$: Introduce the set $\mC=\mC_-\cup\mC_+$ where $\mC_-$ as $\mG(s+{1\over N},{1\over N})\cup \Gamma_+(Z_{1\over N})$ and $\mC_+=\Gamma_-(Z_{1\over N})$. Then 
\begin{equation}
f_0^{(k+1)}=\big[L_{\mC}f_0^{(1)}\big]_{|\Gamma(Z_{s+{1\over N}})},\label{B10}
\end{equation}
and by Lemma \ref{lem_1p} ($(r_1,r_2)=(1+s+{1\over N}, 1+{1\over N})$) and \eqref{B1} we obtain that 
\begin{equation}
\|f_0^{(k+1)}\|_{k+1}=(1+e^{2\sqrt{s(2+s+{2\over N})}\over\ep})\|f_0^{(1)}\|_1\le (1+e^{2\sqrt{s(2+s+{2\over N})}\over\ep})\|\phi\|_1.\label{B6}
\end{equation}
This proves \eqref{f0e}.

From \eqref{B0} 
$$
\|A_{\mC_{k+1}}f_1^{(k)}\|_{L^p(\Gamma(Z_{s+{1\over N}})\b \mG_-(s+{1\over N},s),d\xi)}\le C_p\|f_1^{(k)}\|_{L^p(\Gamma(Z_s),d\xi)},
$$
and
$$
\|\big[(I-T_{\mC_{k+1}}^{-1}K)^{-1}T_{\mC_{k+1}}^{-1}K L_{\mC_{k+1}}f_0^{(k)}\big]_{|\Gamma(Z_{s+{1\over N}})}
\|_{L^p(\Gamma(Z_{s+{1\over N}}),d\xi)}\le C_{p,\ep} \|f_0^{(k)}\|_k.
$$
Therefore 
\begin{equation}
\|(f_1)^{(k+1)}\|_{L^p(\Gamma(Z_{s+{1\over N}}),d\xi)}\le C_p\|f_1^{(k)}\|_{L^p(\Gamma(Z_s),d\xi)}+C_{p,\ep} \|f_0^{(k)}\|_k.\label{B7}
\end{equation}
Hence
$$
C_p^{-k-1}\|(f_1)^{(k+1)}\|_{L^p(\Gamma(Z_{s+{1\over N}}),d\xi)}\le C_p^{-1}\|f_1^{(1)}\|_{L^p(\Gamma(Z_s),d\xi)}+C_{p,\ep}\sum_{l=1}^k C_p^{-l} \|f_0^{(l)}\|_l.
$$
We use \eqref{B6} and \eqref{B3} and we obtain \eqref{f1e}.\hfill$\Box$

\subsection{Proof of Theorem \ref{thm_eta}}
Any ray $x+\R v$, $(x,v)\in \mU_\eta$, intersects the ball $Z_{1\over N}$ without penetrating $Z_0$ when $\eta\le{1\over 2N}$. In addition 
\begin{equation}
\int_{\mU_\eta}d\xi\le C\eta,\label{B20}
\end{equation}
for a universal constant $C$. And $\Gamma_-(Z_1)\cap {\rm supp} \rho_\eta\subseteq \mU_\eta$.

By Lemma \ref{lem_trace_layer}
$[u_\eta]_{|Z_1}=f_{0,\eta}^{(N)}+f_{1,\eta}^{(N)}$ where $f_{0,\eta}^{(N)}$ is given by \eqref{f0i}. Here by definition of $\phi_\eta$
\begin{equation}
f_0^{(N)}(x,v)=\Big[\chi_{\mU_\eta}(x,v)e^{\tau_+(Z_1\b\bar Z_{1\over N})(x,v)\over \ep}+\chi_{\mU_{+,1}}(x,v)e^{-{d_{-,1}(x,v)\over \ep}}\Big]\rho_\eta(x-(x\cdot v)v,v),\label{B21}
\end{equation}
for $(x,v)\in \Gamma(Z_1)$.
From \eqref{B21}
$$
\int_{\mU_\eta}f_{0,\eta}^{(N)}d\xi \ge e^{\inf_{\mU_\eta} \tau_+(Z_1\b\bar Z_{1\over N})\over \ep}\int_{\mU_\eta}\rho_\eta(x-(x\cdot v)v,v) d\xi(x,v).
$$
Note that 
$$
\inf_{\mU_\eta} \tau_+(Z_1\b\bar Z_{1\over N})\ge 1-{1\over N},
$$
and 
\begin{eqnarray*}
\int_{\mU_\eta}\rho_\eta(x-(x\cdot v)v,v) d\xi(x,v)&=&\int_{\Gamma_-(Z_1)}\rho_\eta(x-(x\cdot v)v,v) d\xi(x,v)\\
&=&{1\over 2}\int_{\pa Z_1\times V}\rho_\eta(x-(x\cdot v)v,v) d\xi(x,v).
\end{eqnarray*}
Since $\pa Z_1$ is a sphere of radius $2$ and center 0, we have $\nu(x)=2^{-1}x$ for $x\in \pa Z_1$ and using spherical coordinates on $\pa Z_1$, ``$x=2(\sin(\theta)v+\cos(\theta)\omega)$, $\omega\cdot v=0$" we actually compute:
\begin{eqnarray*}
&&\int_{\pa Z_1\times V}\rho_\eta(x-(x\cdot v)v,v) {|\nu(x)\cdot v|}d\mu(x)dv\\
&=&\eta^{-1}2^{d-1}|V||\S^{d-2}|\int_{-{\pi\over 2}}^{\pi\over 2}\rho\big({1+{1\over 2N}-2|\cos(\theta)|\over \eta}\big)\cos(\theta)^{d-2}|\sin(\theta)|d\theta\\
&=&\eta^{-1}2^d|V||\S^{d-2}|\int_0^1\rho\big({1+{1\over 2N}-2s\over \eta}\big)s^{d-2}ds\\
&=&2^{d-1}|V||\S^{d-2}|\int_{-1+{1\over  N}\over \eta}^{1+{1\over 2N}\over \eta}\rho(s)({1+{1\over 2N}-\eta s\over 2})^{d-2}ds\\
&\ge& 2^{d-1}|V||\S^{d-2}|\int_{-1}^{1}\rho(s)({1+{1\over 2N}-\eta s\over 2})^{d-2}ds
\ge 2|V||\S^{d-2}|
\end{eqnarray*}
when $\eta\le (2N)^{-1}$. Therefore
\begin{equation}
\int_{\mU_\eta}f_{0,\eta}^{(N)}d\xi\ge e^{ 1-{1\over N}\over \ep}2|V||\S^{d-2}|\textrm{ when }0<\eta\le {1\over 2N}.\label{B22}
\end{equation}
(We used the convention $|\S^{d-2}|=2$ in dimension $d=2$.)

Then by \eqref{B20} and \eqref{f1e}
\begin{eqnarray}
\big|\int_{\mU_\eta}f_{1,\eta}^{(N)}d\xi\big|
&\le& C^{1\over q}\eta^{1\over q}\|f_{1,\eta}^{(N)}\|_{L^p(\Gamma_-(Z_1),d\xi)}\nonumber\\
&\le&C^{1\over q} C_{p,\ep}C_p \big(C_p^{N-1}+{C_p^N-1\over C_p-1}(1+e^{2\sqrt{3+{2\over N}}\over\ep})\big)\eta^{1\over q}\|\phi_\eta\|_1,\label{B23}
\end{eqnarray}
where $p^{-1}+q^{-1}=1$.
It remains to compute $\|\phi_\eta\|_{L^1(\Gamma_-(Z_{1\over N}),d\xi)}$ and find an upper bound for $\|\phi_\eta\|_1$.

\paragraph{Proof of formula \eqref{B30a}.}
The computation is similar to the derivation of \eqref{B22}. Since $\pa Z_{1\over N}$ is a sphere of radius $1+{1\over N}$ and center 0, we have $\nu(x)=(1+{1\over N})^{-1}x$ for $x\in \pa Z_{1\over N}$ and using spherical coordinates on $\pa Z_{1\over N}$, ``$x=(1+{1\over N})(\sin(\theta)v+\cos(\theta)\omega)$, $\omega\cdot v=0$" we actually compute:
\begin{eqnarray}
&&\|\phi_\eta\|_{L^1(\Gamma_-(Z_{1\over N}),d\xi)}=\int_{\Gamma_-(Z_{1\over N})}|\phi_\eta|d\xi={1\over 2}\int_{\pa Z_{1\over N}\times V}\rho_\eta(x-(x\cdot v)v,v)d\xi(x,v)\nonumber\\
&=&|V|\eta^{-1}(1+{1\over N})^{d-1}|\S^{d-2}|\int_{0}^{\pi\over 2}\rho\big({1+{1\over 2N}-(1+{1\over N})\cos(\theta)\over \eta}\big)\sin(\theta)\cos(\theta)^{d-2}d\theta\nonumber\\
&=&|V||\S^{d-2}|\int_{-{1\over 2\eta N}}^{1+{1\over 2N}\over \eta}\rho(s)({1+{1\over 2N}}-\eta s)^{d-2}ds\nonumber\\
&=&|V||\S^{d-2}|\int_{-1}^{1}\rho(s)({1+{1\over 2N}}-\eta s)^{d-2}ds.\label{B24}
\end{eqnarray}
Then we use that $\int_\R\rho=1$, $\rho\ge 0$ and obtain \eqref{B30a}.

\paragraph{Upper bound for $\|\phi_\eta\|_1$ in dimension $d\ge 3$.}
We recall that
\begin{equation}
\|\phi_\eta\|_{L*(\Gamma(Z_{1\over N}))}=\int_V \Big({\rm supess}_{(\omega,y)\in {\mathcal M}_{v_-}} \int_{\gamma_{v_-,\omega,y}} |\phi_\eta|(\gamma,v_-)|\nu_{P_{v_-,\omega,y}}(\gamma)\cdot v_-|d\gamma \Big)dv_-,
\end{equation}
Here $\gamma_{v_-,\omega,y}$ and ${\mathcal M}$ are defined at the end of section \ref{sec:forward} when $(r_1,r_2)=(1+{1\over N},1)$. The function $\phi_\eta$ is invariant under rotation in the $v_-$- variable and $\phi_\eta$ also vanishes on $\pa Z_0$. Hence 
\begin{equation}
\|\phi_\eta\|_{L*(\Gamma(Z_{1\over N}))}= |V|{\rm supess}_{(v_-,\omega,y)\in {\mathcal M}} \int_{\gamma_{v_-,\omega,y}\cap \pa Z_{1\over N}}
|\phi_\eta|(\gamma,v_-)|\nu_{P_{v_-,\omega,y}}(\gamma)\cdot v_-|d\gamma.\label{B40}
\end{equation}

Let  $(v_-,\omega,y)\in {\mathcal M}$. We omit the indices ${v_-,\omega,y}$ for $\gamma_{v_-,\omega,y}$ and $P_{v_-,\omega,y}$:
$$
\int_{\gamma\cap \pa Z_{1\over N}} |\phi_\eta|(\gamma,v_-)|\nu_{P}(\gamma)\cdot v_-|d\gamma
=\eta^{-1}\int_{\gamma\cap \pa Z_{1\over N}} \rho({1+{1\over 2N}-|\gamma-(\gamma\cdot v_-)v_-|\over \eta})|\nu_{P}(\gamma)\cdot v_-|d\gamma.
$$
As above we introduce the coordinates $\gamma=y+\sqrt{(1+{1\over N})^2-|y|^2}(\cos(\theta)\omega+\sin(\theta)v_-)$ and the change of variables $s=\cos(\theta)$ and we obtain
\begin{eqnarray*}
\int_{\gamma\cap \pa Z_{1\over N}} |\phi_\eta|(\gamma,v_-)|\nu_{P}(\gamma)\cdot v_-|d\gamma
&=&{4\over \eta}\sqrt{((1+{1\over N})^2-|y|^2)}\\
&&\times\int_0^1 \rho({1+{1\over 2N}-\sqrt{|y|^2+((1+{1\over N})^2-|y|^2)s^2}\over \eta})ds.
\end{eqnarray*}
(The length of the circle $\gamma\cap \pa Z_{1\over N}$ is $2\pi \sqrt{(1+{1\over N})^2-|y|^2}$.)
Then we perform the change of variables 
$$
s={\sqrt{(1+{1\over 2N}+\eta t)^2-|y|^2}\over \sqrt{(1+{1\over N})^2-|y|^2}},\ ds={\eta(1+{1\over 2N}+\eta t) dt\over \sqrt{(1+{1\over N})^2-|y|^2}\sqrt{(1+{1\over 2N}+\eta t)^2-|y|^2}}
$$
and 
\begin{eqnarray*}
\int_{\gamma\cap \pa Z_{1\over N}} |\phi_\eta|(\gamma,v_-)|\nu_{P}(\gamma)\cdot v_-|d\gamma
&=&4\int^{{1\over 2N\eta}}_{-1-{1\over 2N}+|y|\over \eta}{(1+{1\over 2N}+\eta t)\rho(t)\over \sqrt{(1+{1\over 2N}+\eta t)^2-|y|^2}}dt\\
&\le&4\int^{1}_{\min\big({-1-{1\over 2N}+|y|\over \eta},1\big)}{\rho(t)\sqrt{1+{1\over 2N}+\eta t}\over \sqrt{1+{1\over 2N}+\eta t-|y|}}dt.
\end{eqnarray*}
Let us write $y=1+{1\over 2N}+t_y\eta$. Then
$$
\int_{\gamma\cap \pa Z_{1\over N}} |\phi_\eta|(\gamma,v_-)|\nu_{P}(\gamma)\cdot v_-|d\gamma
\le {4\sqrt{1+{1\over N}}\over \sqrt{\eta}}\int^{1}_{\min(t_y,1)}{\rho(t)\over \sqrt{t-t_y}}dt\le \tilde C |V|^{-1}\eta^{-{1\over 2}},
$$
where 
$$
\tilde C=2^{3\over 2}|V|\sup_{\sigma\in (-\infty,1)}\int_{-1}^1{\rho(t)\over \sqrt{|t-\sigma|}}d\sigma
$$

Finally we obtain
\begin{equation}
\|\phi_\eta\|_{L*(\Gamma(Z_{1\over N}))}\le \tilde C \eta^{-{1\over 2}}.\label{B40b}
\end{equation}

\paragraph{Proof of \eqref{B30c} in dimension $d=2$.} 
From \eqref{B23} and \eqref{B30a} $(\|\phi_\eta\|_1=2)$ it follows that
\begin{eqnarray}
\big|\int_{\mU_\eta}f_{1,\eta}^{(N)}d\xi\big|
&\le&|V||\S^{d-2}|C^{1\over q} C_{p,\ep}C_p \big(C_p^{N-1}+{C_p^N-1\over C_p-1}(1+e^{2\sqrt{3+{2\over N}}\over\ep})\big)\eta^{1\over q}.\label{B41a}
\end{eqnarray}
Therefore we combine \eqref{B41a} and \eqref{B22}  and we obtain \eqref{B30c} where
$$
\eta^{1\over q}\le {e^{ 1-{1\over N}\over \ep}\over C^{1\over q} C_{p,\ep}C_p \big(C_p^{N-1}+{C_p^N-1\over C_p-1}(1+e^{2\sqrt{3+{2\over N}}\over\ep})\big)}.
$$

\paragraph{Proof of \eqref{B30c} in dimension $d\ge 3$.} 
From \eqref{B23} and \eqref{B40b} it follows that
\begin{eqnarray}
\big|\int_{\mU_\eta}f_{1,\eta}^{(N)}d\xi\big|
&\le&C^{1\over q}\tilde C C_{p,\ep}C_p \big(C_p^{N-1}+{C_p^N-1\over C_p-1}(1+e^{2\sqrt{3+{2\over N}}\over\ep})\big)\eta^{{1\over q}-{1\over 2}}.\label{B41b}
\end{eqnarray}
Therefore we combine \eqref{B41b} and \eqref{B22}  and we obtain \eqref{B30c} when 
$$
\eta^{{1\over q}-{1\over 2}}\le {e^{ 1-{1\over N}\over \ep}|V||\S^{d-2}|\over C^{1\over q} \tilde C C_{p,\ep}C_p \big(C_p^{N-1}+{C_p^N-1\over C_p-1}(1+e^{2\sqrt{3+{2\over N}}\over\ep})\big)}.
$$
{}\hfill $\Box$

\section{Proof of Lemmas \ref{lem_bound1} and \ref{lem_K}}
\label{sec:appbound}

\subsection{Proof of Lemma \ref{lem_bound1}}
We first prove \eqref{i1a}.
We have 
\begin{eqnarray*}
\|L_{\mC,\sigma} g\|_{W^p}&=&\|(\tau\sigma) \tau^{-{1\over p}}L_{\mC,\sigma} g\|_{L^p}+\|\tau^{-{1\over p}}L_{\mC,\sigma} g\|_{L^p}\\
&\le &(1+\|\tau\sigma\|_\infty)\|\tau^{-{1\over p}}L_{\mC,\sigma} g\|_{L^p}\le  e^{\|\tau\sigma\|_\infty}(1+\|\tau\sigma\|_\infty)\|g\|_{L^p(\mC,d\xi)}.
\end{eqnarray*}
Similarly we obtain the estimate of $L_\mC g$ with respect to $\tilde W^p$.

Then we prove \eqref{i2} and \eqref{i3}.
Let $D_\pm:=\{(x,v)\in X\times V\ | \ (x\pm \tau_\pm(x,v)v,v))\in  \mC_\pm\}$. Then 
\begin{eqnarray*}
\Big(\int_{X\times V}\tau^{-\ep p}|T_\mC^{-1}f(x,v)|^p dx dv\Big)^{1\over p}
&\le& e^{\|\tau\sigma\|_\infty}\Big[\Big(\int_{D_+}\tau^{-\ep p}\big(\int_0^{\tau_+(x,v)}\!\!\!\!\!\!\!\!\!|f(x+sv,v)|ds\big)^p dx dv\Big)^{1\over p}\\
&&+\Big(\int_{D_-}\tau^{-\ep p}\big(\int_0^{\tau_-(x,v)}\!\!\!\!\!\!\!\!\!|f(x-sv,v)|ds\big)^p dx dv\Big)^{1\over p}\Big].
\end{eqnarray*}
Next,
\begin{eqnarray*}
&&\int_{D_+}\tau^{-\ep p}\big(\int_0^{\tau_+(x,v)}|f(x+sv,v)|ds\big)^p dx dv\\
&\le & \int_{D_+}\tau^{p(1-\ep) -1}\int_0^{\tau_+(x,v)}|f(x+sv,v)|^p ds dx dv,\\
&=&\int_{\mC_+}\int_0^{\tau(x,v)}\tau(x,v)^{p(1-\ep) -1}\int_0^t|f(x-tv+sv,v)|^p ds dt d\xi(x,v)\\
&=&\int_{\mC_+}\int_0^{\tau(x,v)}\tau(x,v)^{p(1-\ep) -1}(\tau(x,v)-s)|f(x-sv,v)|^p ds  d\xi(x,v)\ \le\ \|\tau^{1-\ep}f\|_{L^p(D_+)}^p.
\end{eqnarray*}
Similarly
\begin{eqnarray*}
\int_{D_-}\tau^{-\ep p}\big(\int_0^{\tau_-(x,v)}|f(x-sv,v)|ds\big)^p dx dv\le \|\tau^{1-\ep}f\|_{L^p(D_-)}^p,
\end{eqnarray*}
which provides \eqref{i2}.

Estimate \eqref{i3} follows from \eqref{i2} and the identity $v\cdot\nabla_x T^{-1}_\mC f+\sigma T^{-1}_\mC f=f$.\hfill $\Box$

\subsection{Proof of Lemma \ref{lem_K}}
We only treat the case $1<p<\infty$.
Let $\phi\in L^p(X\times V)$. Assume \eqref{i10a}. We have
\begin{eqnarray*}
&&\int_{X\times V}|K\phi(x,v)|^p dx dv=\int_{X\times V}\big(\int_V |k(x,v',v)||\phi(x,v')|dv'\big)^p dx dv\\
&&\le \int_{X\times V}\big(\int_V |k(x,v'',v)|^{p\over p-1}dv''\big)^{p-1} \big(\int_V |\phi(x,v')|^p dv'\big) dx dv\\
&\le& \|\phi\|_{L^p(X\times V)}^p \big\|\int_V\big(\int_V|k|^{p\over p-1}(.,v',v)dv'\big)^{p-1}dv\big\|_\infty.
\end{eqnarray*}
Similarly assume \eqref{i11a} and
let $\phi\in \tau^{1\over p}L^p(X\times V)$. Then
\begin{eqnarray*}
&&\int_{X\times V}\tau^{p-1}|K\phi(x,v)|^p dx dv\le \int_{X\times V}\tau^{p-1}\big(\int_V k(x,v',v) \phi(x,v')dv'\big)^p dx dv\\
&\le& \int_{X\times V}\tau^{p-1}\big(\int_V k(x,v',v) \phi(x,v')dv'\big)^p dx dv\\
&\le& \int_{X\times V}\tau^{p-1}\int_V \tau(x,v')^{-1} |\phi(x,v')|^p dv' \big(\int_V |k(x,v'',v)|^{p\over p-1}\tau(x,v'')^{1\over p-1} dv''\big)^{p-1}dx dv \\
&\le& \|\int_V \tau^{p-1} \big(\int_V |k(x,v'',v)|^{p\over p-1}\tau(x,v'')^{1\over p-1} dv''\big)^{p-1}dv\|_{L^\infty(X)}\\
&&\times\int_{X\times V} \tau(x,v')^{-1}|\phi(x,v')|^p dv' dx.
\end{eqnarray*}

Now assume \eqref{i10b}. 
Let $(\phi,\psi)\in L^p(X\times V)\times L^{p\over p-1}(X\times V)$. We have
\begin{eqnarray*}
&&|\int_{X\times V}K\phi(x,v) \psi(x,v)dx dv|\le\int_{X\times V\times V}|k|(x,v',v)|\phi|(x,v')\psi(x,v) dx dv dv'\\
&\le& \big(\int_{X\times V\times V}|k(x,v',v)||\phi|^p(x,v')dv'dx dv\big)^{1\over p}\\
&&\times\big(\int_{X\times V\times V}|k(x,v',v)|\psi|^{p\over p-1}(x,v)dv'dx dv\big)^{p-1\over p}\\
&=&\|\sigma_s\|_{\infty}^{1\over p} \|\sigma_s'\|_{\infty}^{p-1\over p} \|\phi\|_{L^p(X\times V)}\|\psi\|_{L^{p\over p-1}(X\times V)}.
\end{eqnarray*}
Similarly assume \eqref{i11b}.
Let $(\phi,\psi)\in \tau^{1\over p}L^p(X\times V)\times \tau^{p-1\over p}L^{p\over p-1}(X\times V)$. We have
\begin{eqnarray*}
|\int_{X\times V}K\phi(x,v) \psi(x,v)dx dv|
&\le& \big(\int_{X\times V\times V}|k(x,v',v)||\phi|^p(x,v')dv'dx dv\big)^{1\over p}\\
&&\times\big(\int_{X\times V\times V}|k(x,v',v)||\psi|^{p\over p-1}(x,v)dv'dx dv\big)^{p-1\over p}\\
&=&\|\tau\sigma_s\|_{\infty}^{1\over p} \|\tau\sigma_s'\|_{\infty}^{p-1\over p} \|\tau^{-{1\over p}}\phi\|_{L^p(X\times V)}\|\tau^{-{p-1\over p}}\psi\|_{L^{p\over p-1}(X\times V)}.
\end{eqnarray*}
{}\hfill$\Box$

\section{Proof of Lemmas \ref{lem_smooth}, \ref{lem_choice}, \ref{lem_spec} and \ref{lem_specbd}}
\label{sec:five}

\subsection{Proof of Lemma \ref{lem_smooth}}
We note that 
\begin{equation}
KT_{\Gamma_+}^{-1}f(x,v)=-\int_X{e^{|x-y|\int_0^1\sigma(\ep x+(1-\ep) y,\widehat{y-x})d\ep}k(x,\widehat{y-x},v)\over |x-y|^{d-1}}f(y,\widehat{y-x})dy,\label{F4}
\end{equation}
for $f\in L^2(X\times V)$ (we used a change of variables $y=x+tv'$).
Hence
\begin{eqnarray*}
&&|\big(KT_{\Gamma_+}^{-1}\big)^{d+2}f(x_0,v)|\le e^{(n+2)\diam \|\sigma\|_\infty}\|k\|_\infty^{d+2}\\
&&\times \int_{X^{d+2}}{|f(x_{d+2},\widehat{x_{d+2}-x_{d+1}})|dx_1\ldots dx_{d+2}\over\Pi_{i=1}^{d+2}|x_i-x_{i-1}|^{d-1}}\\
&\le &e^{(d+2)\diam\|\sigma\|_\infty}\|k\|_\infty^{d+2}\\
&&\times\Big[\sup_{(z_0,z_{d+1})\in X^2}\int_{X^d}{dz_1\ldots dz_d\over\Pi_{i=1}^{d+1}|z_i-z_{i-1}|^{d-1}}\Big] \int_{X^2} {|f(x_{d+2},\widehat{x_{d+2}-x_{d+1}})|\over |x_{d+1}-x_{d+2}|^{d-1}}dx_{d+1}dx_{d+2},
\end{eqnarray*}
where we recall that (see for instance \cite{BJ})
$$
C:=\sup_{(z_0,z_{d+1})\in X^2}\int_{X^d}{dz_1\ldots dz_d\over\Pi_{i=1}^{d+1}|z_i-z_{i-1}|^{d-1}}<\infty,
$$
and we have 
\begin{eqnarray*}
\int_{X^2} {|f(x_{d+2},\widehat{x_{d+2}-x_{d+1}})|\over |x_{d+1}-x_{d+2}|^{d-1}}dx_{d+1}dx_{d+2}
&=&\int_{X\times V}\int_0^{\tau_-(x_{d+2},v)} |f(x_{d+2},v)|dt dvdx_{d+2}
\end{eqnarray*}
(we used the change of variables $x_{d+2}=x_{d+1}-tv$). Hence 
$$
|\big(KT_{\Gamma_+}^{-1}\big)^{d+2}f(x_0,v)|\le  C\diam e^{(d+2)\diam \|\sigma\|_\infty}\|k\|_\infty^{d+2}\|f\|_{L^1(X\times V)}<\infty.
$$
Therefore 
$
R u=\lambda^{-1}KT_{\Gamma_+}^{-1}R u=(\lambda^{-1})^{d+2}(KT_{\Gamma_+}^{-1})^{d+2}Ru\in L^\infty(X\times V).
$
Then from \eqref{F4} applied twice it follows that
\begin{eqnarray*}
R u(x,v)&=&\lambda^{-2}\int_{X^2}{e^{|x-x_1|\int_0^1\sigma(\ep x+(1-\ep) x_1,\widehat{x_1-x})d\ep
+|x_1-x_2|\int_0^1\sigma(\ep x_1+(1-\ep)x_2,\widehat{x_2-x_1})}\over |x-x_1|^{d-1}|x_1-x_2|^{d-1}}\\
&&\times k(x,\widehat{x_1-x},v)k(x_1,\widehat{x_2-x_1},\widehat{x_1-x})
R u(x_2,\widehat{x_2-x_1})dx_1 dx_2.
\end{eqnarray*}
Hence, from the above formula, the boundedness of $R u$, and the continuity of the optical parameters $\sigma$ and $k$, it follows that $R u$ is a continuous function on $\bar X\times V$.\hfill $\Box$

\subsection{Proof of Lemma \ref{lem_choice}}
First set
$$
\gamma(x,v',v)=R\tilde u_1(x,v')R\tilde u_2(x,v)+R\tilde u_1(x,v)R\tilde u_2(x,v').
$$
Then
\begin{eqnarray*}
&&\int_{X\times V^2} \gamma\big((P_{\rm even}T^{-1}_{\Gamma_+}R\tilde u_1)(x,v')R\tilde u_2(x,v)+R\tilde u_1(x,v')(P_{\rm even}T^{-1}_{\Gamma_+}R\tilde u_2)(x,v)\big) dx dvdv'\\
&=&\int_X \Big(\int_V R \tilde u_1(x,v)T^{-1}_{\Gamma_+} R \tilde u_1)(x,v)dv\Big)
\Big(\int_V |R \tilde u_2(x,v')|^2dv'\Big) dx\\
&&+\int_X \Big(\int_V R \tilde u_2(x,v)(T^{-1}_{\Gamma_+} R \tilde u_2)(x,v)dv\Big)
\Big(\int_V |R \tilde u_1(x,v')|^2dv'\Big) dx\\
&&+\int_X \Big(\int_V R \tilde u_1(x,v)(T^{-1}_{\Gamma_+} R \tilde u_2)(x,v)dv\Big)
\Big(\int_V R \tilde u_1(x,v')R \tilde u_2(x,v')dv'\Big) dx\\
&&+\int_X \Big(\int_V R \tilde u_2(x,v)(T^{-1}_{\Gamma_+} R \tilde u_1)(x,v)dv\Big)
\Big(\int_V R \tilde u_2(x,v')R \tilde u_1(x,v')dv'\Big) dx.
\end{eqnarray*}
Now we use the property \eqref{i302} and the identity $RT^{-1}_{\Gamma_+} R \tilde u_j=\lambda \tilde u_j$,and we obtain
\begin{eqnarray*}
&&\lambda^{-1}\int_{X\times V^2} \gamma\big((T^{-1}_{\Gamma_+}R\tilde u_1)(x,v')R\tilde u_2(x,v)+R\tilde u_1(x,v')(T^{-1}_{\Gamma_+}R\tilde u_2)(x,v)\big) dx dvdv'\\
&=&\int_X \Big(\int_V  |\tilde u_1(x,v)|^2dv\Big)
\Big(\int_V |R \tilde u_2(x,v')|^2dv'\Big) dx\\
&&+\int_X \Big(\int_V  |\tilde u_2(x,v)|^2dv\Big)
\Big(\int_V |R \tilde u_1(x,v')|^2dv'\Big) dx\\
&&+2\int_X \Big(\int_V (\tilde u_1\tilde u_2)(x,v)dv\Big)
\Big(\int_V (R \tilde u_1 R \tilde u_2)(x,v')dv'\Big) dx\ge 0.
\end{eqnarray*}
We applied Cauchy-Schwarz inequality to each single integral over $V$ at the last line. We also obtain that the right hand side vanishes if and only if
$$
\int_V  |\tilde u_1(x,v)|^2dv
\Big(\int_V |R \tilde u_2(x,v')|^2dv'\Big)=\int_V  |\tilde u_2(x,v)|^2dv
\Big(\int_V |R \tilde u_1(x,v')|^2dv'\Big) ,
$$
\begin{equation}
\int_V (\tilde u_1\tilde u_2)(x,v)dv
\Big(\int_V (R \tilde u_1 R \tilde u_2)(x,v')dv'\Big)
=-\int_V  |\tilde u_1(x,v)|^2dv
\Big(\int_V |R \tilde u_2(x,v')|^2dv'\Big) \label{i306}
\end{equation}
for a.e. $x\in X$. Therefore we consider the measurable set $O$ of $X$ where 
$$
\int_V  |\tilde u_1(x,v)|^2dv
\Big(\int_V |R \tilde u_2(x,v')|^2dv'\Big)=\int_V  |\tilde u_2(x,v)|^2dv
\Big(\int_V |R \tilde u_1(x,v')|^2dv'\Big)\not=0
$$
and we obtain by equality in Cauchy-Schwarz inequality applied to \eqref{i306} 
\begin{equation}
\tilde u_1(x,v)=\ep(x)\tilde u_2(x,v) \textrm{ a.e. }(x,v)\in O\times V,\label{i305}
\end{equation}
where $\ep$ is a real valued measurable function.
But going back to \eqref{i306} we see that the left hand side is the nonnegative number
$$
\ep^2(x)\int_V |\tilde u_1|^2(x,v)dv
\Big(\int_V |R \tilde u_1|^2 (x,v')dv'\Big).
$$
Hence either $O$ is a negligible set or \eqref{i304} does not hold. In the latter case we found a deformation defined by $\gamma$ so that all $\alpha_j$'s can not have the same value.\\

Therefore we now assume that $O$ is negligible. Then we have 
\begin{equation}
\int_V  |\tilde u_1(x,v)|^2dv
\Big(\int_V |R \tilde u_2(x,v')|^2dv'\Big)=\int_V  |\tilde u_2(x,v)|^2dv
\Big(\int_V |R \tilde u_1(x,v')|^2dv'\Big)=0\label{i310}
\end{equation}
for a.e. $x\in X$. In particular we also have 
\begin{equation}
 \int_V (\tilde u_1\tilde u_2)(x,v)dv
\Big(\int_V (R \tilde u_1 R \tilde u_2)(x,v')dv'\Big) =0\label{i311}
\end{equation}
for a.e. $x\in X$.

Now set
$$
\gamma(x,v',v)=R\tilde u_1(x,v')R\tilde u_1(x,v),
$$
and replace the couple of vectors $(\tilde u_1,\tilde u_2)$ by the orthonormal couple 
$({\tilde u_1-\tilde u_2\over \sqrt{2}},{\tilde u_1+\tilde u_2\over \sqrt{2}})$. We compute as before and we use the identities \eqref{i310} and \eqref{i311} and we obtain
\begin{eqnarray}
&&\int_{X\times V^2} \gamma\big(T^{-1}_{\Gamma_+}R{\tilde u_1-\tilde u_2\over \sqrt{2}}(x,v')R{\tilde u_1+\tilde u_2\over \sqrt{2}}(x,v)\nonumber\\
&&+R{\tilde u_1-\tilde u_2\over \sqrt{2}}(x,v')(T^{-1}_{\Gamma_+}R{\tilde u_1+\tilde u_2\over \sqrt{2}})(x,v)\big) dx dvdv'\nonumber\\
&=&{\lambda\over 2}\int_X \Big(\int_V  |\tilde u_1(x,v)|^2dv\Big)
\Big(\int_V |R \tilde u_1(x,v')|^2dv'\Big) dx.\label{i312}
\end{eqnarray}
Then for a.e. $x\in X$ we have 
\begin{eqnarray}
\int_V |R \tilde u_1(x,v)|^2dv&=&\int_V (K\tilde u_1)(x,v)\tilde u_1(x,v)dv\nonumber\\
&\le &\|k\|_{\infty}\big|\int_V  \tilde u_1(x,v)\big|^2\le \|k\|_{\infty}|V|\int_V |\tilde u_1(x,v)|^2dv.\label{i313}
\end{eqnarray}
Therefore we combine \eqref{i312} and \eqref{i313} and we obtain that there exists a positive constant $C$ so that 
\begin{eqnarray*}
&&\int_{X\times V^2} \gamma\big(T^{-1}_{\Gamma_+}R{\tilde u_1-\tilde u_2\over \sqrt{2}}(x,v')R{\tilde u_1+\tilde u_2\over \sqrt{2}}(x,v)\nonumber\\
&&+R{\tilde u_1-\tilde u_2\over \sqrt{2}}(x,v')(T^{-1}_{\Gamma_+}R{\tilde u_1+\tilde u_2\over \sqrt{2}})(x,v)\big) dx dvdv'
\ge C\int_X \
\Big(\int_V |R \tilde u_1(x,v')|^2dv'\Big)^2 dx.
\end{eqnarray*}
Hence the right hand side vanishes if and only if 
$
R\tilde u_1=0,
$
which contradicts that $\tilde u_1$ is a unit eigenvector of the operator $RT^{-1}_{\Gamma_+}R$.
\hfill $\Box$

\subsection{Proof of Lemma \ref{lem_spec}}
Introducing spherical coordinates we have 
for $d\ge 2$
$$
Ff(x)=-\int_0^1 r^{d-1}\int_{\S^{d-1}}{e^{\sigma\sqrt{|x|^2+r^2-2r(x\cdot\omega)}}\over 
\big(|x|^2+r^2-2r(x\cdot\omega)\big)^{d-1\over 2}} f(r\omega)d\omega dr.
$$
For $f(r\omega)=g(r)\phi(\omega)$, $\phi\in H_l(V)$, we obtain that 
$$
F(f)(s\omega)=-\int_0^1 r^{d-1}g(r)\int_{\S^{d-1}}{e^{\sigma\sqrt{s^2+r^2-2rs(\omega'\cdot\omega)}}\over 
\big(s^2+r^2-2rs(\omega'\cdot\omega)\big)^{d-1\over 2}} \phi(\omega') d\omega'dr.
$$
Next, following \cite[VII.3.2]{Na},
$
F(f)(s\omega)=\phi(\omega)\int_0^1 r^{d-1}g(r)f_l(s,r) dr,
$
where
\begin{equation}
f_l(s,r)=-|\S^{d-2}|\int_{-1}^1 {(1-t^2)^{d-3\over 2} G_l^{d-2\over 2}(t) e^{\sigma\sqrt{s^2+r^2-2rst}}\over 
\big(s^2+r^2-2rs t\big)^{d-1\over 2}} dt,\ s\not=r,
\label{D1}
\end{equation}
and $G_l^{d-2\over 2}$ denotes the Gegenbauer polynomial of degree $l$ with normalizing condition $G_l^{d-2\over 2}(1)$ $=1$ which is the unique polynomial (with values 1 at 1) of the equation \cite[Chapter V, 5.1.3]{MOS}
\begin{equation}
(1-x^2){d^2 G_l^{d-2\over 2}\over dx^2}(x)-(d-1)x{d G_l^{d-2\over 2}\over d x}(x)+l(l+d-2)G_l^{d-2\over 2}(x)=0, x\in 
(-1,1).\label{F5b}
\end{equation}
Hence the operator $F_l$ defined on $L^2([0,1], r^{d-1}dr)$ by 
\begin{equation}
F_l g(s)=\int_0^1 g(r)f_l(s,r)r^{d-1}dr\label{D2}
\end{equation}
is a selfadjoint compact (in fact Hilbert-Schmidt) operator in $L^2([0,1], r^{d-1}dr)$. 

For these last statements we used the following estimates
\begin{eqnarray}
r^{d-1}|f_l(s,r)|&\le & C_l\int_0^\pi {r^{d-1}\sin(\theta_1)^{d-2}\over 
\big(s^2+r^2-2rs\cos(\theta_1)\big)^{d-1\over 2}} d\theta_1\nonumber\\
&=&C_l\int_0^\pi {r^{d-1}\sin(\theta_1)^{d-2}\over 
\big(r^2\sin(\theta_1)^2+(s-r\cos(\theta_1))^2\big)^{d-2\over 2}\sqrt{(s^2+r^2-2rs\cos(\theta_1)}\big)} d\theta_1\nonumber\\
&\le&C_l\int_0^\pi {r\over 
\sqrt{(s^2+r^2-2rs\cos(\theta_1)}\big)} d\theta_1\le C_l'(\ln({1\over r-s})+1),\label{D10}
\end{eqnarray}
for any $(s,r)\in [0,1]^2$, $s<r$, and for some constants $C_l$ and $C_l'$.\hfill $\Box$

\subsection{Proof of Lemma \ref{lem_specbd}}
Let us first consider the Gegenbauer polynomials introduced above in \eqref{D1}. We recall the property \cite[Chapter V, 5.2.1 and 5.1.3]{MOS}:
$
G_l^{d-2\over 2}(-x)=(-1)^l G_l^{d-2\over 2}(x).
$
Hence $l$ being odd we actually have
$
G_l^{d-2\over 2}(-1)=-G_l^{d-2\over 2}(1)=-1.
$
Therefore there exist positive numbers $\ep$ and $\eta$, $\eta< 1/2$ so that 
$$
-|\S^{d-2}|G_l^{d-2\over 2}(t)\ge \ep
$$
for $t\in (-1,-1+\eta)$. 
Now let $(s,r)\in[\eta,1]^2$, $s\not=r$. Then the above estimate gives
$$
-|\S^{d-2}|\int_{-1}^{-1+\eta} {(1-t^2)^{d-3\over 2} G_l^{d-2\over 2}(t) e^{\sigma\sqrt{s^2+r^2-2rst}}\over 
\big(s^2+r^2-2rs t\big)^{d-1\over 2}} dt
\ge \ep \int_{-1}^{-1+\eta} {(1-t^2)^{d-3\over 2}  e^{\sigma\sqrt{s^2+r^2-2rst}}\over 
\big(s^2+r^2\big)^{d-1\over 2}} dt.
$$
The integrand on the right hand side is non-negative and the integral over $(-1,-1+\eta)$ is bounded below by the integral of the same integrand over the smallest interval $(-1+\eta/2,-1+\eta)$. In addition using monotonicity in $t$ of the denominator and of the exponential we easily obtain
\begin{equation}
-|\S^{d-2}|\int_{-1+\eta/2}^{-1+\eta} {(1-t^2)^{d-3\over 2} G_l^{d-2\over 2}(t) e^{\sigma\sqrt{s^2+r^2-2rst}}\over 
\big(s^2+r^2\big)^{d-1\over 2}} dt
\ge c(\ep,\eta)   e^{\sigma\sqrt{(s+r)^2-2\eta sr}}\label{D4}
\end{equation}
where $c(\ep,\eta)$ is the positive constant  $\ep\eta(\eta-\eta^2/4)^{d-3\over 2}/(2^{d+1\over 2}\eta^{d-1})$. We also have the following bound
\begin{eqnarray}
-|\S^{d-2}|\int_{-1+\eta}^{1} {(1-t^2)^{d-3\over 2} G_l^{d-2\over 2}(t) e^{\sigma\sqrt{s^2+r^2-2rst}}\over 
\big(s^2+r^2-2rs t\big)^{d-1\over 2}} dt\nonumber\\
\ge-C_l
(2\eta^2)^{1-d\over 2}\int_{-1+\eta}^0  e^{\sigma\sqrt{s^2+r^2-2rst}} dt-C_le^{\sigma\sqrt{s^2+r^2}}\int_0^1 {(1-t^2)^{d-3\over 2}dt\over 
\big(s^2+r^2-2rs t\big)^{d-1\over 2}}\label{D7}
\end{eqnarray}
for some positive constant $C_l$ that depends only on $l$.
We combine  \eqref{D1}, \eqref{D4} and \eqref{D7} and obtain
\begin{eqnarray}
&&f_l(s,r)\ge c(\ep,\eta)   e^{\sigma\sqrt{(s+r)^2-2\eta sr}}\nonumber\\
&&\times \Big[1-{C_l\over c(\ep,\eta) }(2\eta^2)^{1-d\over 2}\int_{-1+\eta}^0  e^{\sigma(\sqrt{s^2+r^2-2rst}-\sqrt{(s+r)^2-2\eta sr})} dt
\label{D8}\\
&&-{C_l\over c(\ep,\eta) }e^{\sigma(\sqrt{s^2+r^2}-\sqrt{(s+r)^2-2\eta sr})}\int_0^1 {(1-t^2)^{d-3\over 2}dt\over 
\big(s^2+r^2-2rs t\big)^{d-1\over 2}}
\Big].\nonumber
\end{eqnarray}
We note that the second term inside the brackets of the right hand side goes to $0$ uniformly in $(s,r)\in [\eta,1]^2$ as $\sigma\to +\infty$. Indeed we use the identity ``$a-b=(a^2-b^2)/(a+b)$'' and obtain that
$$
\int_{-1+\eta}^0  e^{\sigma(\sqrt{s^2+r^2-2rst}-\sqrt{(s+r)^2-2\eta sr})} dt
\le \int_0^{1-\eta} e^{-\sigma \eta^2 t\over 2} dt.
$$
Hence there exists $\sigma_0(l,\eta)$ so that for any $\sigma\ge \sigma_0(l,\eta)$ the second term is bounded by ${ c(\ep,\eta)(2\eta^2)^{d-1\over 2}\over 2 C_l}$, and we obtain 
\begin{equation}
f_l(s,r)
\ge c(\ep,\eta)
\Big[{ e^{\sigma\sqrt{(s+r)^2-2\eta sr}}\over 2}-{C_l\over c(\ep,\eta)}e^{\sigma(\sqrt{s^2+r^2})}\int_0^1 {(1-t^2)^{d-3\over 2}dt \over 
\big(s^2+r^2-2rs t\big)^{d-1\over 2}} \Big]
\label{D11}
\end{equation}
when $\sigma\ge \sigma_0(l,\eta)$.

Now pick any point $r_0\in (\eta,1)$. There exist $\delta \in (0,\min(r_0-\eta,1-r_0))$ and $\ep_1>0$ so that 
\begin{equation}
\inf_{(r,s)\in [r_0-\delta,r_0+\delta]^2}\sqrt{(s+r)^2-2\eta sr}\ge \ep_1+\sup_{(r,s)\in [r_0-\delta,r_0+\delta]^2}\sqrt{s^2+r^2}.\label{D14}
\end{equation}
We denote $\ep_2$ the positive number $\sup_{(r,s)\in [r_0-\delta,r_0+\delta]^2}\sqrt{s^2+r^2}$.

Let $g$ be any continuous and non-negative function on $[0,1]$ normalized by the condition $\|g\|_{L^2([0,1],r^{d-1}dr)}=1$ so that $g$ is compactly supported inside $[r_0-\delta,r_0+\delta]$.
We combine \eqref{D2}, \eqref{D11}, \eqref{D14} and we obtain for any $\sigma\ge \sigma_0(l,\eta)$
\begin{eqnarray*}
&&\l F_l g, g\r_{L^2([0,1],r^{d-1}dr)}=\int_{[0,1]^2} g(r)g(s)f_l(s,r)(rs)^{d-1}drds\nonumber\\
&\ge& c(\ep,\eta)e^{\sigma(\ep_1+\ep_2)}\int_{[0,1]^2} g(r)g(s)(rs)^{d-1} \Big[{1\over 2}-{C_le^{-\sigma\ep_1}\over c(\ep,\eta)}\int_0^1 {(1-t^2)^{d-3\over 2}dt \over 
\big(s^2+r^2-2rs t\big)^{d-1\over 2}} \Big]dr ds.
\end{eqnarray*}
The integrand in $r,s,t$ inside the integral in $t$ is integrable in $L^1([0,1]^2\times [-1,1], (rs)^{d-1}dr$ $ds dt)$, see \eqref{D10}. 
Hence we obtain that there exists a constant $\sigma(l,g)$ so that for any $\sigma \ge \sigma(l,g)$ 
$$
\l F_l g, g\r_{L^2([0,1],r^{d-1}dr)}
\ge  {c(\ep,\eta)\over 2}e^{\sigma\ep_1/2}\Big(\int_0^1 g(r)r^{d-1}dr\Big)^2.
$$
Obviously there exists a constant $\sigma_1(l,g)\ge \sigma(l,g)$ so that for any $\sigma \ge \sigma_1(l,g)$
$$
\l F_l g, g\r_{L^2([0,1],r^{d-1}dr)} 
\ge 2,
$$
which proves that the selfadjoint compact operator $F_l$ has a positive eigenvalue greater or equal to 2.\hfill $\Box$


\bibliographystyle{siam}

\end{document}